\documentclass[reqno,11pt]{amsart}
\usepackage{amssymb,amsmath,amsthm}
\usepackage{amscd}
\usepackage{amsfonts}
\usepackage{amssymb}
\usepackage{latexsym}
\usepackage{color}
\usepackage{esint}
\usepackage{graphicx, subfigure}
\usepackage{caption}
\usepackage{amsthm}
 \usepackage{tikz}
 \usetikzlibrary{shapes.geometric}

 \usepackage[makeroom]{cancel}
 \usetikzlibrary{decorations.pathreplacing}

\let\hide\iffalse
\let\unhide\fi

\newtheorem{theorem}{Theorem}

\newtheorem{corollary}[theorem]{Corollary}
\newtheorem{definition}[theorem]{Definition}
\newtheorem{lemma}[theorem]{Lemma}
\newtheorem{proposition}[theorem]{Proposition}
\newtheorem{remark}[theorem]{Remark}

\let\e=\varepsilon

\let\p=\partial

\let\O=\Omega

\let\o=\omega
\let\g=\gamma

\let\b=\beta

\newcommand{\R}{\mathbb{R}}

\renewcommand{\S}{\mathbb{S}}

\newcommand{\be}{\begin{equation}}
\newcommand{\bm}{\begin{multline}}
\newcommand{\ee}{\end{equation}}
\newcommand{\dd}{\mathrm{d}}

\newcommand{\xb}{x_{\mathbf{b}}}

\newcommand{\tb}{t_{\mathbf{b}}}

\newcommand{\X}{\mathbf{x}}
\newcommand{\V}{\mathbf{v}}
\newcommand{\T}{\Theta}

\newcommand{\Ga}{\Gamma_{\text{gain}}}
\newcommand{\bv}{\bar{v}}
\newcommand{\bx}{\bar{x}}
\newcommand{\tv}{\tilde{v}}
\newcommand{\tx}{\tilde{x}}

\newcommand{\Bes}{\begin{eqnarray*}}
\newcommand{\Ees}{\end{eqnarray*}}
\newcommand{\Be}{\begin{equation}}
\newcommand{\Ee}{\end{equation}}

 \numberwithin{equation}{section}
 \numberwithin{theorem}{section}

\def\p{\partial}

\def\O{\Omega}
\def\R{\mathbb{R}}

\def\B{\begin{equation}}
\def\E{\end{equation}}
\def\BN{\begin{eqnarray*}}
\def\EN{\end{eqnarray*}}

\usepackage{color}

\begin{document}
	\date{\today}
	
	\title
 {H{\"o}lder Regularity of the Boltzmann equation Past an Obstacle}

	\author{Chanwoo Kim \and Donghyun Lee}

	\address{Department of Mathematics, University of Wisconsin, Madison, WI 53706 USA; Department of Mathematical Sciences, Seoul National University, Seoul, 08826, Korea
	\\ckim.pde@gmail.com, chanwoo.kim@wisc.edu}  
	\address{Department of Mathematics, POSTECH, Pohang-si, Gyeongsangbuk-do, 37673 Republic of Korea
	\\donglee@postech.ac.kr}

	\begin{abstract} Regularity and singularity of the solutions according to the shape of domains is a challenging research theme in the Boltzmann theory (\cite{Kim11,GKTT1}). In this paper, we prove an H\"older regularity in $C^{0,\frac{1}{2}-}_{x,v}$ for the Boltzmann equation of the hard-sphere molecule, which undergoes the elastic reflection in the intermolecular collision and the contact with the boundary of a convex obstacle. In particular, this H\"older regularity result is a stark contrast to the case of other physical boundary conditions (such as the diffuse reflection boundary condition and in-flow boundary condition), for which the solutions of the Boltzmann equation develop discontinuity in a codimension 1 subset (\cite{Kim11}), and therefore the best possible regularity is BV, which has been proved in \cite{GKTT2}.
		
	\hide
	Formation and propagation of singularities according to the shape of domains is a challenging research theme in the Boltzmann theory (\cite{VM_p,Kim11}). In this paper, we investigate the formation of singularity and its quantitative properties of propagation and reflection in the Boltzmann theory of the hard-sphere molecule, which undergoes the elastic reflection in the intermolecular collision and the contact with the boundary of a convex obstacle. Our novel and comprehensive study demonstrate that 
	(1) solutions develop the non-smooth singularity stably on so-called grazing trajectories emanated from tangent space of the boundary; (2) this singularity propagates along the particle trajectory; (3) this type of singularities are reflected at the boundary; (4) the solutions are locally Lipschitz continuous away from the grazing trajectories; (5) the solutions are H\"older regular $C^{0,1/2}_{x,v}$ uniformly everywhere. 
	In particular, the result of the regularity and reflection of singularities is a stark contrast to the case of other physical boundary conditions (such as the stochastic boundary condition), for which the solutions of the Boltzmann equation develop discontinuity in general and such singularity completely fades away when the trajectory bounces back from the boundary again (\cite{Kim11}). 
	\unhide
	\end{abstract}
		
	\maketitle

 \tableofcontents

	\section{Introduction}
The Boltzmann equation is one of the fundamental mathematical equation for rarefied collisional gases. It describes the motion of binary collisional gas as a partial differential equation of distribution function $F(t,x,v)$. 
	Taking no external forces into account, the distribution function $F(t,x,v)$ solves 
		\begin{equation} \label{Boltzmann}
		\p_{t}F + v\cdot\nabla_{x} F = Q(F,F), \ \   F(0,x,v)= F_{0}(x,v),
		\end{equation}
	where nonlinear quadratic term $Q(F,F)$ means collision operator which has form of (we abbreviate arguments $t,x$ for convenience)
	\Be \label{Q}
		Q(F_{1},F_{2}) = \int_{\R^{3}}\int_{\S^{2}} B(v-u,\o) [F_{1}(u^{\prime}) F_{2}(v^{\prime}) - F_{1}(u) F_{2}(v)] d\o du,
	\Ee
	with the post collision velocities $u^{\prime}$ and $v^{\prime}$ which can be written as $u^{\prime} = u - ((v-u)\cdot \o)\o$ and $v^{\prime} = v + ((v-u)\cdot \o)\o$, respectively, for $\o\in\S^{2}$. For collision kernel $B(v-u ,\o)= |v-u|^\kappa q_0 \big(\big| \frac{v-u}{|v-u|} \cdot \o\big| \big)$, we work on a hard sphere model $B(v-u,\o) = |(v-u) \cdot \o |$, or $\kappa=1$ with $q_0(|\cdot|)=|\cdot |$.  \\
	
\hide
	Due to its own importance and analogy to statistical physics, there have been a lots of studies for the Boltzmann equation. In the seminal work \cite{DiLion}, they proved global existence of renormalized solution with large data. And in \cite{DV}, Desvillettes and Villani made a big progress about the large data problem and its convergence to equilibrium. Under assumptions of high order apriori bound and Gaussian lower bound they proved the solution converges to equilibrium. Moreover, C. Imbert and L. Silvestre \cite{CyrilLouis} recently derived $C^{\infty}$ apriori estmate of the Boltzmann equation using the theory of integro-differential equations. All these important results, however, are about $\R^{3}$ or $\mathbb{T}^{3}$ problem spatially. Taking general boundary condition problems are notoriously hard and expected to behave totally differently to $\mathbb{T}^{3}$ (or $\R^{3}$) problems. \\
\unhide

	Due to its own importance and applications in the statistical physics, there have been extensive researches on the Boltzmann equation (\cite{CIP,DiLion}). Guo established the energy method for the Boltzmann equation to construct global \textit{smooth} solutions and their asymptotical stability near Maxwellians (e.g. \cite{Guo_VMB}). In \cite{DV}, Desvillettes and Villani proved the asymptotic stability of the global Maxwellian under the assumptions of high order apriori bound and Gaussian lower bound of the \textit{smooth} solutions. More recently, Imbert and Silvestre \cite{CyrilLouis} derived conditional $C^{\infty}$-\textit{smoothness} of the non-cutoff Boltzmann equation, using the theory of integro-differential equations. Forementioned works, especially \cite{DiLion, Guo_VMB, CyrilLouis}, deal with idealized periodic domains or whole space, in which the solutions can remain \textit{smooth} if initially so.  \\

\hide
{\color{blue}
However, a physical boundary present in many applications, which changes global properties (such as \textit{smoothness}) of solutions in general. [DELETE “In [32], the authors …. propagation of singularities according to the shape [of domains]”] In [21], Guo et al…. for the first time. [INSERT: In [CK1], Chen and Kim proved some higher regularity up to $C^{1,1-}$ away from the boundary for the steady problem of diffuse boundary.]… Indeed in [22], they proved this optimal BV regularity for the diffuse and inflow boundary.  \\
} 
\unhide

However, a physical boundary present in many applications, which changes global properties (such as \textit{smoothness}) of solutions in general. 
In \cite{GKTT1}, Guo et al proved that in convex domains, the first order derivatives are continuous away from a grazing set, where a particle hits the boundary with tangential velocity, for the first time. In \cite{CK1}, Chen and Kim proved some higher regularity up to $C^{1,1-}$ away from the boundary for the steady problem of diffuse the boundary. When the domain is non-convex, Kim proved the formation and propagation of discontinuity along the characteristics emanating from the grazing set of non-convex boundary of the diffuse/inflow/bounce-back reflection in \cite{Kim11}. As such characteristics form a codimension 1 subset in the phase space, the optimal regularity is BV. Indeed in \cite{GKTT2}, they proved this optimal BV regularity for the diffuse and inflow boundary. We also refer very recent result \cite{C2} which shows that higher order regularity is verd to obtain in general even for free transport equation in the presence of physical boundary conditions. \\

The regularity question of the Boltmann equation with the specular reflection boundary condition (a particle hits the boundary and bounces back like a billiard) in the non-convex domains has been a challenging open problem. Even the global well-posedness of such problem has been an outstanding open problem since an announcement of \cite{SA} without full justification in 1977. Only very recently, the question is settled affirmatively in \cite{Guo10, KimLee}. In particular, in \cite{KimLee}, Kim-Lee constructed the first unique global-in-time solution and proved asymptotic stability of the Boltzmann equation near equilibria in smooth convex domains, which completely settled the classical long-standing (40 years) open question of the kinetic community in the affirmative. As the key of the proof, they establish a novel $L^2-L^\infty$ estimate using triple iterations of the Duhamel representation along the billiard trajectory in smooth convex domains. In \cite{KimLee2}, they further extended the result of \cite{KimLee} to the cylindrical non-convex domains, which is the first result for any non-convex domains without any symmetry. \\

In this paper, we are studying H\"older regularity of the Boltzmann equation in some non-convex domains with the specular reflection boundary condition in a full 3D setting. For the reader’s convenient, we state an informal statement of Theorem \ref{theo:Holder} :
\begin{theorem}[Informal statement]
	A local-in-time solution of the Boltzmann equation outside a general convex domain satisfying the specular reflection boundary condition is $C^{0,\frac{1}{2}-}_{x,v}$ if the initial datum is $C^{0,1-}_{x,v}$.
\end{theorem}
The major difficulty of the problem is that the result of the velocity lemma of \cite{GKTT1}, which is a powerful tool of the regularity estimate in convex domains, is no longer true for non-convex domains. Therefore we have to develop a \textit{completely new} technique in the paper. Moreover this H\"older regularity result is a stark contrast to the result of other physical boundary conditions \cite{GKTT2, Kim11}, for which the formation and propagation of discontinuity happens in general.\\

Beyond the theoretical importance in general, the quantitative regularity estimate has several significant applications. Based on the regularity estimate results \cite{GKTT1}, later Cao-Kim-Lee studied Vlasov-Poisson-Boltzmann systems. When particles are charged (e.g. plasma), they interact through not only collisions but also the Lorentz force generated by charged particles. Master equations for
this situation are Vlasov-Boltzmann equation coupled with Maxwell system or simpler Poisson equation. Mathematical theory of boundary problems of such coupled system is not developed satisfactorily mainly due to the intrinsic singularity of solutions. In \cite{CKL1}, Cao-Kim-Lee construct the first unique global-in-time solution of Vlasov-Poisson-Boltzmann system with the diffuse reflection boundary condition in smooth convex domains and proved its convergence toward equilibria when the initial data is small enough. In \cite{CKL2} the authors construct a unique local-in-time solution of Vlasov-
Poisson-Boltzmann system with a generalized diffuse reflection boundary condition in smooth convex domains without the size restriction.

Lastly, we note that the $L^2-L^\infty$ method established in \cite{Guo10, KimLee} has been widely used to various problems, such as global wellposedness with large amplitude external potential (\cite{K2}), rotational setting \cite{KY}, and large amlitude problem with smooth convex domain \cite{DKL}. We also refer recent global well-posedness results with large amplitude initial data for various setting \cite{DHWY, DHWZ, DW, KLP}.   \\

First let us define a convex ball $\mathcal{O}$ and the domain $\O$.
	\begin{definition}\label{def:domain}
	Throughout this paper, we assume that the domain $\O=\R^3 \backslash \overline{\mathcal{O}}$ is an exterior domain of $\mathcal{O}$, which is a uniformly convex bounded open subset in $\R^{3}$: $\mathcal{O}= \{x \in \R^3: \xi(x)>0\}$ and there exists $\theta_{\O} >0$ such that 
	\Be\label{convex_xi}
\zeta \cdot(- \nabla^2 \xi(x))   \zeta	:=  \sum_{i,j=1}^3  (-\p_{ij} \xi(x)) \zeta_i\zeta_j \geq \theta_{\O} |\zeta|^2 \ \ \text{for all} \ \zeta \in \R^3.
	\Ee	 
	In other words, $\O = \{x \in \R^3: \xi(x)<0 \}$ with $\xi$ of \eqref{convex_xi}. 
We define unit normal vector of the exterior domain $\O$ as 
	\Be \notag 
	n(x) =  \frac{\nabla \xi (x)}{|\nabla \xi(x)|} .  \\
	\Ee
	\end{definition}
	\noindent Moreover, we impose the specular reflection boundary condition on the boundary $\p\O$,
		\begin{equation} \label{specular}
		F(t,x,v) = F(t,x,R_{x}v) \  \textit{ for } x\in\p\O, \  \textit{ where } \ R_x v = v- 2 (n(x)\cdot v)n(x).
		\end{equation}
In the classical Boltzmann theory, an equilibrium state of the problem \eqref{Boltzmann} and \eqref{specular} is well-known as a global Maxwellian
\Be \notag
\mu(v) = e^{-  \frac{|v|^2}{2} }.   \\
\Ee 

\section{Main Theorem and Scheme of Proof}

		We rewrite the Boltzmann equation (\ref{Boltzmann}) and the specular reflection boundary condition \eqref{specular} using $F(t,x,v) = \sqrt{\mu}f(t,x,v)$ to obtain  
			\Be\label{f_eqtn} 
	\p_{t}f + v\cdot \nabla_{x} f  
	 = \Gamma(f,f) := \Gamma_{\text{gain}}(f,f) - \nu(f) f ,  
	 \ \ 
	f|_{t=0} =f_0, \ \ f(t,x,v)|_{ \p\O}  = f(t,x,R_{x}v),
 \Ee
	where
	\begin{align}
	\Gamma_{\text{gain}}(f_1,f_2) &:= \frac{1}{\sqrt{\mu}}Q_{\text{gain}}(\sqrt{\mu}f_1,\sqrt{\mu}f_2),\label{Gamma}\\
	\nu(f) &:= \iint_{\mathbb{R}^{3}\times\mathbb{S}^{2}} |(v-u)\cdot\omega| \sqrt{\mu(u)} f(t,x,u) d\omega du. \notag 
\end{align}

We define a backward exit time and position as 
 	\begin{equation} \notag 
 	\begin{split}
 	\tb(x,v)   &:=  \sup \big\{  s \geq 0 :  x-\tau v  \in \O  \ \ \text{for all}  \ \tau \in ( 0, s )  \big\}, \quad t^{1}(x,v) := t-\tb(x,v),  \\
 	\xb(x,v)   &:=  x- \tb( x,v) v.  
 	\end{split}
 	\end{equation}
Then the trajectory of (\ref{f_eqtn}) satisfies that for $(t,x,v)\in[0,\infty)\times\O \times\mathbb{R}^3,$ 
 	\begin{equation}\label{E_Ham} 
 	\frac{d}{ds} (X(s;t,x,v), V(s;t,x,v)  )   =    (V(s;t,x,v), 0),  \ \  \ 
	(X(s;t,x,v)  , V (s;t,x,v)  )|_{s=t}=(x,v). 
 	\end{equation} 
	While the explicit formula is given as 
\Be\label{form_XV}
\begin{split}
&(X(s;t,x,v), V(s;t,x,v) ) \\
&= \begin{cases} (x-(t-s) v, v)  
 \ \ \text{for} \ \ s \in (t-\tb(x,v),t]
   \\
 (\xb(x,v) - (t-\tb(x,v) - s) R_{\xb(x,v)} v, R_{\xb(x,v)} v
 )
  \ \ \text{for} \ \ s \in [0,t-\tb(x,v) ],
\end{cases}
\end{split}
\Ee	
where $R_{\xb(x,v)} = I - 2 n(\xb)\otimes n(\xb)$ is reflection operator. 
	From (\ref{f_eqtn}), we obtain the Duhamel's formula   
\Be \label{f_expan}
\begin{split} 
	f(t,x,v)
	&=  
	e^{- \int^t_ 0 \nu(f) (\tau, X(\tau;t,x,v), V(\tau;t,x,v)) d \tau}
	f(0,X(0;t,x,v), V(0;t,x,v))\\
	& \ \ + \int^t_0
	e^{- \int^t_ s \nu(f) (\tau, X(\tau;t,x,v), V(\tau;t,x,v)) d \tau}
	\Gamma_{\text{gain}}(f,f)(s,X(s;t,x,v), V(s;t,x,v)) ds.
\end{split}
\Ee 

Local in time existence of the mild solultion \eqref{f_expan} is given as the following.
 \begin{lemma}[Local existence] \label{lem loc}
	 Assume initial data $f_{0}$ satisfies $\|w_{0}(v)f_{0}\|_{\infty} := \|e^{\vartheta_{0}|v|^{2}}f_{0}\|_{\infty} < \infty$ for $0 < \vartheta_{0} < \frac{1}{4}$ and initial compatibility condition \eqref{specular}. Then there exists $T^{*} > 0$ and a local in time unique solution $f(t,x,v)$ of (\ref{f_expan}) for $0\leq t \leq T^{*}$. Moreover, the solution satisfies
	\begin{equation*}
	\sup_{0\leq s \leq T^{*}} \|w(v) f(s)\|_{\infty} :=
	\sup_{0\leq s \leq T^{*}} \|e^{\vartheta |v|^{2}}f(s)\|_{\infty} \lesssim \|e^{\vartheta_{0} |v|^{2}}f_{0}\|_{\infty},
	\end{equation*}
	for some $0 < \vartheta < \vartheta_{0}$. 
\end{lemma}
\begin{proof}
We refer \cite{GKTT1}.	 
\end{proof}

Now, we define \textit{shifted} position and velocity which will be \textit{crucially} used when we perform H\"older regularity estimates. 
	\begin{definition}[Shift] \label{def_tilde}
	We define shifted position $\tilde{x} = \tx(x ,\bx, v)$ and velocity $\tilde{v} = \tv(v ,\bv, \zeta)$, respectively. \\
	(i) For fixed $x, \bx, v$, let us  assume
	\begin{equation} \label{assume_x}
		 \text{ $x-\bx \neq 0$ is neither parallel nor anti-parallel to $v\neq 0$,  i.e., $(x-\bx)\cdot v \neq \pm |x-\bx||v|$. } \\
	\end{equation}
	In this case, we define shifted $\tilde{x}$ as
	\begin{equation} \label{def_tildex} 
		\tilde{x} =  \tilde{x}(x, \bx, v) := \bar{x} + \big( (x-\bar{x})\cdot \hat{v} \big) \hat{v}.    \\
	\end{equation}
	(ii) For fixed $v, \bv, \zeta$, let us  assume
	\begin{equation} \label{assume_v}
		\text{ $v+\zeta\neq 0$ is neither parallel nor anti-parallel to $\bv+\zeta\neq 0$,  i.e., $(v+\zeta)\cdot(\bv+\zeta)\neq \pm|v+\zeta||\bv+\zeta|$ }.
	\end{equation}
	In this case, we define shifted $\tilde{v}$ as
	\begin{equation} \label{def_tildev} 
	\tilde{v} + \zeta = \tilde{v}(v, \bv, \zeta) + \zeta := |v+\zeta| \widehat{(\bar{v} +\zeta)} . 
	\end{equation}
	Note that we have used the following notation,
	\[
		\hat{v} := \begin{cases}
		\frac{v}{|v|},\quad v\neq 0, \\ 0,\quad v=0.
		\end{cases}
	\]\\
	\end{definition}

	Next, we parametrize position and velocity using above shifted position and velocity. Before we introduce parametrizations, we first define cross section and argument.
	\begin{definition} \label{def_Sarg}
		(i) Let us assume \eqref{assume_x}. We define
		\begin{equation} \label{S_x}
			S_{(x, \bx, v)} := x + \text{span}\{ x- \tilde{x}, v \} = x + \text{span}\{ x- \bx, v \},
		\end{equation}
		where $\tilde{x} = \tilde{x}(x, \bx, v)$ is defined in \eqref{def_tildex}. \\
		(ii) Let us assume \eqref{assume_v}. We define
		\begin{equation} \label{S_v}
			S_{(x, v, \bv, \zeta)} := x + \text{span}\{ v+\zeta, \tilde{v}+\zeta \} = x + \text{span}\{ v+\zeta, \bv+\zeta \}, 
		\end{equation}
		where $\tilde{v} = \tilde{v}(v, \bv, \zeta)$ is defined in \eqref{def_tildev}. In the plane $S_{(x, v, \bv, \zeta)}$, we define $\arg : S_{(x, v, \bv, \zeta)}\backslash\{0\} \mapsto [0,2\pi)$ by 
		\Be \label{def_arg}
			\arg(w) = \cos^{-1}\big( \hat{w}\cdot \widehat{(\bv+\zeta)} \big),\quad \forall  w\in S_{(x, v, \bv, \zeta)}\backslash\{0\}, \\
		\Ee		
		so that $\arg(\bv+\zeta) = 0$. 
	\end{definition}
	 
	 Note that we are mainly interested in when cross section $\p\O\cap S_{(x, \bx , v)}$ and $\p\O\cap S_{(x, v, \bv, \zeta)}$ have meaningful geometry. So we assume
	 \Be \label{assume_x2}
	 	\text{ $\p\O\cap S_{(x, \bx , v)}$ is closed curve, i.e., $\p\O\cap S_{(x, \bx , v)} \neq \emptyset$ is neither an empty set nor a single point,}
	 \Ee
	 for $x$ perturbation case and 
	 \Be \label{assume_v2}
	 	\text{ $\p\O\cap S_{(x, v, \bv , \zeta)}$ is closed curve, i.e., $\p\O\cap S_{(x, v, \bv , \zeta)} \neq \emptyset$ is neither an empty set nor a single point,}
	 \Ee
	 for $v$ perturbation case. \\
	 
	\begin{definition}[Parametrizations] \label{def_para}
		Recall $\tilde{x}(x, \bar{x}, v)$ in Definition \ref{def_tilde} under assumption \eqref{assume_x}. Also we recall $\tilde{v}(v, \bar{v} ,\zeta)$ in Definition \ref{def_tilde} under assumption \eqref{assume_v}. \\
		(i) We define parametrizations,
		\Be \label{xv para}
		\begin{split}
			\X(\tau) & = 
			\X(\tau; x, \bar x, v)
			= (1-\tau)\tilde{x} (x, \bar x, v) + \tau x ,\quad \dot{\X}:= \dot{\X}(\tau)  = x- \tx, \\
			\V(\tau) &=
			\V(\tau; x, v , \bar v, \zeta) 
				=
			 |v+\zeta| R_{(v,\bv,\zeta)} 
			\begin{bmatrix}
				\cos\T(\tau) \\ \sin\T(\tau) \\ 0
			\end{bmatrix},
		\\
			\T(\tau) &:= \tau\theta, \quad \dot{\T}(\tau) = \dot{\T} := \theta,
		\end{split}\Ee
	   where $\theta = 
	   \cos^{-1} (\widehat{v+\zeta} \cdot \widehat{\bv+\zeta})
	   \in[0,2\pi)$ is the angle between $v+\zeta$ and $\bv+\zeta$, and 
		\[
			R_{(v,\bv,\zeta)} := 
			\begin{bmatrix}
				 & & \\
				 \widehat{\bv+\zeta} & \widehat{v+\zeta} & \widehat{ (\bv+\zeta)\times (v+\zeta) } \\
				 & & \\
			\end{bmatrix}
			\begin{bmatrix}
				1 & \cos\theta & 0 \\
				0 & \sin\theta & 0 \\
				0 & 0 & 1
			\end{bmatrix}^{-1},
		\]
		so that
		\begin{equation*}
		\begin{split}
			&\X(0)=\tx,  \ \  \X(1)=x,   \\ 
			&\T(0) = \arg (\tilde{v}+\zeta) = \arg (\bv+\zeta) = 0 \ \ \text{by \eqref{def_arg}, and} \ \  
			\T(1) = \arg (v+\zeta) := \theta, \\
			&\V(0) = \tilde{v}+\zeta , \ \ \V(1)=v+\zeta .
		\end{split}
		\end{equation*}
		(ii) We further assume \eqref{assume_x2} and \eqref{assume_v2}, respectively. By convexity \eqref{convex_xi} of $\O$, we can define $\tau_{\pm}(x, \bx, v)$ and $\tau_{\pm}(x, v, \bv, \zeta)$ as 
		\begin{equation} \label{tau_pm}
		\begin{split}
			&v\cdot \nabla\xi(\xb(\X(\tau_{\pm}), v)) = 0,\quad \tau_{-} < \tau_{+},\quad \tau_{\pm} = \tau_{\pm}(x, \bx, v), \\
			&\V(\tau_{\pm})\cdot \nabla\xi(\xb(x, \V(\tau_{\pm}))) = 0,\quad \tau_{-} < \tau_{+},\quad \tau_{\pm} = \tau_{\pm}(x, v, \bv, \zeta), \\  
		\end{split}
		\end{equation}
		where $\xb(\X(\tau),v)$ and $\xb(x, \V(\tau))$ are well defined for $\tau_{-} \leq \tau \leq \tau_{+}$, respectively. 
		Here we excuse our abuse of notations, different $\tau_\pm$'s are only distinguished by their arguments, for the sake of briefness in the later use of them.  \\
	\end{definition}

	\begin{remark}
	By Definition \ref{def_tilde} and \ref{def_para}, we can check the following crucial properties:
	\Be \label{perp}
	\begin{split}
		&\dot{\X}\cdot v =0, \quad\dot{\V}(\tau)\cdot \V(\tau) = 0,\quad |\dot\V(\tau)| = \theta|\V(\tau)|,  \\
		&\ddot{\V}(\tau) = -\theta^{2}\V(\tau),\quad \text{and}\quad \ddot{\V}(\tau)\cdot \V(\tau) = - \theta^{2}|\V(\tau)|^{2},
	\end{split}
	\Ee
	and
	\Be \notag 
		|{\V}(\tau)| = |v+\zeta|,\quad |\dot{\V}(\tau)| = \theta|v+\zeta|,\quad \text{and}\quad |\ddot{\V}(\tau)| = \theta^{2}|v+\zeta|.  \\
	\Ee	
	\end{remark}

\begin{figure}
	\begin{center}
		\begin{tikzpicture}
		\draw[rotate=30] (-0.7,0.4) ellipse (2cm and 1.3cm);
		\draw (-2.5,-1) node [below left] {$\p\O\cap S_{(x, \bx ,v)}$};
		\filldraw (-0.2,1.5) circle[radius=1.5pt];
		\draw (0,1.5) node [above left] {\tiny{$\xb(\X(\tau_{-}),v)$}};
		\draw[arrows=->] (-0.2,1.5)-- (-0.2,1);
		\draw (-0.7,1) node[below]{\tiny{$n_{\parallel}(\xb(\X(\tau_{-}), v))$}} ;
		\filldraw (-1.4,-1.5) circle[radius=1.5pt];
		\draw[arrows=->] (-1.4,-1.5)-- (-1.4,-1);
		\draw (-1.4,-1) node[above]{\tiny{$n_{\parallel}(\xb(\X(\tau_{+}), v))$}} ;
		\filldraw (0.99, 0.9) circle[radius=1.5pt];
		\draw[arrows=->] (0.99,0.9)-- (0.5,0.65);
		\draw (0.65,0.55) node[below]{\tiny{$n_{\parallel}(\xb(\X(\tau), v))$}} ;
		\draw [dashed](-2,2.2) -- (6,2.2) ;
		\draw [dashed](-0.2,1.5) -- (3,1.5) ;
		\draw [dashed](-1.4,-1.5) -- (3,-1.5) ;
		\draw [dashed](1.2,0.9) -- (3,0.9) ;
		\draw (3,-1.5) -- (3,2.2);
		\draw[arrows=->] (3,0.4)-- node[below right]{$v$} (4,0.4);
		\filldraw (3,0.4) circle[radius=1.5pt];
		\draw (3,0.4) node [above left] {\tiny{$x=\X(1)$}};
		\draw[arrows=->] (3,0.9)-- node[below right]{$v$} (4,0.9);
		\filldraw (3,0.9) circle[radius=1.5pt];
		\draw (3,0.9) node [above left] {\tiny{$\X(\tau)$}};
		\draw[arrows=->] (3,1.5)-- node[below right]{$v$} (4,1.5);
		\filldraw (3,1.5) circle[radius=1.5pt];
		\draw (3,1.5) node [above left] {\tiny{$\X(\tau_{-})$}};
		\draw[arrows=->] (3,-1.5)-- node[below right]{$v$} (4,-1.5);
		\filldraw (3,-1.5) circle[radius=1.5pt];
		\draw (3,-1.5) node [above left] {\tiny{$\X(\tau_{+})$}};
		\draw[arrows=->] (3,2.2)-- node[above right]{$v$} (4,2.2);
		\filldraw (3,2.2) circle[radius=1.5pt];
		\draw (3,2.2) node [above left] {\tiny{$\tilde{x}=\X(0)$}};
		\draw[arrows=->] (6,2.2)-- node[above right]{$v$} (7,2.2);
		\filldraw (6,2.2) circle[radius=1.5pt];
		\draw (6,2.2) node [above left] {$\bar{x}$};
		\end{tikzpicture}
	\end{center}
	\caption{$\tilde{x}$, $\X(\tau)$, and trajectories on projected plane $S_{(x, \bx ;v)} := x+ \text{span}\{ x-\tilde{x}, v\}$} \label{fig1}
\end{figure}

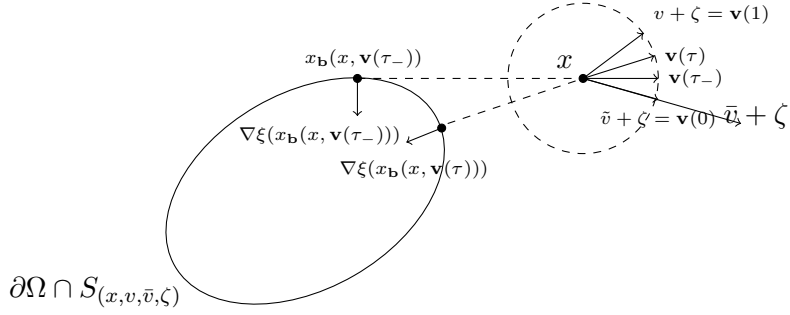
\begin{figure}
	\begin{center}
		\begin{tikzpicture}
		\draw[rotate=30] (-0.6,0.35) ellipse (2cm and 1.3cm);
		\draw[dashed] (3,1.5) circle (1);
		\draw (-2.2,-1) node [below left] {$\p\O\cap S_{(x, v, \bv, \zeta)}$};
		\filldraw (0,1.5) circle[radius=1.5pt];
		\draw (1,1.5) node [above left] {\tiny{$\xb(x, \V(\tau_{-}))$}};
		\draw[arrows=->] (0,1.5)-- (0,1);
		\draw (-.5,1) node[below]{\tiny{$\nabla\xi(\xb(x, \V(\tau_{-})))$}} ;
		\filldraw (1.12,0.84) circle[radius=1.5pt];
		\draw[arrows=->] (1.12,0.84)-- (0.65,0.65);
		\draw (0.75,0.55) node[below]{\tiny{$\nabla\xi(\xb(x, \V(\tau)))$}} ;
		\draw [dashed] (0,1.5) -- (3,1.5) ;
		\draw [dashed] (1.2,0.9)-- (3,1.5);
		\draw (5.3, 1) node{$\bar{v}+\zeta$};
		\draw[arrows=->] (3,1.5)--  (5.1,0.9);
		\draw[arrows=->] (3,1.5)-- (4,1.22);
		\draw (4,1.22) node[below]{\tiny{$\tilde{v}+\zeta = \V(0)$}} ;
		\draw (4,1.5) node[right]{\tiny{$\V(\tau_{-})$}} ;
		\draw[arrows=->] (3,1.5)-- (4,1.5);
		\draw[arrows=->] (3,1.5)-- (3.95,1.8);
		\draw (3.95,1.8) node[right]{\tiny{$\V(\tau)$}} ;
		\draw[arrows=->] (3,1.5)-- (3.8,2.1);
		\draw (3.8,2.1) node [above right] {\tiny{$v+\zeta = \V(1)$}};
		\filldraw (3,1.5) circle[radius=1.5pt];
		\draw (3,1.5) node [above left] {$x$};
		\end{tikzpicture}
	\end{center}
	\caption{$\tilde{v}$, $\V(\tau)$, and  trajectories on projected plane $S_{(x, v, \bv, \zeta)} := x + \text{span}\{ v+\zeta, \bar{v}+\zeta \}$} \label{fig2}
\end{figure}

\hide
\begin{figure}
	\begin{center}
		\begin{tikzpicture}
		\draw(0,0) circle (1.5);
		\draw (-1,-1) node [below left] {$\p\O\cap S_{(x, \bx ;v)}$};
		\filldraw (0,1.5) circle[radius=1.5pt];
		\draw (0,1.5) node [above left] {\tiny{$\xb(\X(\tau_{-}),v)$}};
		\draw[arrows=->] (0,1.5)-- (0,1);
		\draw (0,1) node[below]{\tiny{$n_{\parallel}(\xb(\X(\tau_{-}), v))$}} ;
		\filldraw (0,-1.5) circle[radius=1.5pt];
		\draw[arrows=->] (0,-1.5)-- (0,-1);
		\draw (0,-1) node[above]{\tiny{$n_{\parallel}(\xb(\X(\tau_{+}), v))$}} ;
		\filldraw (1.2,0.9) circle[radius=1.5pt];
		\draw[arrows=->] (1.2,0.9)-- (0.75,0.51);
		\draw (0.75,0.55) node[below]{\tiny{$n_{\parallel}(\xb(\X(\tau), v))$}} ;
		\draw [dashed](0,2.2) -- (6,2.2) ;
		\draw [dashed](0,1.5) -- (3,1.5) ;
		\draw [dashed](0,-1.5) -- (3,-1.5) ;
		\draw [dashed](1.2,0.9) -- (3,0.9) ;
		\draw (3,-1.5) -- (3,2.2);
		\draw[arrows=->] (3,0.4)-- node[below right]{$v$} (4,0.4);
		\filldraw (3,0.4) circle[radius=1.5pt];
		\draw (3,0.4) node [above left] {\tiny{$x=\X(1)$}};
		\draw[arrows=->] (3,0.9)-- node[below right]{$v$} (4,0.9);
		\filldraw (3,0.9) circle[radius=1.5pt];
		\draw (3,0.9) node [above left] {\tiny{$\X(\tau)$}};
		\draw[arrows=->] (3,1.5)-- node[below right]{$v$} (4,1.5);
		\filldraw (3,1.5) circle[radius=1.5pt];
		\draw (3,1.5) node [above left] {\tiny{$\X(\tau_{-})$}};
		\draw[arrows=->] (3,-1.5)-- node[below right]{$v$} (4,-1.5);
		\filldraw (3,-1.5) circle[radius=1.5pt];
		\draw (3,-1.5) node [above left] {\tiny{$\X(\tau_{+})$}};
		\draw[arrows=->] (3,2.2)-- node[above right]{$v$} (4,2.2);
		\filldraw (3,2.2) circle[radius=1.5pt];
		\draw (3,2.2) node [above left] {\tiny{$\tilde{x}=\X(0)$}};
		\draw[arrows=->] (6,2.2)-- node[above right]{$v$} (7,2.2);
		\filldraw (6,2.2) circle[radius=1.5pt];
		\draw (6,2.2) node [above left] {$\bar{x}$};
		\end{tikzpicture}
	\end{center}
	\caption{$\tilde{x}$, $\X(\tau)$, and trajectories on projected plane $S_{(x, \bx ;v)} := x+ \text{span}\{ x-\tilde{x}, v\}$} 
\end{figure}

\begin{figure}
	\begin{center}
		\begin{tikzpicture}
		\draw(0,0) circle (1.5);
		\draw[dashed] (3,1.5) circle (1);
		\draw (-1,-1) node [below left] {$\p\O\cap S_{(x, v, \bv, \zeta)}$};
		\filldraw (0,1.5) circle[radius=1.5pt];
		\draw (0,1.5) node [above left] {\tiny{$\xb(x, \V(\tau_{-}))$}};
		\draw[arrows=->] (0,1.5)-- (0,1);
		\draw (0,1) node[below]{\tiny{$\nabla\xi(\xb(x, \V(\tau_{-})))$}} ;
		\filldraw (1.2,0.9) circle[radius=1.5pt];
		\draw[arrows=->] (1.2,0.9)-- (0.75,0.51);
		\draw (0.75,0.55) node[below]{\tiny{$\nabla\xi(\xb(x, \V(\tau)))$}} ;
		\draw [dashed] (0,1.5) -- (3,1.5) ;
		\draw [dashed] (1.2,0.9)-- (3,1.5);
		\draw (5.3, 1) node{$\bar{v}+\zeta$};
		\draw[arrows=->] (3,1.5)--  (5.1,0.9);
		\draw[arrows=->] (3,1.5)-- (4,1.22);
		\draw (4,1.22) node[below]{\tiny{$\tilde{v}+\zeta = \V(0)$}} ;
		\draw (4,1.5) node[right]{\tiny{$\V(\tau_{-})$}} ;
		\draw[arrows=->] (3,1.5)-- (4,1.5);
		\draw[arrows=->] (3,1.5)-- (3.95,1.8);
		\draw (3.95,1.8) node[right]{\tiny{$\V(\tau)$}} ;
		\draw[arrows=->] (3,1.5)-- (3.8,2.1);
		\draw (3.8,2.1) node [above right] {\tiny{$v+\zeta = \V(1)$}};
		\filldraw (3,1.5) circle[radius=1.5pt];
		\draw (3,1.5) node [above left] {$x$};
		\end{tikzpicture}
	\end{center}
	\caption{$\tilde{v}$, $\V(\tau)$, and  trajectories on projected plane $S_{(x, v, \bv, \zeta)} := x + \text{span}\{ v+\zeta, \bar{v}+\zeta \}$}  
\end{figure}
\unhide

Next definition precisly describes the kernels of integral operators which come from $\Gamma_{\text{gain}}(f,f)$. 
\begin{definition} \label{def_k}
	For $c > 0$, we define
	\Be \label{full k}
	\begin{split}
		k_{c}(v, v+\zeta) := \frac{1}{|u|}e^{ - c|\zeta|^{2} - c \frac{ | |v|^{2}-|v+\zeta|^{2} |^{2} }{|\zeta|^{2}} } ,
	\end{split}
	\Ee
	and 
	\Be\label{bar k}
	\mathbf{k}_{c} (v, \bar v, \zeta) := k_{c}(v, v+\zeta) + k_{c}(\bv, \bv+\zeta).   \\
	\Ee
\end{definition}
Note that the following equalities are obvious. 
\Be \notag 
		k_{c}(v, v+\zeta) = k_{c}(R_{x}v, R_{x}v + R_{x}\zeta),\quad 	\mathbf{k}_{c} (v, \bar v, \zeta) = \mathbf{k}_{c} (R_{x}v, R_{x}\bar v, R_{x}\zeta),\quad x\in\p\O. \\
\Ee

The following two definitions are the most important quantities of this paper. 

\begin{definition}[Specular Singularity] \label{def_S} Let $x, \tilde{x} \in \O$, $v, \tilde{v}, \zeta \in\R^{3}$. We assume \eqref{assume_x} and \eqref{assume_x2} for $\X(\tau)$ case, and assume \eqref{assume_v} and \eqref{assume_v2} for $\V(\tau)$ case, respectively. Recall the notations $\X(\tau)$ and $\V(\tau)$ in \eqref{xv para}. Suppose $\xb(\X(\tau), v)$ and $\xb(x, \V(\tau))$ are well-defined on $\p\O$, respectively. We define a reciprocal of the specular singularity 
\begin{align}
\mathfrak{S}_{sp}(\tau;x,\tilde{x}, v) &: =  
		\frac{- \nabla \xi(\xb( \X(\tau), v)) \cdot v  }{
		\big| \frac{\dot{\X}}{|\dot \X|} \cdot \nabla \xi (\xb( \X(\tau), v))\big|
		}
		, \label{def:S_x} \\
\mathfrak{S}_{vel}(\tau;x,  v,  \tilde {v}, \zeta) &: =  
		\frac{- \nabla \xi(\xb(x, \V(\tau))) \cdot \V(\tau)  }{\tb(x, \V(\tau))
		\big| \frac{\dot{\V}(\tau)}{|\dot{\V}(\tau)|}\cdot \nabla \xi (\xb(x , \V(\tau)))\big|
		}
		.\label{def:S_v}
\end{align}
\end{definition}


\begin{definition}[Seminorms] \label{def_H}
	Let $x, \bx \in \O$ and  $v, \bv, \zeta \in\R^{3}$. For $\varpi > 0$, we define
	\begin{equation} \notag
	\begin{split}
	\mathfrak{H}^{2\b}_{vel}(s) &:= \sup_{\substack{x\in {\O},  \\ 0<|v-\bv|\leq 1}} e^{ - \varpi \langle v  \rangle^{2} s } 
	\int_{\zeta} \mathbf{k}_{c}(v, \bv, \zeta) \frac{ | f(s, x, v+\zeta) - f(s, x, \bv+\zeta) | }{ | v - \bv |^{2\b} } d\zeta,  \\
	\end{split}
	\end{equation}
	\begin{equation} \notag
	\begin{split}
		\mathfrak{H}^{2\b}_{sp}(s) &:=  \sup_{\substack{v\in\R^{3},  \\ 0<|x-\bx|\leq 1}} e^{ - \varpi \langle v  \rangle^{2} s } 
		\int_{\zeta} k_{c}(v, v+\zeta) \frac{ | f(s, x, v+\zeta) - f(s, \bx, v+\zeta) | }{ | x - \bx |^{2\b} } d\zeta,  \\
	\end{split}
	\end{equation}
	where $\langle v \rangle := \sqrt{1 + |v|^2}$.  \\
\end{definition}

\begin{theorem}[Main theorem] \label{theo:Holder}
	Suppose the domain is given as in Definition \ref{def:domain} and \eqref{convex_xi}. Assume $f_0$ satisfies compatibility condition \eqref{specular}, $\|e^{\vartheta_{0}|v|^{2}} f_0 \|_\infty< \infty$ for $0< \vartheta_{0} < \frac{1}{4}$, and 
	\[
		\sup_{ \substack{ v\in \R^{3} \\ 0 < |x - \bx|\leq 1   } }
		\langle v \rangle  \frac{|f_{0}( x, v ) - f_{0}(\bx, v)|}{|x - \bx|^{2\b}}  
		+ \sup_{ \substack{ x\in \overline{\O} \\ 0<|v - \bv|\leq 1   } }  \langle v \rangle^{2} \frac{|f_{0}( x, v ) - f_{0}( x, \bv)|}{|v - \bv|^{2\b}}  
		< \infty,\quad \b < \frac{1}{2}.
	\] 
	Then there exists $0< T \ll 1$ such that we have a unique solution $f(t,x,v)$ of \eqref{f_eqtn} for $0 \leq t \leq T$ with $\sup_{0 \leq t \leq T}\| e^{\vartheta|v|^2} f (t) \|_\infty\lesssim \| e^{\vartheta_{0}|v|^2} f_0 \|_\infty$ for some $0\leq \vartheta <\vartheta_{0}$. Moreover $f(t,x,v)$ is H\"older continuous in the following sense:
		\begin{equation} \label{est:Holder}
		\begin{split}
		&\sup_{0\leq t \leq T} 
		\sup_{ \substack{ (x,v)\in \overline{\O}\times \R^{3} \\ 0<|(x,v)-(\bx, \bv)|\leq 1   } }
		\Big| 
		\langle v \rangle^{-2\b} e^{-\varpi\langle v \rangle^{2}t}\frac{ | f(t,x,v) - f(t, \bx, \bv) | }{ |(x,v) - (\bx, \bv)|^{\b} }
		\Big|    \\
		&\lesssim_{\b} 
		\|e^{\vartheta_{0}|v|^{2}} f_{0}\|_{\infty}
		\Big[
		\sup_{ \substack{ v\in \R^{3} \\ 0 < |x - \bx|\leq 1   } }
		\langle v \rangle  \frac{|f_{0}( x, v ) - f_{0}(\bx, v)|}{|x - \bx|^{2\b}}  
		+ \sup_{ \substack{ x\in \overline{\O} \\ 0<|v - \bv|\leq 1   } }  \langle v \rangle^{2} \frac{|f_{0}( x, v ) - f_{0}( x, \bv)|}{|v - \bv|^{2\b}}  
		\Big]
		+ \mathcal{P}_{2}(\|e^{\vartheta_{0}|v|^{2}} f_{0}\|_{\infty}),
		\end{split}
		\end{equation}
		where $\mathcal{P}_{2}(s) = |s| + |s|^2$.   \\
\end{theorem}

Now we explain main ideas of the proof .
\subsection{A roadmap to the main theorem} 
To the sake of simplicity we drop out $\nu(f) f$ from (\ref{f_eqtn}) in this exhibition of key ideas:
\Be\label{f_eqtn:simple} 
\p_{t}f + v\cdot \nabla_{x} f  
= \Gamma_{\text{gain}}(f,f)  ,  
\ \ 
f|_{t=0} =f_0, \ \ f(t,x,v)|_{ \p\O}  = f(t,x,R_{x}v).
\Ee
From Hamiltonian \eqref{E_Ham} and \eqref{form_XV}, $f$ of \eqref{f_eqtn:simple} is given by 
\Be	\label{simple_expan}
\begin{split}
	f(t,x,v)  =  &  \ 
	f(0,X(0;t,x,v), V(0;t,x,v))
	+ \int^t_0
	\Gamma_{\text{gain}}(f,f)(s,X(s;t,x,v), V(s;t,x,v)) d s.
\end{split}
\Ee

\noindent {\it Step 1} ($C^{0,\frac{1}{2}}_{x,v}$ trajectory and nonlocal to local iteration) In order to estimate the H\"older semi-norm of $f$ we consider the difference quotients directly. Note that $V(s;t,x,v)$ has jump discontinuity at $s=t-\tb(x,v)$ which has to be handled carefully by  applying geometric splitting of $V(s)$ along the trajectory and the specular reflection BC. Then using a version of Carleman's representation and a priori $L^\infty$-bound of $f$, we can derive that for some $\theta$ and $\beta$,
\Be \label{lin_expan}
\begin{split}
	\frac{ | f(t,x,v+\zeta) - f(t, \bar{x}, \bar{v}+\zeta) | }{ |(x,v)-(\bar{x}, \bar{v})|^{\b} }   
	\lesssim&      \int_{0}^{t} 
	\frac{  |  X(s) - \bar{X}(s) |^{\theta} }{ |(x,v)-(\bar{x}, \bar{v})|^{\b} } 
	J_x[f(s)] (V(s);X(s),\bar{X}(s))
	ds  \\
	&  +   \int_{0}^{t}
	\frac{  |  V(s) - \bar{V}(s) |^{\theta} }{ |(x,v)-(\bar{x}, \bar{v})|^{\b} }  
	J_v[f(s)] (X(s);V(s), \bar{V}(s))
	+  good \  terms ,
\end{split}
\Ee
where \textbf{$\mathbf{J}$-seminorms} are defined as 
\Be\label{J_norm}
\begin{split}
	J_x[f(s)] (V(s);X(s),\bar{X}(s))&: = \int_{|u|\lesssim 1} \frac{ |f(s,X(s),u+V(s)) - f(s,\bar{X}(s),u+V(s))| }{ |X(s)-\bar{X}(s)|^{\theta} } du ,\\
	J_v[f(s)] (X(s);V(s), \bar{V}(s))&:=
	\int_{|u|\lesssim 1} \frac{ |f(s,X(s),u+V(s)) - f(s,X(s),u+\bar{V}(s))| }{ |V(s)-\bar{V}(s)|^{\theta} } du . 
\end{split}
\Ee
Here, $X(s)=X(s;t,x,v+\zeta)$, $V(s)=V(s;t,x,v+\zeta)$, $\bar{X}(s)=\bar{X}(s;t,\bar{x},\bar{v}+\zeta)$, and $\bar{V}(s)=\bar{V}(s;t,\bar{x},\bar{v}+\zeta)$. Our first key observation is that $(x,v) \mapsto X (s;t,x,v)$ is $C^{0,\frac{1}{2}}_{x,v}$ and $(x,v) \mapsto V (s;t,x,v) $ with $s \neq t-\tb (x,v)$ belongs to $C^{0,\frac{1}{2}}_{x,v}$! (For example, if we consider 2D circle $x^{2}+(y-1)^{2}=1$, then near grazing regime (let $x = (1,0)$, $\bar{x}=(1, \varepsilon)$, $v=(1,0)$), we have
$  
| \xb(x,v) - \xb(\bar{x}, v) | = \sqrt{2\varepsilon - \varepsilon^{2}} \simeq \sqrt{\varepsilon}
$, for $\varepsilon \ll 1$.)
Thereby a natural choice of $\theta$ is  
\Be\notag 
\theta=2\b. 
\Ee

A schematic bound (\ref{lin_expan}) exhibits the nonlocal effect of the Boltzmann equation. Our key strategy is a \textbf{nonlocal-to-local iteration}: namely we first estimate the $J$-seminorm and then use (\ref{lin_expan}) to bound the H\"older seminorm.  Indeed the $J$-seminorms \eqref{J_norm} can be bounded in a closed form through the Duhamel form of (\ref{simple_expan}) as

\begin{equation} \label{xiter}
\begin{split}
J_{x}[f(s)]( v; x, \bar{x})
&  \lesssim \   good \  terms    +   \int^t_0  
	\int_{|\zeta| \leq 1} \frac{|X(s) - \bar{X}(s)|^{2\beta}}{|x-\bar{x}|^{2\b}}  d\zeta 
ds \times
\sup_{s, v, x\neq \bar{x}}   J_{x}[f(s)] (v; x, \bar{x})   \\
&\quad +   \int^t_0 
	\int_{|\zeta|\leq 1} \frac{|V(s) - \bar{V}(s)|^{2\beta}}{|x-\bar{x}|^{2\b}}  d\zeta 
ds \times
 \sup_{s, x, v\neq \bar{v}}  J_{v} [f(s)](x; v, \bar{v}) ,
\end{split}
\end{equation}
and a similar expression for $J_{v}[f(s)] (x; v, \bar{v})$. Main ingredients of  \eqref{xiter} are difference quotients of $X(s)$ and $V(s)$. In fact the difference quotients would be very singular $\sim \frac{1}{|(v+\zeta) \cdot\nabla\xi(\xb)|}$ when the trajectories hit $\p\O$ near grazing region. However, measuring this singular behavior is very tricky, because of the curvature of the cross section $\p\O\cap \text{span}\{x - \bx, v\}$.  \\

\noindent {\it Step 2} (Shift method and ODE of Specular Singularity)  This step is the most technical step in this paper.  To get precise estimate of the difference quotients, we introduce two very important ideas : \textbf{Specular Singularity} and \textbf{Shift method}. \\

First, let us consider some simple linear parametrizations of position  $x(\tau)$ and velocity $v(\tau)$. If all the trajectories from $(x(\tau), v)$ with $0\leq \tau \leq 1$ hit $\p\O$, then difference quotient of trajectory will look like
\hide
\Be 
\mathfrak{S}_{sp}(\tau;x,\tilde{x}, v) : =  
\Big| \frac{ \nabla \xi(\xb( x(\tau), v)) \cdot v  }{
 \dot{x}(\tau) \cdot \nabla \xi (\xb( x(\tau), v)) 
} \Big|
,\quad
\mathfrak{S}_{vel}(\tau;x,  v,  \tilde {v}, \zeta) : =  \Big|
\frac{ \nabla \xi(\xb(x, v(\tau))) \cdot v(\tau)  }{ 
  \dot{v}(\tau) \cdot \nabla \xi (\xb(x , v(\tau))) 
}  \Big|.
\Ee
\unhide
\Be\label{DQ_x}
\frac{|X(s) - \bar{X}(s)|}{|x-\bar{x}|} 
\sim \int_{0}^{1} |\nabla X(s; t, x(\tau), v)| d\tau 
\sim \int_{0}^{1} \frac{1}{|v\cdot\nabla\xi(\xb(x(\tau), v))|} d\tau, 
\Ee
since $\nabla_{x}X(s) \sim \frac{1}{|v\cdot\nabla\xi(\xb(x(\tau), v))|} $. (We get simliar result for $v-\bar{v}$ case with $|v\cdot\nabla\xi(\xb(x, v(\tau)))|$.) Evidently, the success of estimating $J_x[f(s)]$ and $J_v[f(s)]$ relies on an efficient estimate of the integration with respect to $\zeta$ of the different quotients in (\ref{DQ_x}). Unfortunately, however, estimate \eqref{DQ_x} is not efficient : it does not consider the angle between $\dot{x}(\tau)$ (or $\dot{v}(\tau)$) and $\nabla\xi$ which is not uniform in general and hence causes trouble, i.e., it fails to reveal geometric property of $\O$ very well. \\

Meanwhile, to clarify the effect of convexity of $\O$ to the difference quotients, we introduce {\bf shifted} position and velocity. We shift postion and velocity to define $\tx$ and  $\tv$  as in \eqref{def_tildex} and \eqref{def_tildev}, respectively. Now, we use new parametrizations $\X(\tau)$ and $\V(\tau)$ in \eqref{xv para} with shifted $\tx$ and $\tv$ to redefine Specular Singularity (See \eqref{def:S_x} and \eqref{def:S_v} for their exact forms.)
\Be \label{spec_sing}
\mathfrak{S}_{sp}(\tau;x,\tilde{x}, v) : =  
\Big| \frac{ \nabla \xi(\xb( \X(\tau), v)) \cdot v  }{
	\dot{\X}(\tau) \cdot \nabla \xi (\xb( \X(\tau), v)) 
} \Big|
,\quad
\mathfrak{S}_{vel}(\tau;x,  v,  \tilde {v}, \zeta) : =  \Big|
\frac{ \nabla \xi(\xb(x, \V(\tau))) \cdot \V(\tau)  }{ 
	\dot{\V}(\tau) \cdot \nabla \xi (\xb(x , \V(\tau))) 
}  \Big|.
\Ee
These crucial quantities measure how singular region does the trajectory hits by its numerators. Moreover, the denominator takes into account the angle between $\dot{\X}(\tau)$ (also for $\dot{\V}(\tau)$) and $\nabla\xi$ and hence clarifies the effect of non-convexity in ODE method which will be explained very soon. Note that with the new parametrizations, we also have the following orthogonal geometric properties
\Be \label{perp short}
	\dot{\X} \perp v\quad\text{and}\quad \V(\tau)\perp\dot{\V}(\tau)
\Ee
which are used \textit{crucially} throughout this paper. See Figure \ref{fig1} and \ref{fig2}. Now, with \eqref{spec_sing}, \eqref{DQ_x} newly look like
\Be\label{DQ_x new}
\frac{|X(s) - \bar{X}(s)|}{|x-\bar{x}|} 
\sim \int_{0}^{1} |\nabla X(s; t, x(\tau), v)| d\tau 
\sim \int_{0}^{1} \frac{1}{\mathfrak{S}_{sp}(\tau; x, \tx, v)} d\tau.
\Ee
To perform estimate for \eqref{DQ_x new}, our next step is {\bf ODE method for $\mathfrak{S}_{sp, vel}$}. Using convexity \eqref{convex_xi} and \eqref{perp short} {\it crucially}, the Specular Singularity $\mathfrak{S}_{sp}$ solves ODE look like
\Be \notag %
\begin{split}
	\frac{d}{d\tau}\mathfrak{S}^{2}_{sp}(\tau) \gtrsim_{\O} \frac{1}{|\dot{\X}\cdot\nabla\xi(\xb(\X(\tau),v))|} \mathfrak{S}^{2}_{sp}(\tau).
\end{split}
\Ee
ODE for $\mathfrak{S}_{vel}$ is also similar. (See \eqref{ODE:S_x} and \eqref{ODE:S_v} in Lemma \ref{lemma:ODE} for exact ODEs). We note that the denominators of \eqref{spec_sing} are very important in its definition when we derive above type of ODEs. Its qualitative aspect was already explained. In the technical aspect, $\frac{d}{d\tau}\mathfrak{S}^{2}_{sp, vel}$ gain positive lower bound via convexity (see \eqref{lower:dS_force} and \eqref{lower:dS_force v}) only when they contain the denominators. 
Therefore, we obtain difference quotients estimates via integration of Specular Singularity, (in general we will use $v+\zeta$ instead of $v$ for notational convenience in velocity integration)
\Be \label{fraction scheme}
\frac{|X(s) - \bar{X}(s)|}{|x-\bar{x}|} 
\lesssim
\frac{1}{|(v+\zeta)\cdot \nabla\xi(\xb(x,v+\zeta))|} + \frac{1}{|(v+\zeta)\cdot \nabla\xi(\xb(\tilde{x},v+\zeta))|}.
\Ee
(See \eqref{int:Sx} and \eqref{int:Sv} of Prposition \ref{prop_avg S} for exact estimates.) We also note that $\tilde{x}$ depends on $\zeta$ surely, but $(v+\zeta)\cdot \nabla\xi(\xb(\tilde{x},v)) = (v+\zeta)\cdot \nabla\xi(\xb(\bx,v))$ by definition of $\tx$ in \eqref{def_tildex}. Thus {\bf we can fix positions on $x$ or $\bx$}, independent to $\zeta$, for each terms in \eqref{fraction scheme}. This is a huge advantage in terms of \eqref{xiter}, because we should integrate \eqref{fraction scheme} with respect to $\zeta\in\R^{3}$ again. If the position moves depending on velocity $\zeta$, it looks nearly impossible to perform the integration!  \\

Now performing sharp integral estimate (Lemma \ref{lem_int sing})
\[
	\int_{\zeta} \frac{1}{|(v+\zeta)\cdot \nabla\xi(\xb(x,v+\zeta))|^{2\b}} d\zeta \lesssim 1,\quad \b<\frac{1}{2}, 
\]
we obtain uniform bound of $J_{x}[f]$ and $J_{v}[f]$ in \eqref{J_norm}. Exact definitions and estimate of the {\bf J-seminorms} are given in Definition \ref{def_H} and \eqref{est : H} in Proposition \ref{prop_unif H}, respectively.   \\

Lastly, from uniform estimates of J-seminorm and $C^{0,\frac{1}{2}}_{x,v}$ estimate of trajectory (Lemma \ref{lem:Holder_tb}), we obtain $C^{0,\frac{1}{2}-}_{x,v}$ regularity of $f$ via nonlocal-to-local estimate using \eqref{lin_expan}. \\

\section{Preliminaries}

 In this section, we state lemmas for some estimates of $\Gamma_{\text{gain}}(f,f)$.  
	
	\begin{lemma}[Carleman estimate] \label{lem_Carl}
	Recall (\ref{Q}) and (\ref{Gamma}). We have the following expressions:
		\Be \label{Carl1}
		\Gamma_{\text{gain}}(g_{1}, g_{2})(v) = C \int_{\R^{3}} d u \int_{u\cdot z = 0} d z g_{1}(v+z) g_{2}(v+u) q_{0}^{*}\Big( \frac{|u|}{|u+z|} \Big) \frac{|u+z|^{\kappa-1}}{|u|} e^{-\frac{ |u+v+z|^{2} }{4} },
		\Ee 
		\Be \label{Carl2}
		\Gamma_{\text{gain}}(g_{1}, g_{2})(v) = C \int_{\R^{3}} d z \int_{u\cdot z = 0} d u g_{1}(v+z) g_{2}(v+u) q_{0}^{*}\Big( \frac{|u|}{|u+z|} \Big) \frac{|u+z|^{\kappa-1}}{|z|} e^{-\frac{ |u+v+z|^{2} }{4} },
		\Ee
		where $q_{0}^{*}(|\cos\theta|) := \frac{1}{|\cos\theta|}q_{0}(|\cos\theta|)$. For the hard sphere case $(\kappa=1)$, $q_{0}^{*}(|\cos\theta|) := \frac{1}{|\cos\theta|}|\cos\theta|=1$. 
		
		\hide
		Moreover 
		\Be\label{k_gain}
		\begin{split}
		\Gamma_{\text{gain}}(g, \sqrt{\mu})(v)&=
	\int_{\R^3}	\mathbf{k}_{\text{gain}}(v,u) g(u) \dd u,\\
		\mathbf{k}_{\text{gain}}(v,u)&=.
	\end{split}
		\Ee
		\unhide
		
	\end{lemma}
 
	\begin{proof}
		Using $u^{\prime} = u + ((v-u)\cdot \omega) \omega$ and $v^{\prime} = v - ((v-u)\cdot \omega) \omega$, we write
		\[
		\Gamma_{\text{gain}}(g_{1}, g_{2}) = \int_{u\in\R^{3}}\int_{\omega\in\mathbb{S}^{2}} |v-u|^{\kappa}q_{0}\Big( \frac{|(v-u)\cdot\omega|}{|v-u|} \Big) \sqrt{\mu}(u) g_{1}(u + ((v-u)\cdot \omega) \omega) g_{2}(v - ((v-u)\cdot \omega) \omega) d \omega d u.
		\]
		We change $u\rightarrow u+v$ to get
		\[
		\Gamma_{\text{gain}}(g_{1}, g_{2}) = \int_{u\in\R^{3}}\int_{\omega\in\mathbb{S}^{2}} |u|^{\kappa}q_{0}\Big( \frac{|u\cdot\omega|}{|u|} \Big) \sqrt{\mu}(u+v) g_{1}(u+v - (u\cdot \omega) \omega) g_{2}(v + (u\cdot \omega) \omega) d \omega d u.
		\]
		Let us decompose and define $u_{\parallel} := (u\cdot\omega)\omega$ and $u_{\perp} := u - u_{\parallel}$. Then,
		\[
		\Gamma_{\text{gain}}(g_{1}, g_{2}) = \int_{u\in\R^{3}}\int_{\omega\in\mathbb{S}^{2}} |u_{\parallel}+u_{\perp}|^{\kappa}q_{0}\Big( \frac{|u_{\parallel}|}{|u_{\parallel}+u_{\perp}|} \Big) e^{\frac{1}{4}|u_{\parallel}+u_{\perp}+v|^{2}} g_{1}(u_{\perp}+v) g_{2}(v + u_{\parallel}) d \omega d u.
		\]
		Now, for (\ref{Carl1}), for fixed $\omega\in\mathbb{S}^{2}$, we have $d u = 2 d u_{\perp} d |u_{\parallel}|$, where $u_{\parallel} = |u_{\parallel}|\omega$. Therefore (also see \cite{Guo_Classic}),
		\[
		d u d \omega = \frac{2}{|u_{\parallel}|^{2}} d u_{\perp} |u_{\parallel}|^{2} d |u_{\parallel}| d \omega = \frac{2}{|u_{\parallel}|^{2}} d u_{\perp} d u_{\parallel}.
		\]
		Now we just rewrite $u_{\parallel}\rightarrow u$ and $u_{\perp}\rightarrow w$ to obtain
		\Be \notag
		\begin{split}
			\Gamma_{\text{gain}}(g_{1}, g_{2}) &= \int_{u\in\R^{3}}\int_{u\cdot w=0} |u+w|^{\kappa}q_{0}
			\Big( \frac{|u 
			|}{|u +w|} \Big) 
			e^{\frac{1}{4}|u+w+v|^{2}} g_{1}(v+w) g_{2}(v + u) \frac{2}{|u|^{2}} d w d u  \\
			&= 2 \int_{u\in\R^{3}}\int_{u\cdot w=0} \frac{|u+w|^{\kappa-1}}{|u|} 
			\frac{|u+w|}{|u|} 
			q_{0}\Big( \frac{|u |}{|u +w|} \Big) e^{\frac{1}{4}|u+w+v|^{2}} g_{1}(v+w) g_{2}(v + u) d w d u .  \\
		\end{split}
		\Ee
		We define $q_{0}^{*}(\cos\theta) := \frac{1}{|\cos\theta|}q_{0}(\cos\theta)$ and this finishes the proof for (\ref{Carl1}).

		For (\ref{Carl2}), for fixed $\omega\in\mathbb{S}^{2}$, we choose $\bar{\omega} = \bar{\omega}(\o)\in \mathbb{S}^{2}$ so that $\bar{\o}\perp \o$ and a map $\o \mapsto \bar{\o}(\o)$ is measure preserving bijective map via rotation from $\mathbb{S}^{2}$ to $\mathbb{S}^{2}$. Then we have 
		\[
		dud{\o} = 2|u_{\perp}| d|u_{\parallel}|d|u_{\perp}|d\theta d\o,
		\]  
		where $\theta$ is angle of polar coordinate with $u_{\parallel}$ axis. Now consider the rotation $\o\rightarrow \bar{\o}$ which makes $\theta\rightarrow \bar{\theta}$, where $\bar{\theta}$ is angle of polar coordinate with $u_{\perp}$. Therefore (also see \cite{Guo_Classic}),  
		\[
		dud{\o} = 2|u_{\perp}| d|u_{\parallel}|d|u_{\perp}|d\bar{\theta} d\bar{\o} = 2 \frac{1}{|u_{\perp}|} \big( |u_{\perp}|^{2} d|u_{\perp}|d\bar{\o} \big) \frac{1}{|u_{\parallel}|}\big( |u_{\parallel}|d|u_{\parallel}| d\bar{\theta} \big) = \frac{2}{|u_{\parallel}||u_{\perp}|} du_{\parallel} du_{\perp},
		\]
		where $u_{\perp} \in\R^{3}$ and $u_{\parallel} \perp u_{\perp}$. Finally rewriting $u_\parallel \mapsto u$ and $u_\perp \mapsto w$, we get
		\Be \notag
		\begin{split}
			\Gamma_{\text{gain}}(g_{1}, g_{2}) 
			&= 2 \int_{w\in\R^{3}}\int_{u\cdot w=0} |u+w|^{\kappa} q_{0}\Big( \frac{|u |}{|u +w|} \Big) e^{\frac{1}{4}|u+w+v|^{2}} g_{1}(v+w) g_{2}(v + u) \frac{1}{|w||u|} d u d w ,   
		\end{split} 
		\Ee
		and this gives (\ref{Carl2}) with the definition of $q_0^*$.
	\end{proof}

	\begin{lemma}[Nonlinear estimates] \label{lem_Gamma}
		Let $w(v) = e^{\vartheta|v|^{2}}$ for $0<\vartheta < \frac{1}{4}$ and 
		$x, \bx \in \O$, $v, \bv, \zeta \in\R^{3}$. For $|(x,v) - (\bx, \bv)| \leq 1$, we have the following estimates
		\Be \label{full k v}
		\begin{split}
			 \frac{| \Gamma_{\text{gain}}(f, f)(x,v) - \Gamma_{\text{gain}}(f, f)(x,\bar{v}) |}{|v-\bar{v}|^{2\b}}   
			\lesssim & \ \|wf\|_{\infty} \int_{\R^3} 
			\mathbf{k}_{c}(v, \bar v, u)
			\frac{ | f(x, v+u) - f(x, \bar{v}+u) | }{ |v-\bar{v}|^{2\b} }  du \\
			&
			+ \|wf\|_{\infty}^{2}  \min \{	 \langle v\rangle^{-1}	, \langle \bar v\rangle^{-1}		   \},
		\end{split}
		\Ee
	\Be \label{full k x} 
			 \frac{| \Gamma_{\text{gain}}(f , f )(x,v) - \Gamma_{\text{gain}}(f , f )(\bar{x},v) |}{|x-\bar{x}|^{2\b}}   
			 \lesssim \|wf \|_{\infty} \int_{\R^3} 	k_{c}(v, v+u)\frac{ |f (x, v+u) - f (\bar{x}, v+u)| }{ |x-\bar{x}|^{2\b} } du, 
		\Ee
		for some $c>0$, where $k_{c}$ and $\mathbf{k}_{c}$ are defined in \eqref{full k} and \eqref{bar k} of Definition \ref{def_k}.
	\end{lemma}
	\begin{proof}
Applying (\ref{Carl1}) to $\Gamma_{\text{gain}}(f , f )(x,v)$, and 
(\ref{Carl2}) to $\Gamma_{\text{gain}}(f , f )(x,\bar{v})$, we derive 
\hide where we replace $v$ with $\bar{v}$ we get that $\Gamma_{\text{gain}}(f , f )(x,v) - \Gamma_{\text{gain}}(f , f )(x,\bar{v}) = C\int_{u}\int_{u\cdot z=0} \frac{1}{|u|} \sqrt{\mu}(u+v+z)f (v+z) f (v+u) dz du 
			- C\int_{z}\int_{u\cdot z=0} \frac{1}{|z|}\sqrt{\mu}(u+\bar{v}+z)f (\bar{v}+u) f (\bar{v}+z) du dz$ equals \unhide
		\hide
		{\color{blue}
			\Be 
			\begin{split}
				&\Gamma_{\text{gain}}(f_{1}, f_{2})(x,v) - \Gamma_{\text{gain}}(f_{1}, f_{2})(x,\bar{v})  \\
				&= C\int_{u}\int_{u\cdot z=0} \frac{1}{|u|} \sqrt{\mu}(u+v+z)f_{1}(v+z)\big( f_{2}(v+u) - f_{2}(\bar{v}+u) \big) dz du \\
				&\quad + C\int_{z}\int_{u\cdot z=0} \frac{|u|}{|z|^{2}}\sqrt{\mu}(u+\bar{v}+z)f_{2}(\bar{v}+u)\big( f_{1}(v+z) - f_{1}(\bar{v}+z) \big) du dz \\
				&\quad + C\int_{u}\int_{u\cdot z=0} \frac{1}{|u|} f_{1}(v+z)f_{2}(\bar{v}+u) \big( \sqrt{\mu}(u+v+z) - \sqrt{\mu}(u + \bar{v} +z) \big) dz du ,  \\
			\end{split}
			\Ee
		} 
		\unhide
		\Be \label{Gamma expan}
		\begin{split}
			&\Gamma_{\text{gain}}(f , f )( v) - \Gamma_{\text{gain}}(f , f )( \bar{v}) = C\int_{\R^3}du \int_{u\cdot z=0}dz \frac{1}{|u|} \sqrt{\mu(u+v+z)}f (v+z) f (v+u)  \\
			&	  \ \ \ \	- C\int_{\R^3}dz \int_{u\cdot z=0}du   \frac{1}{|z|}\sqrt{\mu(u+\bar{v}+z)}f (\bar{v}+u) f (\bar{v}+z) 
			\\
			& = \ C\int_{ u \in \R^3}\int_{u\cdot z=0} \frac{1}{|u|} \sqrt{\mu(u+v+z)}f (v+z)\big( f (v+u) - f (v+u-(v-\bar{v})) \big) dz du \\
			&\quad + C\int_{z\in \R^3}\int_{u\cdot z=0} \frac{1}{|z|}\sqrt{\mu (u+\bar{v}+z)}f (\bar{v}+u)\big( f (v+z) - f (v+z-(v-\bar{v})) \big) du dz \\
			& \quad + C
			\Big\{\int_{ R^3}\int_{u\cdot z=0} \frac{ \sqrt{\mu (u+v+z)}}{|u|}  f (v+z)f (\bar{v}+u) dz du    \\
			&\quad\quad\quad\quad	-  \int_{ \R^3}\int_{u\cdot z=0} \frac{\sqrt{\mu (u + \bar{v} +z)}}{|z|}  f (v+z)f (\bar{v}+u)du  dz
			\Big\} .  
	\end{split}
		\Ee
		\hide
		where we have used (\ref{Carl1}) in the first term of RHS and (\ref{Carl2}) in the second term. For the first term on the RHS, we use Young's inequality to claim 
		\Be \begin{split}\label{integral_Carleman}
			\int_{ u  \cdot z=0}
			\frac{e^{- \frac{|  u +  v+ z|^2}{4}}}{w( v + z)|u |}
			d z 
			\lesssim    e^{- \frac{\theta}{4}
				|u|^2
			}
			\int_{ u  \cdot z=0}
			\frac{e^{- \frac{\theta}{4}
					|v+ z|^2} }{|u|}
			d z
			\lesssim  
			\frac{e^{- \frac{\theta}{4}
					|u|^2
			}}{|u|} \int_{ u  \cdot (z-v)=0}
			{e^{- \frac{\theta}{4}
					|  z|^2} } 
			d z
			\lesssim
			\frac{e^{- \frac{\theta}{4}
					|u|^2
			}}{|u|} ,
		\end{split}\Ee 
		for some small constant $\theta \ll 1$. For the second term, we repeat (\ref{integral_Carleman}) changing $dz$ into $du$ to obtain
		\[
		\int_{ u  \cdot z=0}
		\frac{e^{- \frac{|  u +  \bar{v} + z|^2}{4}}}{w( \bar{v} + u)|z |}
		d u 
		\lesssim   
		\frac{e^{- \frac{\theta}{4}
				|z|^2
		}}{|z|}.
		\]
		\[
		\int_{ u  \cdot z=0}
		\frac{e^{- \frac{|  u +  \bar{v}+ z|^2}{4}}}{w( \bar{v} + u)} |u|
		d u
		\lesssim    e^{- \frac{\theta}{4}
			|z|^2
		}
		\int_{ u  \cdot z=0}
		e^{- \frac{\theta}{4}
			|\bar{v}+ u|^2} 
		|u| d u
		\lesssim  
		e^{- \frac{\theta}{4}
			|z|^2
		} \int_{ z  \cdot (u-\bar{v})=0}
		{e^{- \frac{\theta}{4}
				| u|^2} } 
		|u-\bar{v}| d u
		\lesssim
		\langle \bar{v} \rangle 
		e^{- \frac{\theta}{4}
			|z|^2
		} ,
		\]
		and therefore,
		\Be \label{integral_Carleman z} 
		\begin{split} 
			&\int_{z}\int_{u\cdot z=0} \frac{|u|}{|z|^{2}}\sqrt{\mu}(u+\bar{v}+z)f_{2}(\bar{v}+u)\big( f_{1}(v+z) - f_{1}(\bar{v}+z) \big) du dz  \\
			&\lesssim \|wf_{2}\|_{\infty} \int_{z} \frac{e^{-\frac{\theta}{4}|z|^{2}}}{|z|^{2}} \langle \bar{v} \rangle | f_{1}(v+z) - f_{1}(\bar{v}+z) | dz  \\
			&\lesssim \|wf_{2}\|_{\infty} \Big\{ \int_{|z|\geq \langle \bar{v} \rangle} + \int_{|z|\leq \langle \bar{v} \rangle}    \Big\}  \\
			&\lesssim \|wf_{2}\|_{\infty} \int_{z} \frac{e^{-\frac{\theta}{4}|z|^{2}}}{|z|} | f_{1}(v+z) - f_{1}(\bar{v}+z) | dz \\
			&\quad + \|wf_{2}\|_{\infty} \|wf_{1}\|_{\infty}  \langle \bar{v} \rangle \int_{|z|\leq \langle \bar{v} \rangle}  \frac{e^{-\frac{\theta}{4}|z|^{2}}}{|z|^{2}} \big( w^{-1}(v+z) + w^{-1}(\bar{v}+z) \big) dz  \\
			&\lesssim \|wf_{2}\|_{\infty} \int_{z} \frac{e^{-\frac{\theta}{4}|z|^{2}}}{|z|} | f_{1}(v+z) - f_{1}(\bar{v}+z) | dz \\
			&\quad + \|wf_{2}\|_{\infty} \|wf_{1}\|_{\infty}  \langle \bar{v} \rangle \int_{|z|\leq \langle \bar{v} \rangle}  \frac{e^{-\frac{\theta}{8}|z|^{2}}}{|z|^{2}} \big( e^{-\frac{\theta}{8}|v|^{2}} + e^{-\frac{\theta}{8}|\bar{v}|^{2}} \big) dz  \\
			&\lesssim \|wf_{2}\|_{\infty} \int_{z} \frac{e^{-\frac{\theta}{4}|z|^{2}}}{|z|} | f_{1}(v+z) - f_{1}(\bar{v}+z) | dz + \|wf_{2}\|_{\infty} \|wf_{1}\|_{\infty} \big( 	e^{- \frac{\theta}{16} |\bar{v}|^2 	} 	+ 	e^{- \frac{\theta}{16} |v|^2 	}  \big)
		\end{split}\Ee 
		for some small constant $\theta \ll 1$ since $|v-\bar{v}|\leq 1$.  \\
		\unhide
		
		For the first term in RHS of (\ref{Gamma expan}), from the standard estimates (e.g. (3.11) and (3.52) in \cite{gl}), we derive, for some $c_{\vartheta^{\prime}}>0$, 
		\Be \notag 
		\begin{split}
	(\ref{Gamma expan})_1		&\lesssim \|wf \|_{\infty} \int_{u}\int_{u\cdot z=0} \frac{1}{|u|} e^{-\frac{|u+v+z|^{2}}{4}} e^{-\vartheta^{\prime}|v+z|^{2}} 
	| f (v+u) - f(v+u-(v-\bar{v})) | dz du \\
			&\lesssim \|wf \|_{\infty} \int_{\R^3} \frac{1}{|u|}e^{ - c_{\vartheta^{\prime}}|u|^{2} - c_{\vartheta^{\prime}}\frac{ | |v|^{2}-|v+u|^{2} |^{2} }{|u|^{2}} } \big| f (v+u) - f (v+u-(v-\bar{v})) \big| du.
		\end{split}
		\Ee
	  Similarly, the second term in RHS of (\ref{Gamma expan}) is bounded by 
		\Be \notag 
	(\ref{Gamma expan})_2 \lesssim	\|wf \|_{\infty} \int_{\R^3} \frac{1}{|u|}e^{ - c_{\vartheta^{\prime}}|u|^{2} - c_{\vartheta^{\prime}}\frac{ | |\bar{v}|^{2}-|\bar{v}+u|^{2} |^{2} }{|u|^{2}} } \big| f (v+u) - f (v+u-(v-\bar{v})) \big| du.
		\Ee

We only need to bound last line of (\ref{Gamma expan}) by the last term of \eqref{full k v}. Following the proof of Lemma \ref{lem_Carl}, we re-express the second integral of (\ref{Gamma expan}) into $u \in \R^3$ integration (and an integral over $z \in \{z\cdot u=0\}$). Then we bound the last line of (\ref{Gamma expan}) by  
		\Be  \label{integral_Carleman 2}
		 |v-\bar{v}|^{2\b} \| w f\|_\infty^2 \times  \underbrace{	 \int_{u}\int_{ u  \cdot z=0}
			\frac{ 1
			 }{w( v + z)w( \bar{v} + u)|u |}
			\frac{| \sqrt{\mu (u+v+z) }- \sqrt{\mu (u + \bar{v} +z)} |}{|v-\bar{v}|^{2\b}}
			d z  du }_{\eqref{integral_Carleman 2}_*} . 
			\Ee
				Note that $|\sqrt{ \mu (u+\bar{v} + v- \bar{v}+z)} - \sqrt{\mu (u + \bar{v} +z)}|  
		=  \Big| \int_0^1  (v- \bar v) \cdot \nabla \sqrt{\mu (u+\bar{v} + (v- \bar{v})s+z)}  \dd s\Big| 
		 \lesssim    |v-\bar{v} | 
		 .$
		Hence 
		\Be
		\begin{split}
		\eqref{integral_Carleman 2}_* & \lesssim 
		|v-\bar v|^{1-2\beta}\int_0^1 \dd s 
		\int_{\R^3}\dd u  |u|^{-1} e^{- \vartheta^\prime | \bar{v} + u|^2} \int_{ u\cdot z=0} \dd z e^{- \vartheta^\prime |v+z|^2} \lesssim \langle \bar v \rangle ^{-1}{|v-\bar v|^{1-2\beta}}. \notag
		\end{split}\Ee
Hence we conclude that the last line of (\ref{Gamma expan}) is bounded above by $\langle \bar v \rangle ^{-1} |v-\bar v|  \| w f\|_\infty^2$. Now changing $v$ and $\bar v$ and then following the same argument we also get the upper bound $\langle v \rangle ^{-1} |v-\bar v|  \| w f\|_\infty^2$.\hide
 
		Note that 
		\Be
		\begin{split}
		\frac{  e^{-\frac{|u+\bar{v} + (v- \bar{v})s+z|^2}{4+\delta}} 
			 }{w( v + z)w( \bar{v} + u) } & = e^{-\frac{1}{4+\delta}|u+\bar{v} + (v- \bar{v})s+z|^2
			 - \vartheta^\prime |v+z|^2 - \vartheta^\prime |\bar v+u|^2
			 }\\
			 &= e^{- \frac{\vartheta^\prime}{8}
			 \{
			 \}
			 }
		\end{split}
		\Ee

			{\color{red} [CK] I don't think the next estimates are correct.}
			\Be \begin{split}
			&\lesssim   |v-\bar{v}|^{2\b}
			\int_{u}\int_{ u  \cdot z=0} \frac{1}{|u|}
			\Big( 
			\frac{ e^{- \frac{\theta| v+ z|^2}{8}} 	  e^{- \frac{\theta}{8} |u|^2 	} }{w( \bar{v} + u)}   
			+
			\frac{ e^{- \frac{\theta|  u +  \bar{v} |^2}{8}} 	  e^{- \frac{\theta}{8} |z|^2 	} }{w( v + z)}   
			\Big)   dz du  \\
			&\lesssim   |v-\bar{v}|^{2\b}
			\int_{u}\int_{ u  \cdot z=0} \frac{1}{|u|}
			\Big( 
			e^{- \frac{\theta| v+ z|^2}{8}} 	  e^{- \frac{\theta}{16} |\bar{v}|^2 	} e^{- \frac{\theta}{16} |\bar{v}+u|^2 	} 
			+
			e^{- \frac{\theta|  u +  \bar{v} |^2}{8}}  e^{- \frac{\theta}{16} |v|^2 	} e^{- \frac{\theta}{16} |v+z|^2 	}  
			\Big)   dz du  \\
			&\lesssim   |v-\bar{v}|^{2\b}
			\big( 
			e^{- \frac{\theta}{16} |\bar{v}|^2 	} 
			+
			e^{- \frac{\theta}{16} |v|^2 	}
			\big).    
		\end{split}\Ee

		Now, combining (\ref{Full k 1}), (\ref{Full k 2}), and (\ref{integral_Carleman 2}), and dividing by $|(x,v) - (\bar{x}, \bar{v})|^{2\b}$, we obtain
		\Be \notag
		\begin{split}
			&\frac{| \Gamma_{\text{gain}}(f, f)(x,v) - \Gamma_{\text{gain}}(f, f)(x,\bar{v}) |}{|v-\bar{v}|^{2\b}}  \\
			&\lesssim \|wf\|_{\infty} \int_{u} \frac{1}{|u|}e^{ - c|u|^{2} - c\frac{ | |v|^{2}-|v+u|^{2} |^{2} }{|u|^{2}} } \frac{ | f(v+u) - f(v+u-(v-\bar{v})) }{ |v-\bar{v}|^{2\b} }  du  \\
			&\quad + \|wf\|_{\infty} \int_{u} \frac{1}{|u|}e^{ - c|u|^{2} - c\frac{ | |\bar{v}|^{2}-|\bar{v}+u|^{2} |^{2} }{|u|^{2}} } \frac{ | f(\bar{v}+u) - f(\bar{v}+u-(\bar{v}-v)) | }{|v-\bar{v}|^{2\b}}  du \\
			&\quad + \|wf\|_{\infty}^{2} \big( 	e^{- \frac{c}{16} |\bar{v}|^2 	} 	+ 	e^{- \frac{c}{16} |v|^2 	}  \big)  \\
		\end{split}
		\Ee
		
		\unhide The proof of (\ref{full k x}) is similar but simpler.%
		%
		\hide
		\Be
		\begin{split}
			&\Gamma_{\text{gain}}(f_{1}, f_{2})(x,v) - \Gamma_{\text{gain}}(f_{1}, f_{2})(\bar{x},v)  \\
			&= C\int_{u}\int_{u\cdot z=0} \frac{1}{|u|} \sqrt{\mu}(u+v+z)\big( f_{1}(x,v+z) f_{2}(x,v+u) - f_{1}(\bar{x},v+z) f_{2}(\bar{x},v+u) \big)  dz du \\
			&= C\int_{u}\int_{u\cdot z=0} \frac{1}{|u|} \sqrt{\mu}(u+v+z) f_{1}(x,v+z) \big( f_{2}(x,v+u) - f_{2}(\bar{x},v+u) \big)  dz du \\
			&\quad + C\int_{u}\int_{u\cdot z=0} \frac{1}{|u|} \sqrt{\mu}(u+v+z) f_{2}(\bar{x},v+z) \big( f_{1}(x,v+u) - f_{1}(\bar{x},v+u) \big)  dz du \\
			&\lesssim \|wf\|_{\infty} \int_{u} \frac{e^{-\theta|u|^{2}}}{|u|} \frac{ |f_{1}(x,u+v) - f_{1}(\bar{x},u+v)| }{ |x-\bar{x}|^{2\b} } du  
			+ \|wf\|_{\infty} \int_{u} \frac{e^{-\theta|u|^{2}}}{|u|} \frac{ |f_{2}(x,u+v) - f_{2}(\bar{x},u+v)| }{ |x-\bar{x}|^{2\b} } du 
		\end{split}
		\Ee
		\unhide
	\end{proof}

	Similar estimate for $\nu(f)$ is simpler.
	\begin{lemma} \label{lem_nu}
		Let $w(v) = e^{\vartheta|v|^{2}}$ for $0<\vartheta < \frac{1}{4}$, and $x, \bx \in \O$, $v, \bv \in\R^{3}$. 
		For some positive $c > 0$,  we obtain
		\Be \label{full nu v}
		\begin{split}
			\nu(f)(t, x, v) - \nu(f)(t, x, \bv)  
			&\leq C|v - \bv| \|f(t)\|_{\infty},
		\end{split}
		\Ee
		and
		\Be \label{full nu x} 
			\nu(f)(t, x, v) - \nu(f)(t, \bx, v)  
			\leq C\langle v \rangle \int_{\R^{3}} k_{c}(0,u) |f(t, x, u) - f(t, \bx, u)| du.   \\
		\Ee
	\end{lemma}
	\begin{proof}
		For \eqref{full nu v},
		\begin{equation} \notag %
		\begin{split}
		\nu(f)(t, x, v) - \nu(f)(t, x, \bv)  &= C\int_{\R^{3}_{u}} |v - u| \sqrt{\mu(u)} f(t, x, u) du - C\int_{\R^{3}_{u}} |\bv - u| \sqrt{\mu(u)} f(t, x, u) du \\
		&\leq C|v - \bv| \|f(t)\|_{\infty}.
		\end{split}
		\end{equation}
		For \eqref{full nu x},
		\begin{equation}\notag %
		\begin{split}
			 \nu(f)(t, x, v) - \nu(f)(t, \bx, v)  &\leq C\int_{\R^{3}_{u}} |v-u| \sqrt{\mu(u)} |f(t, x, u) - f(t, \bx, u)| du \\
			 &\leq C\langle v \rangle \int_{\R^{3}} k_{c}(0,u) |f(t, x, u) - f(t, \bx, u)| du. \\
		\end{split}
		\end{equation}
		for some generic small $c>0$. 
	\end{proof}
	
	\begin{lemma}[Uniform negativity] \label{lem_nega}
		 For $|\varpi s| < c$,
		\Be \label{ws nega}
		e^{-\varpi(1 + |v|^{2})s} e^{\varpi(1 + |v+\zeta|^{2})s} k_{c}(v, v+\zeta)  \leq k_{\frac{c}{2}}(v, v+\zeta).
		\Ee
		When $|v-\bar{v}| \leq 1$ and $\varpi s < (\sqrt{20}-4)\frac{c}{2}$,
		\Be \label{ws nega bar}
		e^{ -\varpi(1+|\bar{v}|^{2})s + \varpi(1+|\bar{v}+\zeta|^{2})s } k_{c}(v,v+\zeta) \lesssim k_{\frac{c}{2}}(v,v+\zeta).
		\Ee
		Moreover, from \eqref{ws nega} and \eqref{ws nega bar},
		\Be \label{bf nega}
		\begin{split}
			e^{ -\varpi(1+|v|^{2})s + \varpi(1+|v+\zeta|^{2})s } \mathbf{k}_{c}(v, \bv, \zeta) &\lesssim \mathbf{k}_{\frac{c}{2}}(v, \bv, \zeta), \\
			e^{ -\varpi(1+|\bv|^{2})s + \varpi(1+|\bv+\zeta|^{2})s } \mathbf{k}_{c}(v, \bv, \zeta) &\lesssim \mathbf{k}_{\frac{c}{2}}(v, \bv, \zeta),
		\end{split}
		\Ee
		hold when $|v-\bar{v}| \leq 1$ and $\varpi s < (\sqrt{20}-4)\frac{c}{2}$.
	\end{lemma}
	\begin{proof}
		For $0 \leq \theta < 2c$, we have
		\Be \notag 
		-\theta|v|^{2} + \theta|v+\zeta|^{2} - c|\zeta|^{2} - c\frac{1}{|\zeta|^{2}} \big| |v|^{2}-|v+\zeta|^{2} \big|^{2} < 0,
		\Ee
		by uniform negativity of quadratic form (refer Lemma 3 of \cite{Guo10}). (\ref{ws nega}) is obtained by replacing $\theta$ into $\varpi s$. When $|v-\bar{v}| \leq 1$, for $\theta < (\sqrt{20}-4)c$,
		\Be \notag 
		-\theta|\bar{v}|^{2} + \theta|\bar{v}+\zeta|^{2} - c|\zeta|^{2} - c\frac{1}{|\zeta|^{2}} \big| |v|^{2}-|v+\zeta|^{2} \big|^{2} < 2\theta,
		\Ee
		\begin{equation} \label{P comp}
		\begin{split}
		&-\theta|\bar{v}|^{2} + \theta|\bar{v}+\zeta|^{2} - c|\zeta|^{2} - c\frac{1}{|\zeta|^{2}} \big| |v|^{2}-|v+\zeta|^{2} \big|^{2} \\
		&= 2\theta(\bar{v}\cdot \zeta) + \theta|\zeta|^{2} - c|\zeta|^{2} - c\Big( |\zeta|^{2} + 4(v\cdot\zeta) + 4\frac{|v\cdot\zeta|^{2}}{|\zeta|^{2}} \Big)  \\
		&\leq 2\theta(v\cdot \zeta) + 2\theta|\zeta| + \theta|\zeta|^{2} - c|\zeta|^{2} - c\Big( |\zeta|^{2} + 4(v\cdot\zeta) + 4\frac{|v\cdot\zeta|^{2}}{|\zeta|^{2}} \Big)  \\
		&\leq 2\theta(v\cdot \zeta) + 2\theta(1 + |\zeta|^{2}) + \theta|\zeta|^{2} - c|\zeta|^{2} - c\Big( |\zeta|^{2} + 4(v\cdot\zeta) + 4\frac{|v\cdot\zeta|^{2}}{|\zeta|^{2}} \Big).  \\
		\end{split}
		\end{equation}
		By choosing $\theta < \min\{ (\sqrt{20}-4)c, \frac{2}{3}c \} = (\sqrt{20}-4)c$ and considering quadratic form of $\zeta$ and $\frac{v\cdot\zeta}{|\zeta|}$,
		\begin{equation} \notag
		\begin{split}
		(\ref{P comp}) -2\theta &\leq - (2c-3\theta)|\zeta|^{2} + (2\theta-4c) (v\cdot\zeta) - 4c \frac{|v\cdot\zeta|^{2}}{|\zeta|^{2}} < 0,  \\
		\end{split}
		\end{equation}
		because discriminant of RHS satisfies 
		\[
		4(\theta+4c)^{2} - 80c^{2} < 0,
		\]
		when $\theta < (\sqrt{20}-4)c$. (\ref{ws nega bar}) is obtained by replacing $\theta$ into $\varpi s$.  
	\end{proof}	
	
	\begin{lemma}[Specular reflection of $\Gamma(f,f)$ and $\nu(f)$] \label{lem_specular G}
		If $f$ satisfies the specular reflection boundary condition (\ref{specular}), then $\Gamma(f,f)$ also satitsfies the specular reflection boundary condition, i.e.,
		\Be \notag 
		\Gamma_{\text{gain}}(f,f)(t,x,v) = \Gamma_{\text{gain}}(f,f)(t,x,R_{x}v) \quad \text{and}\quad \Gamma_{\text{loss}}(f,f)(t,x,v) = \Gamma_{\text{loss}}(f,f)(t,x,R_{x}v),
		\Ee
		where $R_{x} = I-2n(x)\otimes n(x)$ is specular reflection operator on $x\in\p\O$. Moreover,  $\nu(f)$ also satisfy \eqref{specular}. \\
	\end{lemma}
	\begin{proof}
		Since $R_{x}$ is orthonormal we have $(R_{x}v - R_{x}u)\cdot R_{x}\sigma = (v-u)\cdot\sigma$. Also using specular reflection condition,
		\Be \notag %
		\begin{split}
			&\Gamma_{\text{gain}}(f,f)( t,x,v )  \\
			&= \iint |(v-u)\cdot\sigma| \sqrt{\mu}(u) 
			f(t, x, u + ((v-u)\cdot\sigma)\sigma )  f(t, x, v - ((v-u)\cdot\sigma)\sigma )  d\sigma du  \\
			&= \iint |(R_{x}v - R_{x}u)\cdot R_{x}\sigma| \sqrt{\mu}(R_{x}u)  \\
			&\quad \times f(t, x, R_{x}u + (R_{x}(v-u)\cdot R_{x}\sigma) R_{x}\sigma )  
			f(t, x, R_{x}v - (R_{x}(v-u)\cdot R_{x}\sigma) R_{x}\sigma )  d\sigma du \\
			&= \iint |(R_{x}v - R_{x}u)\cdot R_{x}\sigma| \sqrt{\mu}(R_{x}u)  \\
			&\quad \times f(t, x, R_{x}u + (R_{x}(v-u)\cdot R_{x}\sigma) R_{x}\sigma )  
			f(t, x, R_{x}v - (R_{x}(v-u)\cdot R_{x}\sigma) R_{x}\sigma )  dR_{x}\sigma d R_{x}u \\
			&= \iint |(R_{x}v-u)\cdot\sigma| \sqrt{\mu}(u) 
			f(t, x, u + ((R_{x}v-u)\cdot\sigma)\sigma )  f(t, x, R_{x}v - ((R_{x}v-u)\cdot\sigma)\sigma )  d\sigma du  \\
			&= \Gamma_{\text{gain}}(f,f)( t,x,R_{x}v ).  \\
		\end{split}
		\Ee
		Similarly, for $\Gamma_{\text{loss}}(f,f)$,
		\Be \notag %
		\begin{split}
			\Gamma_{\text{loss}}(f,f)( t,x,v ) 
			&= f(t, x, v) \iint |(v-u)\cdot\sigma| \sqrt{\mu}(u) f(t, x, u)  d\sigma du  \\
			&= f(t, x, R_{x}v) \iint |(R_{x}v-R_{x}u)\cdot R_{x}\sigma| \sqrt{\mu}(u) f(t, x, R_{x}u)  d\sigma du  \\
			&= f(t, x, R_{x}v) \iint |(R_{x}v-R_{x}u)\cdot R_{x}\sigma| \sqrt{\mu}(u) f(t, x, R_{x}u)  d R_{x}\sigma d R_{x}u  \\
			&= f(t, x, R_{x}v) \Gamma_{\text{loss}}(f,f)( t,x, R_{x}v ).   
		\end{split}
		\Ee
	\end{proof}
	
	\hide
	\subsection{Geometric lemmas}  
	In this subsection, we state some geometric lemmas about uniformly convex object $\O$. First, let us consider a plane $S$ such that $S\cap\p\O$ is neither the empty set nor a single point. In this case, we define
	\Be \label{S assump}
	\p\O_{S} := S\cap \p\O.
	\Ee
	Since $\O\subset \R^{3}$ is uniformly convex object, $\p\O_{S}$ is a uniformly convex closed curve in $S$. For an inward unit normal vector $n(x)$ for $x\in \p\O_{S}$, (since our domain is $\overline{\O}^{c}$), we define
	\Be \label{def n_parallel}
	n_{\parallel}(x)  :=  (I-\hat{q}\otimes \hat{q})n(x) = \text{projection of $n(x)$ on $S$}
	\Ee
	where $\hat{q}$ is a unit vector perpendicular to the plane $S$.  \\
	
	\begin{lemma}  
		(Uniform comparability of $n_{\parallel}$)
		Let us consider a plane $S$ which satisfies (\ref{S assump}). Then $| n_{\parallel}(x) |$ is uniformly comparable for all $x\in \p\O\cap S$,  i.e., there exist uniformly positive constants $c, C \gtrsim_{\O} 1$ such that
		\[
		1 \lesssim_{\O} c < \frac{|n_{\parallel}(x)|}{|n_{\parallel}(y)|} \leq C, \quad \forall x,y \in \p\O\cap S.
		\]
	\end{lemma}
	\begin{proof}
		Appendix
	\end{proof}

	We equip the following geometric lemma, which will be used repeatedly.  
	\unhide
	
		\section{Geometric Lemmas}

	\begin{definition} \label{Def normal}
		Let $S$ be a plane in $\R^{3}$ and $\O$ is given as in Definition \ref{def:domain}. Assume 
		\Be \label{S assume}
		\text{ $\p\O\cap S$ is closed curve, i.e., $\p\O\cap S$ is neither an empty set nor a single point.}
		\Ee
		We define projected normal vector,
		\begin{equation} \label{def n}
		\begin{split}
		n_{\parallel}(x) &:= \text{Proj}_{ S}n(x) = \text{projection of $n(x)$ on $S$} \\
		&= (I-\hat{q}\otimes \hat{q})n(x) ,
		\end{split}
		\end{equation}
		where $x\in\p\O$ and $\hat{q}$ is a unit vector orthogonal to the plane $S$.  
		We also parametrize the curve $\p\O \cap S$ as  regularized curve $\mathfrak{r}_{S} : [0, L_{S}) \rightarrow \p\O\cap S$ (regularized means $|\mathfrak{r}_{S}^{\prime}(s)|=1$), where $L_{S}>0$ is the length of $\p\O \cap S$. Note that we do not specify $S$ in the definition $n_{\parallel}(x)$, because it can be understood properly in the context.  \\ 
	\end{definition}
	
	\begin{lemma}[Uniform comparability of $n_{\parallel}$] \label{lem_unif n}
		Let us consider a plane $S$ which satisfies (\ref{S assume})  with a domain $\O$ as in Definition \ref{def:domain}. Then $| n_{\parallel}(x) |$ in \eqref{def n} is uniformly comparable for all $x\in \p\O\cap S$,  i.e., there exist uniformly positive constants $c$ and $C$, which only depend on $\O$,  such that
		\Be 
		c < \frac{|n_{\parallel}(x)|}{|n_{\parallel}(y)|} \leq C, \quad \forall x,y \in \p\O\cap S   \ \text{and } \ \forall  S \text{ which satisfies } \eqref{S assume}, \label{unif n}
		\Ee
		and 
		\Be 
		c  < \frac{|\mathfrak{r}^{\prime\prime}_{S}(x)|}{|\mathfrak{r}^{\prime\prime}_{S}(y)|} \leq C , \quad \forall x,y \in \p\O\cap S \ \text{and } \ \forall  S \text{ which satisfies } \eqref{S assume}.\label{unif kappa}
		\Ee
	 \\
		(For example, it is obvious that
		\[
			|\mathfrak{r}^{\prime\prime}_{S}(x)| = |\mathfrak{r}^{\prime\prime}_{S}(y)| \ \text{and} \ |n_{\parallel}(x)| = |n_{\parallel}(y)|\quad \text{for all $x,y \in \p\O \cap S$}
		\]
		for any $S$ that satisfies \eqref{S assume}, if $\mathcal{O}$ is a sphere.)
	\end{lemma}	
	\begin{proof}  
		We parametrize the curve $\p\O \cap S$ as  regularized curve  ($|\mathfrak{r}_{S}^{\prime}(s)|=1$), $\mathfrak{r}_{S} : [0, L_{S}) \rightarrow \p\O\cap S$, where $L_{S}>0$ is the length of $\p\O \cap S$. \\
		
		Considering normal curvature $n(x)\cdot \mathfrak{r}^{\prime\prime}_{S}(s)$,
		\Be \label{n curvature}
		\begin{split}
			|n_{\parallel}(x)| 
			&= \Big| n(x)\cdot \frac{\mathfrak{r}_{S}^{\prime\prime}(s)}{|\mathfrak{r}_{S}^{\prime\prime}(s)|} \Big| 
			= |n(x)\cdot \mathfrak{r}^{\prime\prime}_{S}| \times  \frac{1}{|\mathfrak{r}^{\prime\prime}_{S}(s)|}  
			,\quad x=\mathfrak{r}_{S}(s) \in \p\O\cap S.  \\
		\end{split}
		\Ee
		
		\noindent\textbf{Step 1} First, let us recall some standard definitions and properties of differential geometry (\cite{dC}). We can consider Gauss map $N: \p\O \mapsto \mathbb{S}^{2}$, $N(x) = n(x) = \frac{\nabla\xi(x)}{|\nabla\xi(x)|}$. For Gauss map $N$, we define differential $dN_{x}:T_{x}(\p\O) \mapsto T_{x}(\p\O)$ as
		\Be \notag
			dN_{x}(\alpha^{\prime}(0)) =  \frac{d}{dt}N(\alpha(t))\bigg\vert_{t=0},
		\Ee
		where $\alpha$ is a curve on $\p\O$ such that $|\alpha^{\prime}(0)|=1$ and $\alpha(0)=x\in\p\O$. For self-adjoint linear map $dN_{x}$, there exists an orthonormal basis $\hat{x}_{1}, \hat{x}_{2}\in T_{x}(\p\O)$ and $k_{1}, k_2\in\R$ such that
		\Be \notag 
			dN_{x}(\hat{x}_{1}) = -k_{1}\hat{x}_{1},\quad dN_{x}(\hat{x}_{2}) = -k_{2}\hat{x}_{2},\quad k_{1}\geq k_{2}.
		\Ee
		$\hat{x}_{1}$ and $\hat{x}_{2}$ are called principal directions at $x\in\p\O$ and $k_{1}$ (resp, $k_{2}$) is maximum  (resp, minimum) normal curvature. (See page 140, 144 of \cite{dC} for above definitions and properties.)     \\
			Now, let us fix $x\in\p\O$ and tangential plane $T_{x}(\p\O)$.  We consider a local parametrization $\Phi: B(0,\varepsilon) \rightarrow B(x,\varepsilon)$  
		\Be \label{def Phi}
			\Phi(y) = x + y_{1}\hat{x}_{1} + y_{2}\hat{x}_{2} + y_{3}n(x),\quad y = (y_{1}, y_{2}, y_{3}),
		\Ee		
		so that
		\Be \notag %
		\nabla\Phi = \begin{pmatrix}
			& & \\
			\hat{x}_{1} & \hat{x}_{2} & n(x)  \\
			& & \\
		\end{pmatrix},\quad \text{an orthonormal matrix}. 
		\Ee
		We can locally define height function $h(y)$ in $B(0,\varepsilon)$ 
		\Be \notag 
			\xi(\Phi(y)) = h(y_{1}, y_{2}) - y_{3},
		\Ee
		so that
		\[
			\{ x : \xi(x)=0 \} = \Phi\{ y : h(y_{1}, y_{2}) = y_{3}  \}  \subset \p\O,\quad \text{locally}. \\
		\]
		\\
		If we choose $\zeta \in T_{x}(\p\O)$ ,there exists $w=\begin{pmatrix}
		w_{1} \\ w_{2} \\ 0 
		\end{pmatrix}$ such that $\zeta = \nabla\Phi w$ and 
		\Be \label{quadratic h}
			\nabla^{2}(h(y_{1}, y_{2})-y_{3}) =\nabla\big( \nabla \xi^{T} \nabla\Phi \big) = \nabla \Phi^{T} \nabla^{2}\xi \nabla\Phi.
		\Ee
		From \eqref{quadratic h}, we obtain 
		\Be \label{quadra}
		\begin{split}
			\zeta\cdot \nabla^{2}\xi(x) \zeta = w_{h}\cdot  \nabla^{2}h(0,0) w_{h}, 
		\end{split}
		\Ee
		where $w_{h}=(w_{1}, w_{2})$ and $\nabla_{h}^{2}h$ is upper left $2\times 2$ matrix of $\nabla^{2}h$. (\ref{quadra}) is quadratic form of height function based on $T_{x}(\p\O)$ when $|\zeta| = |w| = 1$. It is well-known that (see page 173 of \cite{dC} for example) the quadratic form of a height function (RHS of \eqref{quadratic h}) is normal curvature at $x\in\p\O$ in the direction $\zeta\in\S^{2}$ i.e.,
		\Be\notag %
			|w_{h}\cdot  \nabla^{2}h(0,0) w_{h}|  = | n(x)\cdot \alpha^{\prime\prime}(0) | 
		\Ee
		for any (parametrized) curve $\alpha$ such that $\alpha(0)=x \in\p\O$ and $\alpha^{\prime}(0)=\zeta$ with $|\alpha^{\prime}(0)| = |\zeta|=1$. Hence, combining with	(\ref{convex_xi}) and \eqref{quadra}, we obtain
		\Be \label{curv uniformx}
		\theta_{\O} \leq |\zeta\cdot \nabla^{2}\xi(x) \zeta| = | n(x)\cdot \alpha^{\prime\prime}(0) |  \lesssim_{\O} \|\xi\|_{C^{2}}, \ \text{where} \  \alpha(0)=x, \ \alpha^{\prime}(0)=\zeta, \ |\alpha^{\prime}(0)| = |\zeta| = 1.
		\Ee
		From \eqref{curv uniformx}, for any point $x\in\p\O$ and any direction $\zeta\in T_{x}(\p\O)$, corresponding normal curvature is uniformly bounded from below and above. \\
		
		\noindent\textbf{Step 2} Now, to consider all possible planes $S$ that satisfy (\ref{S assume}), we first parametrize such planes. We choose a fixed point $p \in \O$ and  a unit vector $\ell \in \S^{2}$. Then, let us use $S_{p,\ell}$ to denote the plane
		\[
		S_{p,\ell} := \{ x\in\R^{3} : (x-p)\cdot \ell = 0 \},
		\]
		the plane which is perpendicular to $\ell$ and passes $p\in\O$. Moreover, by uniform convexity, for fixed $p\in\O$ and $\ell\in\S^{2}$, there exists $a_{p,\ell} < b_{p,\ell}$ such that plane $S_{p+r\ell,\ell}$ also satisfies (\ref{S assume}) for all $a_{p,\ell} < r < b_{p,\ell}$, i.e., plane perpendicular to $\ell$ and passes $p+r\ell$. Since $\ell$ is unit,
		\Be \label{b-a}
			|b_{p,\ell}- a_{p, \ell}| \leq \max_{x,y\in\p\O} |x-y| \lesssim_{\O} 1 \quad \forall \ell\in\S^{2}.
		\Ee
		For fixed $p, \ell$, we split
		\[
		(a_{p,\ell}, b_{p,\ell}) = (a_{p,\ell}, a_{p,\ell}+\varepsilon_{p,\ell}] \cup [a_{p,\ell}+\varepsilon_{p,\ell}, b_{p,\ell}-\varepsilon_{p,\ell}] \cup [b_{p,\ell}-\varepsilon_{p,\ell}, b_{p,\ell}). 
		\]
		There exists $\varepsilon_{p,\ell} \ll 1$ such that for $r\in (a_{p,\ell}, a_{p,\ell}+\varepsilon_{p,\ell}]$, we can parametrize 
		\Be \notag 
			\{ \p\O\cap S_{p+r\ell, \ell} \ : \ r\in [a_{p,\ell}, a_{p,\ell}+\varepsilon_{p,\ell}] \}
		\Ee
		as a graph $\eta_{p,\ell}$ over tangent plane $S_{p+a_{p,\ell}\ell, \ell}$.	Similar as local parametrization of Step 1 (rotation and translation similar as \eqref{def Phi}), we consider local height function $\eta_{p,\ell}$ such that $ \vec{x} \mapsto  (\vec{x}, \eta_{p,\ell} (\vec{x}) )$ with $\vec{x} = (x,y)^{T}$ is an orthogonal parametrization near $S_{p+a_{p,\ell}\ell, \ell}\cap \p\O$. It is obvious that $\eta_{p,\ell} (0) = \nabla\eta_{p,\ell} (0)=0$ and its principal directions are $(1,0)$ and $(0,1)$ (for maximum and minimum curuvature, respectively). 
		By the Taylor expansion, 
		\Be \notag %
		\eta_{p,\ell}(\vec{x}) = \frac{1}{2}\vec{x}\cdot \nabla^{2}\eta_{p,\ell}|_{(0,0)} \vec{x} + O_{\O}(|\vec{x}|^{3}) = C_{p,\ell,1}x^{2} + C_{p,\ell,2}y^{2} + O_{\O}(|\vec{x}|^{3}),\quad C_{p,\ell,1} > C_{p,\ell,2} > 0.
		\Ee
		Here, note that two princiapl directions are eigenvectors of $\nabla^{2}\eta_{p,\ell}$ in fact. Moreover, $\nabla^{2}\eta_{p,\ell}$ is  self-adjoint and positivie definite by \eqref{convex_xi} and \eqref{quadra}, so both two eigenvalues have positive sign. Thus, we obtain the RHS above. In other words, $\p\O \cap S_{p+r\ell, \ell}$ is a small perturbation of ellipse for any $r\in (a_{p,\ell}, a_{p,\ell}+\varepsilon_{p,\ell}]$, i.e., 
		\Be \label{ellipse}
			C_{p,\ell,1}x^{2} + C_{p,\ell,2}y^{2} + O_{\O}(|\vec{x}|^{3}) = r - a_{p,\ell}.
		\Ee
  		We can chosen sufficienlty small $\varepsilon_{p,\ell} \ll 1$ and note that $|\vec{x}|\rightarrow 0$ as $\varepsilon_{p,\ell} \rightarrow 0$.  \\

		The curve and corresponding curvature that satisfies \eqref{ellipse} is sufficiently smooth and therefore, the ratio
		\Be \notag %
			\frac{\min_{s}|\mathfrak{r}^{\prime\prime}_{S_{p+ r\ell}}(s)|}{ \max_{s}|\mathfrak{r}^{\prime\prime}_{S_{p+r\ell}}(s)|} 
		\Ee
		is continuous in $r\in (a_{p,\ell}, a_{p,\ell}+\varepsilon_{p,\ell}]$. Moreover,  
		\Be \notag %
			1\geq \frac{\min_{s}|\mathfrak{r}^{\prime\prime}_{S_{p+ r\ell}}(s)|}{ \max_{s}|\mathfrak{r}^{\prime\prime}_{S_{p+r\ell}}(s)|} \rightarrow \Big(\frac{C_{p,\ell,2}}{C_{p,\ell,1}}\Big)^{\frac{3}{2}}  \ \ \text{as} \ \ r\rightarrow a_{p,\ell},
		\Ee
		because the maximum and minimum curvatures of the ellipse  $C_{p,\ell,1}x^{2} + C_{p,\ell,2}y^{2} = r - a_{p,\ell}$ are given by $\frac{C_{p,\ell,1}}{\sqrt{ (r - a_{p,\ell})C_{p,\ell,2} }}$ and $\frac{C_{p,\ell,2}}{\sqrt{ (r - a_{p,\ell})C_{p,\ell,1} }}$, respectively. (Note that it is easy to check the maximum and minimum curvatures of the ellipse  $\frac{x^{2}}{a^{2}} + \frac{y^{2}}{b^{2}} = 1$ are given by $\frac{a}{b^{2}}$ and $\frac{b}{a^{2}}$, respectively when $a\geq b$.) Now by continuity argument, we can choose $\varepsilon_{p, \ell} \ll 1$ such that
		\Be \label{near a-p-ell}
			\frac{1}{2} \Big(\frac{C_{p,\ell,1}}{C_{p,\ell,2}}\Big)^{3} \leq \frac{\min_{s}|\mathfrak{r}^{\prime\prime}_{S_{p+ r\ell}}(s)|}{ \max_{s}|\mathfrak{r}^{\prime\prime}_{S_{p+r\ell}}(s)|} \leq 1,\quad \forall r\in (a_{p,\ell}, a_{p,\ell}+\varepsilon_{p,\ell}].
		\Ee
		
		\hide
		 {\color{red} Curvature of a curve is reprensented by second derivatives of regularized parametrization, so we obtain    [You need some argument to get it from the above expression of $\eta_{p,\ell} (\vec{x})$.] 
		\Be 
		\frac{1}{2}\frac{C_{p,\ell,2}}{C_{p,\ell,1}} \leq \frac{C_{p,\ell,2}}{C_{p,\ell,1}} - O(|\vec{x}|) \lesssim_{\O} \frac{|\mathfrak{r}^{\prime\prime}_{S_{p+ r\ell}}(s_{1})|}{|\mathfrak{r}^{\prime\prime}_{S_{p+r\ell}}(s_{2})|} \lesssim_{\O} \frac{C_{p,\ell,1}}{C_{p,\ell,2}} + O(|\vec{x}|) \leq 2\frac{C_{p,\ell,2}}{C_{p,\ell,1}}   ,\quad \forall s_{1}, s_{2} \in [0,L_{S_{p+r\ell}}),\  r\in (a_{p,\ell}, a_{p,\ell}+\varepsilon_{p,\ell}],
		\Ee
			}	  
		for some constants $C_{p,\ell,1}$ and $C_{p,\ell,2}$.   \\
		\unhide
		
		\noindent Similarly,  we repeat above process near $\p\O\cap S_{p+b_{p,\ell}\ell}$ using height function, say $\tilde{\eta}_{p,\ell}(\vec{x})$ to get
		\Be  \notag %
			\tilde{\eta}_{p,\ell}(\vec{x}) = \frac{1}{2}\vec{x}\cdot \nabla^{2}\tilde{\eta}_{p,\ell}|_{(0,0)} \vec{x} + O_{\O}(|\vec{x}|^{3}) = C^{\prime}_{p,\ell,1}x^{2} + C^{\prime}_{p,\ell,2}y^{2} + O_{\O}(|\vec{x}|^{3}),\quad C^{\prime}_{p,\ell,1} > C^{\prime}_{p,\ell,2} > 0.
		\Ee
		By same argument, we can choose $\varepsilon_{p,\ell} \ll 1$ (even smaller if necessary)
		\Be \label{near b-p-ell}
			\frac{1}{2} \Big(\frac{C^{\prime}_{p,\ell,1}}{C^{\prime}_{p,\ell,2}}\Big)^{3} \leq \frac{\min_{s}|\mathfrak{r}^{\prime\prime}_{S_{p+ r\ell}}(s)|}{ \max_{s}|\mathfrak{r}^{\prime\prime}_{S_{p+r\ell}}(s)|} \leq 1,\quad \forall r\in  [b_{p,\ell}-\varepsilon_{p,\ell}, b_{p,\ell}). 
		\Ee
		
		\hide
		\Be 
		\frac{1}{2}\frac{C^{\prime}_{p,\ell,2}}{C^{\prime}_{p,\ell,1}} \leq \frac{C^{\prime}_{p,\ell,2}}{C^{\prime}_{p,\ell,1}} - O(|\vec{x}|) \lesssim_{\O} \frac{|\mathfrak{r}^{\prime\prime}_{S_{p+ r\ell}}(s_{1})|}{|\mathfrak{r}^{\prime\prime}_{S_{p+r\ell}}(s_{2})|} \lesssim_{\O} \frac{C^{\prime}_{p,\ell,1}}{C^{\prime}_{p,\ell,2}} + O(|\vec{x}|) \leq 2\frac{C^{\prime}_{p,\ell,2}}{C^{\prime}_{p,\ell,1}}   ,\quad \forall s_{1}, s_{2} \in [0,L_{S_{p+r\ell}}),\  r\in (b_{p,\ell} - \varepsilon_{p,\ell}, b_{p,\ell}],
		\Ee
		for some constants $C^{\prime}_{p,\ell,1}$ and $C^{\prime}_{p,\ell,2}$. \\
		\unhide
		
		\noindent For $r\in [a_{p,\ell}+\varepsilon_{p,\ell}, b_{p,\ell}-\varepsilon_{p,\ell}] $, we obtain
		\Be \notag %
			0 < c_{p,\ell,r} \leq \frac{\min_{s}|\mathfrak{r}^{\prime\prime}_{S_{p+ r\ell}}(s)|}{ \max_{s}|\mathfrak{r}^{\prime\prime}_{S_{p+r\ell}}(s)|} \leq 1,
		\Ee
		where strict positivity of $c_{p,\ell,r} > 0$ comes from \eqref{curv uniformx} and $\max_{s}|\mathfrak{r}^{\prime\prime}_{S_{p+r\ell}}(s)| < \infty$. Moreover, this is continuous function in $r\in[a_{p,\ell}+\varepsilon_{p,\ell}, b_{p,\ell}-\varepsilon_{p,\ell}]$. Combining with compactness of $[a_{p,\ell}+\varepsilon_{p,\ell}, b_{p,\ell}-\varepsilon_{p,\ell}]$ (by \eqref{b-a}), we derive some $r$- independent $\mathcal{D}_{p,\ell} > 0$ such that 
		\Be \label{middle}
		D_{p,\ell} \lesssim_{\O} \frac{\min_{s}|\mathfrak{r}^{\prime\prime}_{S_{p+ r\ell}}(s)|}{\max_{s}|\mathfrak{r}^{\prime\prime}_{S_{p+r\ell}}(s)|} \leq 1 ,\quad r\in[a_{p,\ell}+\varepsilon_{p,\ell}, b_{p,\ell}-\varepsilon_{p,\ell}].
		\Ee
		Combining (\ref{near a-p-ell}), (\ref{near b-p-ell}), and (\ref{middle}), we have $\mathcal{C}_{p,\ell} > 0$ such that
		\[
		\mathcal{C}_{p,\ell} \lesssim_{\O} \frac{\min_{s}|\mathfrak{r}^{\prime\prime}_{S_{p+ r\ell}}(s)|}{\max_{s}|\mathfrak{r}^{\prime\prime}_{S_{p+r\ell}}(s)|} \leq 1 \quad \text{for all}\quad r \in (a_{p,\ell}, b_{p,\ell}).
		\]
		Now, we consider all possible $\ell\in \S^{2}$ and  can use similar continuity argument as (\ref{middle}). From compactness of $\S^{2}$, we can drop $\ell$ independence from lower bound and obtain
		\Be \label{unif k ratio}
		\mathcal{C}^{\prime}_{p} \leq \frac{\min_{s}|\mathfrak{r}^{\prime\prime}_{S_{p+ r\ell}}(s)|}{\max_{s}|\mathfrak{r}^{\prime\prime}_{S_{p+r\ell}}(s)|} \leq 1,
		\Ee
		which gives \eqref{unif kappa}. Also, combining \eqref{unif k ratio},  \eqref{n curvature} and \eqref{curv uniformx}, we obtain \eqref{unif n}.
	\end{proof}
	 
	 \hide
	Now we study various properties of $\dot{\X} {\color{green}(\tau)}\cdot \nabla\xi {\color{green}(\X (\tau), v)}$ and $\dot{\V}{\color{green}(\tau)}\cdot \nabla\xi {\color{green}(\X (\tau), v)}$. Before we state lemma, let us briefly write simple facts{\color{red} It is too sloppy to state \eqref{2D x} and \eqref{2D v} without any proof or argument.} : when \eqref{assume_x} and \eqref{assume_x2} hold{\color{red}[you need to recall what is $\X$,  $\tilde{x}, \tilde{v}$]},
	\Be  
	|\dot{\X}| = | \dot{\X} \cdot \widehat{n_{\parallel}}(\xb(\X(\tau_{\pm}), v))|  = \max_{\tau_{-}\leq \tau\leq \tau_{+}}
	|\dot{\X} \cdot \widehat{n_{\parallel}}(\xb(\X(\tau), v))| {\color{red}\tau_+ \text{ is not defined yet?}}
	\Ee
	Similarly, when \eqref{assume_v} and \eqref{assume_v2} hold,{\color{red}[you need to recall what is $\V$, $\tilde{x}, \tilde{v}$]}
	\Be  
	|\dot{\V}(\tau)| = |\dot{\V}(\tau_{\pm}) \cdot \widehat{n_{\parallel}}( \xb(x, \V(\tau_{\pm})))| = \max_{\tau_{-}\leq \tau\leq \tau_{+}} |\dot{\V}(\tau) \cdot \widehat{n_{\parallel}}( \xb(x, \V(\tau)))|
	\Ee
	where $|\dot{\V}(\tau)| = \theta|v+\zeta|$ for all $\tau$ by by \eqref{dotv norm}.	Both \eqref{2D x} and \eqref{2D v} are trivial by convexity of $\O$ and $\p\O\cap S$. See Figure \ref{fig1} and Figure \ref{fig2}. \\
	\unhide
	
	\begin{lemma}\label{lem_mono}
		Suppose the domain $\O$ is given as in Definition \ref{def:domain} and \eqref{convex_xi}.  
		
		\noindent (i) Let two distinct points $x, \bar{x} \in \O$, and a nonzero velocity $v\neq 0$ are given. We assume \eqref{assume_x} and \eqref{assume_x2}, and recall $\X(\tau)$ in Definition \ref{def_para}. 
		There exists $\tau_{0}(x, \bar{x}, v) \in (\tau_{-}(x, \bar{x}, v), \tau_{+}(x, \bar{x}, v)) $ such that
		\Be \label{tau_0}
		\begin{cases}
			\dot{\X}\cdot \nabla \xi(\xb(\X(\tau),v)) > 0 \quad \text{for}\quad \tau_{-}(x, \bar{x}, v) \leq \tau < \tau_{0}(x, \bar{x}, v) , \\
			\dot{\X}\cdot \nabla \xi(\xb(\X(\tau),v)) = 0 \quad \text{for}\quad \tau = \tau_{0}(x, \bar{x}, v) , \\
			\dot{\X}\cdot \nabla \xi(\xb(\X(\tau),v)) < 0 \quad \text{for}\quad \tau_{0}(x, \bar{x}, v) < \tau \leq \tau_{+}(x, \bar{x}, v),
		\end{cases}\Ee
		where $\tau_{\pm}(x, \bx, v)$ is defined in \eqref{tau_pm}. Moreover, there exists $C_{\O} \gtrsim 1$ such that 
		\Be \label{mono Cst 1}
		|\dot{\X}\cdot \nabla\xi(\xb(\X(\tau),v))| \leq C_{\O} |\dot{\X}\cdot \nabla\xi(\xb(\X(\tau_{-}),v))|, \quad \tau_{-} \leq \tau \leq \tau_{0},
		\Ee
		\Be \label{mono Cst 2}
		|\dot{\X}\cdot \nabla\xi(\xb(\X(\tau),v))| \leq C_{\O} |\dot{\X}\cdot \nabla\xi(\xb(\X(\tau_{+}),v))|, \quad \tau_{0} \leq \tau \leq \tau_{+},  \\
		\Ee
		and
		\Be \label{mono Cst 3}
		|v \cdot \nabla\xi(\xb(\X(\tau),v))| \leq C_{\O} |v\cdot \nabla\xi(\xb(\X(\tau_{0}),v))|, \quad \tau_{-} \leq \tau \leq \tau_{+}.  \\
		\Ee

		\noindent (ii) Let $x  \in  \O$ and $v, \bar v, \zeta \in \R^3$ are given. We assume \eqref{assume_v} and \eqref{assume_v2}, and recall $\V(\tau)$ in Definition \ref{def_para}. There exist{\color{red}s} $\tau_{0}(x, v, \bar{v},\zeta) \in (\tau_{-}(x, v, \bar{v},\zeta), \tau_{+}(x, v, \bar{v},\zeta))$ such that 
		\Be \label{tau_0_v}
		\begin{cases}
			\dot{\V}(\tau)\cdot \nabla \xi(\xb(x,\V(\tau))) <0 \quad \text{for}\quad \tau_{-}(x, v, \bar{v},\zeta) \leq \tau < \tau_{0}(x, v, \bar{v},\zeta),  \\
			\dot{\V}(\tau)\cdot \nabla \xi(\xb(x ,\V(\tau))) = 0 \quad \text{for}\quad \tau = \tau_{0}(x, v, \bar{v},\zeta), \\
			\dot{\V}(\tau)\cdot \nabla \xi(\xb(x ,\V(\tau))) > 0 \quad \text{for}\quad \tau_{0}(x, v, \bar{v}, \zeta) < \tau \leq \tau_{+}(x, v, \bar{v}, \zeta).   
		\end{cases}\Ee
		Here, $\tau_{\pm}(x, v, \bv, \zeta)$ is defined in Definition \ref{def_para}. Moreover, there exists $C_{\O} \gtrsim 1$ such that 
		\Be \label{mono Cst 1 v}
		|\dot{\V}(\tau) \cdot \nabla\xi(\xb(x, \V(\tau)))| \leq C_{\O} |\dot{\V}(\tau_-)\cdot \nabla\xi(\xb(x, \V(\tau_{-})))|, \quad \tau_{-} \leq \tau \leq \tau_{0},
		\Ee
		\Be \label{mono Cst 2 v}
		|\dot{\V}(\tau) \cdot \nabla\xi(\xb(x, \V(\tau)))| \leq C_{\O} |\dot{\V}(\tau_+)\cdot \nabla\xi(\xb(x, \V(\tau_{+})))|, \quad \tau_{0} \leq  \tau \leq \tau_{+}, \\
		\Ee
		and
		\Be \label{mono Cst 3 v}
		|{\V}(\tau) \cdot \nabla\xi(\xb(x, \V(\tau)))| \leq C_{\O} |{\V}(\tau_{0})\cdot \nabla\xi(\xb(x, \V(\tau_{0})))|, \quad \tau_{-} \leq \tau \leq \tau_{+}.  \\
		\Ee
	\end{lemma}
	\begin{proof} 
		(i) First we prove \eqref{tau_0}. From \eqref{ODE:dot_x_n}, \eqref{ODE:dot_v_n}, and \eqref{convex_xi}, we derive that 
		\Be \label{ODE1:dot_x_n}
		\begin{split}
			& \frac{d}{d\tau}\big[ (\dot{\X} \cdot \nabla \xi (\xb(\X(\tau), v) )) 
			e^{
				\int^\tau_{\tau_-} 
				\frac{(v\cdot  \nabla^2 \xi (\xb(\X(s), v) )  \cdot \dot{\X} )}{ \nabla \xi (\xb(\X(s), v) ) \cdot v } \dd s 
			}
			\big] \\
			& \ \  = 	-\big(-\dot{\X} \cdot   \nabla^2 \xi (\xb(\X(\tau), v)) \cdot \dot{\X} \big)
			e^{
				\int^\tau_{\tau_-} 
				\frac{(v\cdot  \nabla^2 \xi (\xb(\X(s), v) )  \cdot \dot{\X} )}{\nabla \xi (\xb(\X(s), v) ) \cdot v } \dd s 
			}< 0. \\
		\end{split}
		\Ee
		Since $\xb(\X(\tau),v)$ is well-defined for $\tau_{-} \leq \tau \leq \tau_{+}$ on $\p\O$, we have 
		\Be \label{sign dotx}
			\dot{\X}\cdot\nabla\xi(\xb(\X(\tau_{-}),v)) > 0 \quad \text{and} \quad \dot{\X}\cdot\nabla\xi(\xb(\X(\tau_{+}),v)) < 0.
		\Ee
		Since $e^{
			\int^\tau_{\tau_-} 
			\frac{(v\cdot  \nabla^2 \xi (\xb(\X(s), v) )  \cdot \dot{\X})}{ \nabla \xi (\xb(\X(s), v) ) \cdot v } \dd s 
		} > 0$ always, combining with \eqref{ODE1:dot_x_n} and \eqref{sign dotx}, there exists a unique $\tau_{0}(x, \bx, v)\in (\tau_{-}(x, \bar{x}, v), \tau_{+}(x, \bar{x}, v)) $ such that
		\[
			(\dot{\X} \cdot \nabla \xi (\xb(\X(\tau_0), v) )) 
			e^{
				\int^{\tau_0}_{\tau_-} 
				\frac{(v\cdot  \nabla^2 \xi (\xb(\X(s), v) )  \cdot \dot{\X})}{ \nabla \xi (\xb(\X(s), v) ) \cdot v } \dd s 
			} = 0 \ \ \Leftrightarrow \ \ \dot{\X} \cdot \nabla \xi (\xb(\X(\tau_0), v) )=0.
		\]
		This proves \eqref{tau_0}. To prove \eqref{mono Cst 1},  first note that $\dot{\X}$ and $\widehat{n_{\parallel}}(\xb(\X(\tau_{\pm}), v))$ are parallel or antiparallel to each other because $\dot{\X}$ and $\widehat{n_{\parallel}}(\xb(\X(\tau_{\pm}), v))$ belong to the plan $S_{(x, \bar x, v)}$, and $\dot{\X}\perp v$ and $\widehat{n_{\parallel}}(\xb(\X(\tau_{\pm}), v)) \perp v$ by \eqref{perp short} and \eqref{tau_pm}, respectively, i.e.,
		\Be \label{2D x}
		|\dot{\X}| = | \dot{\X} \cdot \widehat{n_{\parallel}}(\xb(\X(\tau_{\pm}), v))|  = \max_{\tau_{-}\leq \tau\leq \tau_{+}}
		|\dot{\X} \cdot \widehat{n_{\parallel}}(\xb(\X(\tau), v))|.
		\Ee
		Now, since $1\lesssim_{\O} |\nabla\xi| \lesssim_{\O} 1$, using \eqref{2D x} and Lemma \ref{lem_unif n},  
		\Be \label{dot x opt}
		\begin{split}
		|\dot{\X}\cdot \nabla\xi(\xb(\X(\tau),v))| 
		& \lesssim 	|\dot{\X}\cdot n(\xb(\X(\tau),v))|
		\\
		&\lesssim |\dot{\X}\cdot n_{\parallel}(\xb(\X(\tau),v))| = |n_{\parallel}(\xb(\X(\tau),v))| |\dot{\X}\cdot \widehat{{n}_{\parallel}}(\xb(\X(\tau),v))| \\
			&\leq C_{\O} |n_{\parallel}(\xb(\X(\tau_{-}),v))| |\dot{\X}\cdot \widehat{{n}_{\parallel}}(\xb(\X(\tau_{-}),v))|
			 \\
			&= C_{\O} |\dot{\X}\cdot n_{\parallel}(\xb(\X(\tau_{-}),v))| 
			= C_{\O}  |\dot{\X}\cdot n(\xb(\X(\tau_{-}),v))|   \\
			&\lesssim C_{\O}  |\dot{\X}\cdot \nabla\xi(\xb(\X(\tau_{-}),v))|.   \\
		\end{split}
		\Ee
		Estimate (\ref{mono Cst 2}) is obtained similarly.  \\
		Now, let us prove \eqref{mono Cst 3}. $v$ and $n_{\parallel}(\xb(\X(\tau_0), v))$ are parallel or antiparallel to each other because $\dot{\X}\perp v$ (by \eqref{perp}) and \eqref{tau_0}, i.e.,
		\Be \label{2D x 0}
			|v \cdot n_{\parallel}(\xb(\X(\tau_0), v))| = |v||n_{\parallel}(\xb(\X(\tau_0), v))|.
		\Ee
		Therefore, using \eqref{2D x 0} and Lemma \ref{lem_unif n},
		\begin{equation*}
		\begin{split}
			|v \cdot \nabla\xi(\xb(\X(\tau_0),v))| &\gtrsim |v \cdot n_{\parallel}(\xb(\X(\tau_0), v))| 
			= |v||n_{\parallel}(\xb(\X(\tau_0), v))|  \\
			&\gtrsim_{\O} |v\cdot \widehat{{n}_{\parallel}}(\xb(\X(\tau), v))||n_{\parallel}(\xb(\X(\tau), v))|   \\
			&= |v \cdot n_{\parallel}(\xb(\X(\tau), v))| \gtrsim |v \cdot \nabla\xi(\xb(\X(\tau),v)) | .
		\end{split}
		\end{equation*}
		
		\noindent (ii) To prove \eqref{tau_0_v},  we derive that
		\Be \notag 
		\begin{split}
			&\frac{d}{d\tau}\big[ (\dot{\V} (\tau)  \cdot \nabla \xi (\xb(x, \V(\tau))) )
			e^{\int^\tau_{\tau_-}
				\frac{-\tb(x, \V(s))
					(
					-\V(s)\cdot  \nabla^2 \xi (\xb(x, \V(s))) \cdot \dot{\V} (s) )
				}{ \nabla \xi (\xb(x, \V(s))) \cdot \V(s) }
				\dd s	}
			\big]\\
			& \ \
			= 
			\Big[ \tb(x, \V(\tau))
			\big(- \dot{\V} (\tau) \cdot   \nabla^2 \xi (\xb(x, \V(\tau))) \cdot \dot{\V}(\tau)\big)
			+ \ddot{\V}(\tau)\cdot\nabla\xi(\xb(x, \V(\tau)))
			\Big] \\
			&\quad\quad  \times 
			e^{\int^\tau_{\tau_-}
				\frac{-\tb(x, \V(s))
					(
					- \V(s)\cdot  \nabla^2 \xi (\xb(x, \V(s))) \cdot \dot{\V} (s) )
				}{ \nabla \xi (\xb(x, \V(s))) \cdot \V(s) }
				\dd s	} > 0,
		\end{split}
		\Ee
		where we have used the fact: $\ddot{\V}(\tau)\cdot\nabla\xi(\xb(x, v(\tau))) = -\theta^{2} \V(\tau)\cdot\nabla\xi(\xb(x, \V(\tau))) \geq 0$, by \eqref{perp} and $\V(\tau)\cdot\nabla\xi(\xb(x, \V(\tau))) \leq  0$. Similar as \eqref{sign dotx}, we have  
		\Be \notag 
		\dot{\V}(\tau_{-})\cdot\nabla\xi(\xb(\X(\tau_{-}),v)) < 0 \quad \text{and} \quad \dot{\V}(\tau_{+})\cdot\nabla\xi(\xb(\X(\tau_{+}),v)) > 0.
		\Ee
		Now, \eqref{tau_0_v} is obtained similar as proof of \eqref{tau_0}, since $\tb(x, \V(\tau)) > 0$ always. To prove \eqref{mono Cst 1 v}, note that $\dot{\V}(\tau_{\pm})$ and $\widehat{n_{\parallel}}(\xb(x, \V(\tau_{\pm})))$ are parallel or antiparallel to each other because  $\dot{\V}(\tau_{\pm})$ and $\widehat{n_{\parallel}}(\xb(x, \V(\tau_{\pm})))$ belong to the plan $S_{(x,v, \bar v, \zeta)}$, and $\dot{\V}(\tau)\perp \V(\tau)$ and $\widehat{n_{\parallel}}(\xb(x, \V(\tau_{\pm}))) \perp \V(\tau_{\pm})$ by \eqref{perp} and \eqref{tau_pm}, respectively, i.e.,
		\Be \label{2D v}
		|\dot{\V}(\tau)| = |\dot{\V}(\tau_{\pm}) \cdot \widehat{n_{\parallel}}( \xb(x, \V(\tau_{\pm})))| = \max_{\tau_{-}\leq \tau\leq \tau_{+}} |\dot{\V}(\tau) \cdot \widehat{n_{\parallel}}( \xb(x, \V(\tau)))|.
		\Ee
		Using \eqref{2D v} and Lemma \ref{lem_unif n},
		\Be \notag %
		\begin{split}
		 |\dot{\V}\cdot \nabla\xi(\xb(x, \V(\tau)))|  
			&\lesssim 	|\dot{\V}\cdot n(\xb(x, \V(\tau)))| \\
			&= |\dot{\V}\cdot n_{\parallel}(\xb(x,\V(\tau)))| = |n_{\parallel}(\xb(x, \V(\tau)))| |\dot{\V}\cdot \widehat{n_{\parallel}}(\xb(x, \V(\tau)))| \\
			&\leq C_{\O} |n_{\parallel}(\xb(x, \V(\tau_{-})))| |\dot{\V}\cdot \widehat{n_{\parallel}}(\xb(x, \V(\tau_{-})))| \\
			&\leq C_{\O}  |\dot{\V}\cdot n(\xb(x, \V(\tau_{-})))| \\
			&\leq C_{\O}  |\dot{\V}\cdot \nabla\xi(\xb(x, \V(\tau_{-})))|.   \\
		\end{split}
		\Ee
		Estimate (\ref{mono Cst 2 v}) is obtained similarly. \\
		Now, let us prove \eqref{mono Cst 3 v}. $\V(\tau_0)$ and  $n_{\parallel}(\xb(x, \V(\tau_0)))$ are parallel to each other because   $\V(\tau_0)$ and $n_{\parallel}(\xb(x, \V(\tau_0)))$ belong to $S_{(x, v, \bar v, \zeta)}$, and $\dot{\V}(\tau)\perp \V(\tau)$ (by \eqref{perp}) and \eqref{tau_0_v}. Therefore,
		\Be \label{2D v 0}
		|\V(\tau_0) \cdot n_{\parallel}(\xb(x, \V(\tau_0)))| = |\V(\tau_0)||n_{\parallel}(\xb(x, \V(\tau_0)))|.
		\Ee
		Therefore, using \eqref{2D v 0} and Lemma \ref{lem_unif n},
		\begin{equation*}
		\begin{split}
		|\V(\tau_0) \cdot \nabla\xi(\xb(x, \V(\tau_0)))|  &\gtrsim |\V(\tau_0) \cdot n_{\parallel}(\xb(x, \V(\tau_0)))| = |\V(\tau_0)||n_{\parallel}(\xb(x, \V(\tau_0)))|  \\
		&\gtrsim_{\O} |\V(\tau)\cdot \widehat{{n}_{\parallel}}(\xb(x, \V(\tau)))||n_{\parallel}(\xb(x, \V(\tau)))|		 \\
		&= |\V(\tau) \cdot n_{\parallel}(\xb(x, \V(\tau)))| \gtrsim  |\V(\tau) \cdot \nabla\xi(\xb(x, \V(\tau)))|.
		\end{split}
		\end{equation*}
	\end{proof}
	
	\begin{lemma} \label{lem_min_tb}
		Recall $\X(\tau) = \X (\tau; x, \bar x, v)$, $\V(\tau) = \V (\tau; x,v, \bar v, \zeta)$ as in Definition \ref{def_para}, and $\tau_{0}(x, \bar x, v)$ and $\tau_0 (x, v, \bar v, \zeta)$ at \eqref{tau_0} and \eqref{tau_0_v}, respectively in Lemma \ref{lem_mono}. \\
		(i) Recall all assumptions in (i) of Lemma \ref{lem_mono}. 
		When $\tau_{-} \leq \tau_{*} \leq \tau_{0}$, we get
		\Be \label{min tb x-}
			\max_{\tau_{-} \leq  s \leq \tau_{*} } \tb(\X(s), v) = \tb(\X(\tau_{-}), v),\quad \min_{\tau_{-} \leq  s \leq \tau_{*} } \tb(\X(s), v) = \tb(\X(\tau_{*}), v).
		\Ee
		By symmetry, when $\tau_{0} \leq \tau_{*} \leq \tau_{+}$, 
		\Be \notag 
			\max_{\tau_{*} \leq  s \leq \tau_{+} } \tb(\X(s), v) = \tb(\X(\tau_{+}), v),\quad  \min_{\tau_{*} \leq  s \leq \tau_{+} } \tb(\X(s), v) = \tb(\X(\tau_{*}), v).
		\Ee
	 \\
		(ii) Recall all assumptions in (ii) of Lemma \ref{lem_mono}. When $\tau_{-} \leq \tau_{*} \leq \tau_{0}$, , we get
		\Be \label{min tb v-}
		\max_{\tau_{-} \leq  s \leq \tau_{*} } \tb(x, \V(s)) = \tb(x, \V(\tau_{-})),\quad \min_{\tau_{-} \leq  s \leq \tau_{*} } \tb(x, \V(s)) = \tb(x, \V(\tau_{*})).
		\Ee
		By symmetry, when $\tau_{0} \leq \tau_{*} \leq \tau_{+}$, 
		\Be \notag 
		\max_{\tau_{*} \leq  s \leq \tau_{+} } \tb(s, \V(s)) = \tb(x, \V(\tau_{+})),\quad \min_{\tau_{*} \leq  s \leq \tau_{+} } \tb(s, \V(s)) = \tb(x, \V(\tau_{*})).
		\Ee
		also holds.   
		Moreover, for all $\tau_{-} \leq \tau \leq \tau_{+}$, 
		\Be \label{tb minmax}
			\max_{\tau_- \leq \tau \leq \tau_+} (|\V(\tau)|\tb(x, \V(\tau)))  \lesssim_{\O} 1 + \min_{\tau_- \leq \tau \leq \tau_+} (|\V(\tau)|\tb(x, \V(\tau))).
		\Ee
	\end{lemma}
	\begin{proof}
		(i) Let $\tau_{-} \leq \tau_{*} \leq \tau_{0}$. From \eqref{nabla_x_bv_b} and \eqref{tau_0}, 
		\[
			\frac{d}{d\tau} \tb(\X(\tau), v) = \frac{\nabla\xi(\xb(\X(\tau), v))\cdot \dot{\X}}{\nabla\xi(\xb(\X(\tau), v))\cdot v} < 0
		\] 
		since $\nabla\xi(\xb(\X(\tau), v))\cdot v < 0$. $\tau_{0} \leq \tau_{*} \leq \tau_{+}$ case is similar. We also note that $\frac{d}{d\tau} \tb(\X(\tau), v)=0$ only at $\tau_{0}$ by \eqref{tau_0}. \\ 
		(ii) We can prove similary as (i) using \eqref{nabla_x_bv_b} and \eqref{tau_0_v}, since
		\[
		\frac{d}{d\tau} \tb(x, \V(\tau)) = -\tb(x, \V(\tau)) \frac{\nabla\xi(\xb(x, \V(\tau)))\cdot \dot{\V}(\tau)}{\nabla\xi(\xb(x, \V(\tau)))\cdot \V(\tau)} < 0.
		\] 
		To prove \eqref{tb minmax}, let us consider $|\V(\tau)| \tb(x, \V(\tau))$. If $\text{dist} (x, \p\O) \leq 1$, then $\max_{\tau} |\V(\tau)|\tb(x, \V(\tau)) \lesssim_{\O} 1$ also.  If $\text{dist} (x, \p\O) \geq 1$, we have 
		\begin{equation*}
		\begin{split}
		\max_{\tau} (|\V(\tau)|\tb(x, \V(\tau))) &\leq \text{dist}(x,\p\O) + \max_{p,q\in\p\O}|p-q| \\
		&\leq \text{dist}(x,\p\O) \big( 1 + C_{\O} \big) \\
		&\leq  \big( 1 + C_{\O} \big) \min_{\tau} (|\V(\tau)|\tb(x, \V(\tau))).
		\end{split}		
		\end{equation*}
	\end{proof}
	
	\hide
	\begin{lemma} \label{kappa n}{\color{red}[Is this lemma different from Lemma \eqref{lem_unif n}?]}
		Let us consider a plane $S$ which satisfies (\ref{S assume}) and use $\kappa_{S}(y)$ to denote curvature of $\p\O\cap S$ at $y\in \p\O\cap S$, i.e.,
		\Be \notag
		\kappa_{S}(y) := |\mathfrak{r}^{\prime\prime}(s)|,\quad \text{where} \quad \mathfrak{r}(s)=y,
		\Ee 
		where $\mathfrak{r}_{S} : [0, L_{S}) \rightarrow \p\O\cap S$ is regularized{\color{red} [$\dot{r}\neq0$ or $\dot{r}=1$?]} parametriztion of the curve $\p\O \cap S$ when $L_{S} > 0$ is the length of $\p\O \cap S$. Then
		\Be \label{n sim kappa}
		1\lesssim_{\O} c  < |n_{\parallel}(x)| \kappa_{S}(y) \leq C \quad \text{for all $x,y \in \p\O \cap S$.}
		\Ee
		and 
		\Be  
		1 \lesssim_{\O} c  < \frac{|\kappa_{S}(x)|}{|\kappa_{S}(y)|} \leq C , \quad \forall x,y \in \p\O\cap S,	{[	\color{red}Is \ it \ different \ from  \ \eqref{near b-p-ell}?]}
		\Ee
		hold, where $c $ and $C $ are independent to $S$. (For example, $c = C = 1$ if $\O$ is a sphere.) \\
	\end{lemma}
	\begin{proof}
		By Meusnier's theorem (we refer \cite{dC}), normal curvature of $\p\O$ at $y\in \p\O\cap S$ is given by $\kappa_{S}(y) n_{\parallel}(y)$. Since $\O$ is uniformly convex, 
		\[
		\kappa_{1} \leq |\kappa_{S}(y) n_{\parallel}(y)| \leq \kappa_{2},
		\]
		where $\kappa_{1}$ and $\kappa_{2}$  are minimum and maximum principal curvature on $\p\O$ respectively which are uniformly bounded from below and above. Combining with Lemma \ref{lem_unif n}, we obtain both (\ref{n sim kappa}) and (\ref{unif kappa}). We note that $\kappa_{1}$, $\kappa_{2}$, and $c, C$ of Lemma \ref{lem_unif n} are all $S$ independent.  
	\end{proof}
	\unhide

	\begin{lemma} \label{lem_tau ratio}
	(i) Assume \eqref{assume_x} and \eqref{assume_x2}, and recall $\tau_0(x, \bx, v)$ from \eqref{tau_0}. Then,
	\Be \label{tau ratio x}	
		\frac{\tau_{0}(x, \bx, v) - \tau_{-}(x, \bx, v)}{\tau_{+}(x, \bx, v) - \tau_{-}(x, \bx, v)} \gtrsim_{\O} 1.  \\
	\Ee
	(ii) Assume \eqref{assume_v} and \eqref{assume_v2}, and recall $\tau_0(x, v, \bv, \zeta)$ from \eqref{tau_0_v}. Then,
	\Be \label{tau ratio v}
		\frac{\tau_{0}(x, v, \bv, \zeta) - \tau_{-}(x, v, \bv, \zeta)}{\tau_{+}(x, v, \bv, \zeta) - \tau_{-}(x, v, \bv, \zeta)} \gtrsim_{\O} 1.	
	\Ee
	\end{lemma}
	\begin{proof} 
		For fixed $S$, let 
		\Be \label{rR}
			R_{S} := (\min_{\tiny x\in \p\O\cap S}\kappa_{S}(x))^{-1}, \quad r_{S} := (\max_{x\in \p\O\cap S}\kappa_{S}(x))^{-1}. 
		\Ee
		Now, let us consider a circles $C_{1}$ with radius $r$. At any point $y\in\p\O\cap S$, we can place $C_{1}$ inside of $\p\O\cap S$ while $C_{1}$ is  tangential at $y\in\p\O\cap S$ by \eqref{rR}. Similarly, consider another circles $C_{2}$ with radius $R$ and for any point $y\in\p\O\cap S$, we can place $C_{2}$ outside of $\p\O\cap S$ while $C_{2}$ is  tangential at $y\in\p\O\cap S$.  \\
		
		\noindent (i) If we choose $y = \xb(\X(\tau_{0}), v)$, it is obvious that
		\[
		r_{S} \leq |\dot{\X}|(\tau_{0} - \tau_{-}),\quad |\dot{\X}|(\tau_{+} - \tau_{0}) \leq 2R_{S}.
		\]
		Therefore,
		\[
		|\dot{\X}(\tau_{0} - \tau_{-})| \geq r_{S} \geq \frac{r_{S}}{2R_{S}} |\dot{\X}(\tau_{+} - \tau_{0})|
		\]
		holds and we obtain \eqref{tau ratio x} using \eqref{unif kappa}.   \\
		
		\noindent (ii) If $\sin(\theta(\tau_{0} - \tau_{-})) \geq \frac{1}{2}$, we have
		\[
		\theta(\tau_{0} - \tau_{-}) \geq \frac{1}{2}
		\]
		and \eqref{tau ratio v} holds since $\theta(\tau_{+} - \tau_{-})\leq \pi$. \\
		If $0\leq \sin(\theta(\tau_{0} - \tau_{-})) \leq \frac{1}{2}$, let us choose $y = \xb(x, \V(\tau_{0}))$ and circle $C_{1}$ to get
		\[
			r_{S} \leq \big( |\xb(x, \V(\tau_{0})) - x| + r \big) \sin(\theta(\tau_{0} - \tau_{-})) 
		\]
		because $C_{1}$ is inside of $\p\O\cap S$. From $0\leq \sin(\theta(\tau_{0} - \tau_{-})) \leq \frac{1}{2}$,
		\[
		\frac{r_{S}}{2}\leq |\xb(x, \V(\tau_{0})) - x|\sin(\theta(\tau_{0} - \tau_{-})). 
		\]
		Also, considering outer circle $C_{2}$ tangential at $y = \xb(x, \V(\tau_{0}))$, 
		\begin{equation*}
		\begin{split}
		|\xb(x, \V(\tau_{0})) - x|(\theta(\tau_{+} - \tau_{0}))  &\lesssim |\xb(x, \V(\tau_{0})) - x|\sin(\theta(\tau_{+} - \tau_{0})) \\
		&\leq \max_{x,y\in\p\O\cap S}|x-y| \leq 2R_{S},
		\end{split}
		\end{equation*}
		where $A \lesssim B$ means $A \leq CB$ with some generic constant $C > 0$, since $(\theta(\tau_{+} - \tau_{0})) \leq \frac{\pi}{2}$. From above two inequalities, 
		\[
		|\xb(x, \V(\tau_{0})) - x|(\theta(\tau_{0} - \tau_{-})) \gtrsim \frac{r_{S}}{2} \gtrsim \frac{r_{S}}{4R_{S}} |\xb(x, \V(\tau_{0})) - x|(\theta(\tau_{+} - \tau_{0}))
		\]
		holds and we obtain \eqref{tau ratio v} using \eqref{unif kappa}. 	\end{proof}

	\section{Specular Singularity}
	
	\subsection{From fraction to Specular Singularity}
	\begin{lemma}\label{frac sim S}
		Suppose the domain is given as in Definition \ref{def:domain} and \eqref{convex_xi}.  \\
		(i) Let $x, \bx \in \O$, $v\in\R^{3}$ and assume \eqref{assume_x}, \eqref{assume_x2}. Recall shifted position $\tilde{x}=\tilde{x}(x, \bx, v)$ defined in \eqref{def_tildex} of Definition \ref{def_tilde}. For $|x - \tx|\leq 1$,
		\begin{align}
		&\frac{|V(s;t,x,v) - V(s;t, \tilde x, v)|}{|x- \tilde x|}  \mathbf{1}_{ \Big\{\substack{s \leq \min\{ t^{1}(x,v ), t^{1}(\tx, v ) \}  \\ s > \max\{ t^{1}(x,v ), t^{1}(\tx, v ) \}   } \Big\}  } 
		\notag \\
		&\leq |v | + 
		|v |^{2}
		\int_{0}^{1}
		\frac{1}{\mathfrak{S}_{sp}(\tau; x, \tx, v )}
		d\tau,  
		\label{est:V/x}\\
		& \frac{|X(s;t,x,v ) - X(s;t, \tilde x, v )|}{|x- \tilde x|}  \mathbf{1}_{ \Big\{ \substack{s \leq \min\{ t^{1}(x,v ), t^{1}(\tx, v ) \}  \\ s > \max\{ t^{1}(x,v ), t^{1}(\tx, v ) \}   } \Big\}  } 
		\notag \\
		&\leq 
		1 + |v |(t-s) + 
		|v |^{2}(t-s)
		\int_{0}^{1}
		\frac{1}{\mathfrak{S}_{sp}(\tau; x, \tx, v )}
		d\tau.   \label{est:X/x} 
		\end{align}
		(ii) Let $x\in \O$, $v, \bv, \zeta \in\R^{3}$ and assume \eqref{assume_v}, \eqref{assume_v2}. Recall shifted velocity $\tilde{v}=\tilde{v}(v, \bv, \zeta)$ defined in \eqref{def_tildev} of Definition \ref{def_tilde}. For $|v - \tv|\leq 1$, 
		\begin{align}
		& \frac{|V(s;t,x,v+\zeta) - V(s;t, x, \tv+ \zeta)|}{|v- \tilde{v}|}  \mathbf{1}_{ \Big\{ \substack{ s \leq \min\{ t^{1}(x,v+\zeta), t^{1}(x, \tv+\zeta) \}  \\  s > \max\{ t^{1}(x,v+\zeta), t^{1}(x, \tv+\zeta) \}  }  \Big\}  } \notag   \\
		& \lesssim 
		1
		+
		|v+\zeta| (t-s) 
		+
		|v+\zeta|^{2}
		\int_{0}^{1}
		\frac{1}{ \mathfrak{S}_{vel}(\tau; x, v, \tv, \zeta)}
		d\tau, \label{est:V/v}  \\
		& \frac{|X(s;t,x,v+\zeta) - X(s;t, x, \tv+ \zeta)|}{|v- \tv|}  \mathbf{1}_{ \Big\{ \substack{ s \leq \min\{ t^{1}(x,v+\zeta), t^{1}(x, \tv+\zeta) \}  \\  s > \max\{ t^{1}(x,v+\zeta), t^{1}(x, \tv+\zeta) \}}  \Big\}  }
		\notag   \\
		& \lesssim 
		(t-s)
		+
		|v+\zeta| (t-s)^{2} 
		+
		|v+\zeta|^{2} (t-s)
		\int_{0}^{1}
		\frac{1}{ \mathfrak{S}_{vel}(\tau; x, v, \tv, \zeta)}
		d\tau.  \label{est:X/v}
		\end{align}
	\end{lemma}

	\begin{remark}
		In the regularity estimate, we are only interested in the case that $s \in [0,t]$. Therefore, if either $t^1(x,v)=-\infty$ or $t^1(\tilde x,v)=-\infty$, which means either one of the trajectory from $(t,x,v)$ or $(t, \tilde{x}, v)$ missed the boundary $\p\O$, then $\mathbf{1}_{ \Big\{ \substack{ s \leq \min\{ t^{1}(x,v), t^{1}(\tilde x, v) \}  \\  s > \max\{ t^{1}(x,v), t^{1}(\tilde x, v) \}  }  \Big\}  }\equiv 0$. We have the same conclusion for $\mathbf{1}_{ \Big\{ \substack{ s \leq \min\{ t^{1}(x,v+\zeta), t^{1}(x, \tv+\zeta) \}  \\  s > \max\{ t^{1}(x,v+\zeta), t^{1}(x, \tv+\zeta) \}  }  \Big\}  }$. The other nontrivial case, only one trajectory hits the boundary, will be discussed in the next lemma.   \\
		\end{remark}
	
	\begin{proof}
		If  $s > \max\{ t^{1}(x,v), t^{1}(\tx, v)\}$ or $s > \max\{ t^{1}(x,v+\zeta), t^{1}(x, \tv+\zeta)\}$, it is trivial. We consider $s \leq \min\{ t^{1}(x,v+\zeta), t^{1}(x, \tv+\zeta)\}$ and $s \leq \min\{ t^{1}(x,v ), t^{1}(\tx, v )\}$ cases only. \\
		{\bf Step 1}
		From \eqref{nabla_x_bv_b}, for $s \leq t-\tb(x,v)$,
		\Be
		\begin{split}\label{nabla_x V}
			&\nabla_{x} V(s;t,x,v) = \nabla_{x} [v- 2(n (\xb) \cdot v) n(\xb)]\\
			&
			= -2 (v\cdot n(\xb))\frac{1}{|\nabla\xi(\xb)|}\Big( I - n(\xb)\otimes n(\xb)\Big) \nabla^{2}\xi(\xb) \nabla_{x}\xb \\
			&\quad - \frac{2}{|\nabla\xi(\xb)|}(n(\xb)\otimes v) \Big( I - n(\xb)\otimes n(\xb) \Big) \nabla^{2}\xi(\xb) \nabla_{x}\xb \\
			&= - \frac{2}{|\nabla\xi(\xb)|} \big( (v\cdot n(\xb))R_{\xb} + n(\xb)\otimes v \big) \nabla^{2}\xi(\xb) \Big(  I - \frac{v\otimes n(\xb)}{v\cdot n(\xb)} \Big)
		\end{split}
		\Ee
		and
		\Be
		\begin{split}\label{nabla_x X}
			&\nabla_{x} X(s;t,x,v)  = 
			\nabla_{x} [\xb -  (t-\tb -s ) (v- 2 (n(\xb) \cdot v) n(\xb))]\\
			&= I - \frac{v\otimes n(\xb)}{v\cdot n(\xb)} + \big( v - 2(v\cdot n(\xb))n(\xb) \big)\otimes \nabla_{x}\tb - (t-\tb-s)\nabla_{x}V(s) \\
			&= R_{\xb} - (t-\tb-s)\nabla_{x}V(s;t,x,v).
		\end{split}\Ee 
		
		Similarly, we get
		\Be
		\begin{split}\label{nabla_v V}
			&\nabla_{v} V(s;t,x,v) = \nabla_{v} [v- 2(n (\xb) \cdot v) n(\xb)]\\
			&
			= R_{\xb} + \frac{2 \tb}{|\nabla\xi(\xb)|} \big( (v\cdot n(\xb))R_{\xb} + n(\xb)\otimes v \big) \nabla^{2}\xi(\xb) \Big(  I - \frac{v\otimes n(\xb)}{v\cdot n(\xb)} \Big)
		\end{split}
		\Ee
		and
		\Be
		\begin{split}\label{nabla_v X}
			&\nabla_{v} X(s;t,x,v)  = 
			\nabla_{v} [\xb -  (t-\tb -s ) (v- 2 (n(\xb) \cdot v) n(\xb))]\\
			&= -\tb\Big( I - \frac{v\otimes n(\xb)}{v\cdot n(\xb)} \Big) + \big( v - 2(v\cdot n(\xb))n(\xb) \big)\otimes \nabla_{v}\tb - (t-\tb-s)\nabla_{v}V(s) \\
			&= -\tb R_{\xb} - (t-\tb-s)\nabla_{v}V(s;t,x,v).
		\end{split}\Ee 
		
		
		{\bf Step 2}	
		First, we consider \eqref{est:V/x}. 
		From  \eqref{nabla_x V},
		\begin{align} 
		&V(s;t,x,v )- V(s;t, \tx, v ) 
		= \int^{1}_{0} 
		\nabla_x V(s;t, \X(\tau),v ) \dot{\X}
		\dd \tau  \notag \\
		& \lesssim 
		|x - \tx | |v |
		\int_{0}^{1}
		\bigg(1+
		|v |  \frac{\big|
			\frac{
				\dot{\X}}{|\dot{\X}|} \cdot  \nabla \xi(\xb (\X(\tau), v))  \big|}{| 	v \cdot \nabla \xi (\xb
			(\X(\tau), v )
			) |}\bigg)  
		d\tau.
		\notag  
		\end{align}
		Using \eqref{def:S_x}, we derive \eqref{est:V/x}.  \\
		
		For \eqref{est:X/x}, we use \eqref{nabla_x X} and then similarly as above,  
		\begin{align} \notag
		&X(s;t,x,v )- X(s;t, \tx, v ) 
		= \int^{1}_{0} 
		\nabla_x X(s;t,\X(\tau),v ) \dot{\X}
		\dd \tau  \notag \\
		& \lesssim 
		|x-\tx|
		+
		(t-s) |x - \tx | |v |
		\int_{0}^{1}
		\bigg(1+
		|v |  \frac{\big|
			\frac{
				\dot{\X}}{|\dot{\X}|} \cdot  \nabla \xi(\xb (\X(\tau), v ))  \big|}{| 	v \cdot \nabla \xi (\xb
			(\X(\tau), v )
			) |}\bigg)  
		d\tau
		\notag  
		\end{align}
		holds.   \\
		
		For $v$-directional parametrization, using \eqref{nabla_v V},
		\begin{align} \notag
		&V(s;t,x,v+\zeta) - V(s;t, x, \tv +\zeta) 
		= \int^{1}_{0} 
		\nabla_v V(s;t,x, v(\tau)) \dot{\V}
		\dd \tau  \notag \\
		& \lesssim 
		|v-\tv|
		+
		|v-\tv|  |v+\zeta| 
		\int_{0}^{1}
		\bigg( (t-s)  +
		|v+\zeta|  \frac{ \tb(x,\V(\tau)) \big|
			\frac{
				\dot{\V}(\tau)}{|\dot{\V}(\tau)|} \cdot  \nabla \xi(\xb (x, \V(\tau)))  \big|}{| \V(\tau) \cdot \nabla \xi (\xb(x, \V(\tau))) |}\bigg)  
		d\tau
		\notag  
		\end{align}
		which yields \eqref{est:V/v}. \eqref{est:X/v} is also gained similarly using \eqref{nabla_v X}. 
	\end{proof}
	
	\begin{lemma} \label{L sim S}
		Suppose the domain is given as in Definition \ref{def:domain} and \eqref{convex_xi}. \\
		(i) Let $x, \bx \in \O$, $v\in\R^{3}$ and assume \eqref{assume_x}, \eqref{assume_x2}. Recall shifted position $\tilde{x}=\tilde{x}(x, \bx, v)$ defined in \eqref{def_tildex} of Definition \ref{def_tilde}. If $\tb(\tx, v ), \tb(x,v ) < \infty$, 
		\begin{equation} \label{est:Lx}
		\begin{split}
		(t^{1}(x,v ) - s) \mathbf{1}_{t^{1}(\tilde{x},v ) < s \leq t^{1}(x,v )} 
		&\leq |x-\tx| \int_{0}^{1} \frac{1}{\mathfrak{S}_{sp}(\tau; x, \tx, v )} d\tau.
		\end{split}
		\end{equation}
		(ii) Let $x\in \O$, $v, \bv, \zeta \in\R^{3}$ and assume \eqref{assume_v}, \eqref{assume_v2}. Recall shifted velocity $\tilde{v}=\tilde{v}(v, \bv, \zeta)$ defined in \eqref{def_tildev} of Definition \ref{def_tilde}. If $\tb(x, \tv+\zeta), \tb(x,v+\zeta) < \infty$, 
		\begin{equation} \label{est:Lv}
		\begin{split}
		(t^{1}(x,v+\zeta) - s) \mathbf{1}_{ t^{1}(x, \tv+\zeta) < s \leq t^{1}(x,v+\zeta)} 
		&\lesssim |v-\tilde{v}|  \int_{0}^{1} \frac{1}{\mathfrak{S}_{vel}(\tau; x, v, \tv, \zeta)} d\tau.
		\end{split}
		\end{equation}
	\end{lemma}
	\begin{proof}
		For \eqref{est:Lx}
		\Be \notag
		\begin{split}
			(t^{1}(x,v ) - s) \mathbf{1}_{t^{1}(\tilde{x},v ) < s \leq t^{1}(x,v )}  
			&\leq |\tb({x},v ) - \tb(\tx,v )| \\
			&=  \Big| \int_{0}^{1} \nabla_{x}\tb( \X(\tau),v ) \dot{\X} d\tau \Big|   \\
			&\leq  \int_{0}^{1} \Big|  \frac{\dot{\X}\cdot\nabla\xi(\xb(\X(\tau),v ))}{ v\cdot\nabla\xi(\xb(\X(\tau), v))}\Big|  d\tau  \\
			&\leq |\dot{\X}| \int_{0}^{1} \frac{1}{\mathfrak{S}_{sp}(\tau; x, \tx, v)} d\tau.
		\end{split}
		\Ee
		Similarly, for \eqref{est:Lv}
		\Be  \notag
		\begin{split}
			\begin{split}
				(t^{1}(x,v+\zeta) - s) \mathbf{1}_{t^{1}(x, \tv+\zeta) < s \leq t^{1}(x,v+\zeta)}  
				&\leq |\tb({x},v+\zeta) - \tb(x, \tv+\zeta)| \\
				&=  \Big| \int_{0}^{1} \nabla_{v}\tb(x, \V(\tau)) \dot{\V}(\tau) d\tau \Big|   \\
				&\leq \int_{0}^{1} \Big| \tb(x, \V(\tau)) \frac{\dot{\V}(\tau)\cdot\nabla\xi(\xb( x, \V(\tau)))}{ \V(\tau)\cdot\nabla\xi(\xb(x, \V(\tau)))}\Big|  d\tau  \\
				&\leq |\dot{\V}(\tau)| \int_{0}^{1} \frac{1}{\mathfrak{S}_{vel}(\tau; x, v, \tv, \zeta )} d\tau.
			\end{split}
		\end{split}
		\Ee
	\end{proof}

	\subsection{Averaging specular Singularity}

	\hide
	\begin{lemma}\label{comparison_lemma} 
		For given $a_0>0, b_0>0, \tau_2>\tau_1$, assume $|a(\tau)| \leq a_0$ and $0 \leq b(\tau) \leq b_0$ for all $\tau \in [  \tau_1,  \tau_2]$. If a continuous function $y(\tau)$ solves
		\Be
		\frac{d}{d\tau} \left(\frac{y  ^2}{2} \right) = a(\tau) y + b(\tau) \ \ \text{for } \tau \in [\tau_1, \tau
		_2],  \   \text{and} \    \ 
		y(\tau_1)=0, \label{eqtn_x} 
		\Ee
		then 
		\Be
		0 \leq y(\tau) \leq   2 \sqrt{b_0}|\tau-\tau_1|^{1/2}  \  \ \text{for all}  \ \tau  \in \Big[\tau_1, \tau_1 + \frac{b_0}{16(a_0)^2}\Big]
		\cap [\tau_1, \tau_2]
		.\label{compare_x&y}
		\Ee

	\end{lemma}
	
	\begin{proof}In \textit{Step 1}, we will construct a solution $y_\e (\tau)$ which solves, for each $\e>0$,  
		\Be
		\frac{d}{d\tau}\left(\frac{ y_\e ^2}{2}\right)  = a_0 y_\e (\tau)+ b_0  \ \ \text{for } \tau \in [\tau_1, \tau
		_2],  \   \text{and} \    \ 
		y_\e(\tau_1)=\e
		,\label{eqtn_y_e}
		\Ee
		and prove an upper bound 
		\Be\label{upper_y_e} 
		y_\e (\tau) \leq \e + 2 \sqrt{b_0} (\tau-\tau_1)^{1/2} \ \ \text{for }  \  
		\tau \in \Big[\tau_1, \tau_1 + \frac{b_0}{16(a_0)^2}\Big]\cap [\tau_1, \tau_2]
		. 
		\Ee\hide
		Then in \textit{Step 2} by proving an equicontinuity of $\{y_\e (\tau)\}$ we can pass a limit (up to subsequences) to get $y(\tau)$ solves 
		\Be
		\frac{d}{d\tau} \left(\frac{y ^2}{2}\right)  = a_0 y (\tau)+ b_0, \ \ 
		y(0) = 0,\label{eqtn_y}
		\Ee 
		and satisfies 
		\Be\label{upper_y}
		y(\tau) \leq 2 \sqrt{b_0} (\tau-\tau_1)^{1/2} \ \ for \  \tau \in \Big[\tau_1, \tau_1 + \frac{b_0}{16(a_0)^2}\Big]
		.
		\Ee\unhide
		In \textit{Step 2}, we compare (\ref{eqtn_y_e}) and (\ref{eqtn_x}) and conclude the upper bound of (\ref{compare_x&y}). 
		
		\smallskip

		\textit{Step 1. } The ODE in (\ref{eqtn_y_e}) is equivalent to, as long as $y_\e \neq 0$, 
		\Be
		\frac{d y_\e}{d \tau} = \frac{a_0 y_\e + b_0}{y_\e}.\label{ODE_y_e}
		\Ee
		From (\ref{ODE_y_e}), $a_0>0, b_0>0$ and $y_\e (\tau_1) =\e>0$, as long as a solution exists, clearly we have 
		\Be\label{y_e_increasing}
		\frac{d y_\e(\tau)}{d \tau}>0 \ \  \text{and} \ 	\ y_\e (\tau) \geq \e 
		\ \ \text{for} \ \tau \in [\tau_1, \tau_2].
		\Ee 
		Since (\ref{ODE_y_e}) is separable we solve it implicitly as   
		\Be\label{tau_y_e}
		\mathfrak{F} (\tau, y_\e) =0, \ \text{where} \ 
		\mathfrak{F} (\tau, y_\e) := \tau - \tau_1 -  \frac{y_\e- \e}{a_0} +  \frac{b_0}{(a_0)^2} \ln \bigg(\frac{y_\e + {b_0}/{a_0}}{\e + {b_0}/{a_0}}\bigg).
		\Ee
		Note that, using $y_\e >0, a_0>0,$ and $b_0>0$,  
		$  
		\frac{\p\mathfrak{F}}{\p y_\e} 
		= \frac{-1}{a_0} \frac{y_\e}{y_\e + {b_0}/{a_0}}
		\neq 0$. By the implicit function theorem we have a unique smooth solution $y_\e$ solving (\ref{ODE_y_e}) and therefore also solving (\ref{eqtn_y_e}). 
		
		Expanding (\ref{tau_y_e}), we get 
		\begin{align}\notag
		\tau-\tau_1  =  \left( \frac{1}{a_0} - \frac{b_0}{(a_0)^2} \frac{1}{\e + b_0/a_0}\right) (y_\e-\e) 
		+ \frac{b_0}{2 (a_0)^2} \frac{1}{(\e+ b_0/a_0)^2} (y_\e -\e)^2 
		+ \frac{b_0}{(a_0)^2}
		\frac{1}{(\e+ b_0/a_0)^3} O\big(|y_\e -\e|^3\big).\notag
		\end{align}
		Note that $\frac{1}{a_0} - \frac{b_0}{(a_0)^2} \frac{1}{\e + b_0/a_0}\geq
		\frac{1}{a_0} - \frac{b_0}{(a_0)^2} \frac{1}{  b_0/a_0}
		= 0$, and $(y_\e -\e)\geq0$ from (\ref{y_e_increasing}). Therefore we get, for $\e \ll b_0/a_0$,  
		\begin{align*}
		\tau-\tau_1   
		\geq \frac{b_0}{2 (a_0)^2} \frac{1}{(\e+ b_0/a_0)^2} (y_\e -\e)^2+ \frac{b_0}{(a_0)^2}
		\frac{1}{(\e+ b_0/a_0)^3}  |y_\e -\e|^3 
		\geq
		\Big(1 + \frac{a_0}{b_0} |y_\e - \e|  
		\Big) \times 
		\frac{1}{ 4b_0}   (y_\e -\e)^2.
		\end{align*}
		This yields (\ref{upper_y_e}) through a standard continuation argument for $ \tau_1 \leq \tau \leq \tau_1 + \frac{b_0}{16 (a_0)^2}$. 
		
		\smallskip
		
		\textit{Step 2. } We prove that 
		\Be\label{compare_x_and_y}
		0 \leq y(\tau) < y_\e (\tau) \  for \  all  \ \tau\in \Big[\tau_1, \tau_1 + \frac{b_0}{16(a_0)^2}\Big]\cap [\tau_1, \tau_2].
		\Ee
		Clearly $y(\tau_1)=0 < \e=   y_\e (\tau_1)$. Assume that there exists $T\leq \tau_1 + \frac{b_0}{16(a_0)^2}$ such that $y(\tau) < y_\e  (\tau)$ for all $0 \leq \tau < T$ and $y(T) = y_\e  (T)$. First we claim that 
		\Be\notag
		y(\tau)>0 \ 
		\text{  for }  \ 
		\tau_1 <\tau \leq T
		.\Ee
		Otherwise there exists $\tau_*> \tau_1$ such that $y(\tau)>0$ for all $\tau_1 <\tau< \tau_*$ but $\lim_{\tau \uparrow \tau_*} y(\tau)=0$. Then $ \lim_{\tau \uparrow \tau_*}\frac{d}{d\tau}y(\tau) = \lim_{\tau \uparrow \tau_*}\frac{a(\tau) y(\tau) +  b(\tau)}{ y(\tau)} =+ \infty$. This is a contradiction since $y(\tau)>0= y(\tau_*)$ for all $\tau_1 <\tau<\tau_*$, and hence we derive the claim.

		However we derive that, from (\ref{eqtn_x}) and (\ref{eqtn_y_e}),
		\[
		\frac{y (T)^2}{2}   = \int_0^T \{a(\tau) y(\tau) + b(\tau)   \} \dd \tau   <\frac{\e^2}{2}+  \int_0^T \{a_0 y _\e(\tau) + b_0  \} \dd \tau   = \frac{y_\e (T)^2}{2},
		\]
		where we have used $a(\tau) y(\tau) \leq |a(\tau) ||y(\tau)| \leq a_0 y(\tau) \leq a_0 y_\e (\tau)$. This strict inequality leads a contradiction to $y(T)= y_\e (T)$. This proves (\ref{compare_x_and_y}).
		
		From (\ref{upper_y_e}) and (\ref{compare_x_and_y}) we derive that for any $\e>0$ we have $y(\tau) \leq \e + 2\sqrt{b_0} (\tau-\tau_1)^{1/2}$ for all $\tau\in \Big[\tau_1, \tau_1 + \frac{b_0}{16(a_0)^2}\Big]\cap [\tau_1, \tau_2]$. By passing a limit $\e \downarrow 0$ we conclude the upper bound of (\ref{compare_x&y}). 
		%
		%
		%
		%
	\end{proof}	
	\unhide
	We start with the ODEs for the specular singularities. Let $\O= \{x \in \R^3: \xi(x)<0 \}$ for a $C^2$-function $\xi : \R^3 \rightarrow \R$. Let two arbitrary maps $\tau \mapsto \X(\tau) \in \O= \{x \in \R^3: \xi(x)<0 \}$  and $\tau \mapsto \V(\tau) \in \R^3$ are differentiable and $\ddot{\X}\equiv 0$. As long as $\xb(\X(\tau), v)$ is well-defined on $\p\O$, we compute 	
	\Be\begin{split}\label{ODE:dot_x_n}
		&\frac{d}{d\tau}(  \dot{\X} (\tau) \cdot \nabla \xi (\xb(\X(\tau), v)) )  \\
		&=   \frac{-1}{-\nabla \xi (\xb(\X(\tau), v)) \cdot v }
		\Big\{ ( \dot{\X} (\tau) \cdot \nabla \xi (\xb(\X(\tau), v)))(-v\cdot \nabla^2 \xi (\xb(\X(\tau), v))  \dot{\X} (\tau)) \\
		& \ \ \ \ \ \ \ \ \ \ \ \ \ \ \   \ \ \ \ \ \ \ \ \ \ \ \ \  \  \ 
		+(-\nabla \xi (\xb(\X(\tau), v)) \cdot v)(-\dot{\X} (\tau) \cdot \nabla ^2 \xi (\xb(\X(\tau), v))  \dot{\X}(\tau))		
		\Big\} , 
	\end{split}\Ee\Be 
	\begin{split} \label{ODE_x:nabla_xi_v}
		&\frac{d}{d\tau}
		(- \nabla \xi( \xb(\X(\tau), v)) \cdot v ) \\
		&=      \frac{1}{ - \nabla \xi(\xb(\X(\tau), v)) \cdot v }
		\Big\{	
		{ ( \dot{\X} \cdot \nabla \xi( \xb(\X(\tau), v)))   (-
			v \cdot  \nabla^2 \xi( \xb(\X(\tau), v))  v 
			)} 
		\\
		& \ \ \ \ \ \ \ \ \ \ \ \ \ \ \   \ \ \ \ \ \ \ \ \ \ \ \ \ \ \ 
		+   (- \nabla \xi(\xb(\X(\tau), v)) \cdot v ) (-\dot{\X} \cdot  \nabla^2 \xi( \xb(\X(\tau), v))   v )
		\Big\}.
	\end{split}
	\Ee 
As long as $\xb(x, \V(\tau))$ is well-defined on $\p\O$, we obtain  
	\Be\begin{split}\label{ODE:dot_v_n}
		&\frac{d}{d\tau}\big(  \dot{\V} (\tau)  \cdot \nabla \xi (\xb(x, \V(\tau))) \big) 
		= \ddot{\V}(\tau)\cdot\nabla\xi(\xb(x, \V(\tau))) \\	
		&\quad + \frac{ \tb(x, \V(\tau))}{-\nabla \xi (\xb(x, \V(\tau))) \cdot \V(\tau) }
		\Big\{  \big( \dot{\V} (\tau) \cdot \nabla \xi (\xb(x, \V(\tau)))\big)
		\big(
		- \V(\tau)\cdot  \nabla^2 \xi (\xb(x, \V(\tau)))   \dot{\V} (\tau)\big)
		\\ 
		&   \ \ \ \ \ \ \ \ \ \ \ \ \ \ \   \ \ \ \ \ \ \ \ \ \ \ \ \ \ \  \ \ \  
		+\big(-\nabla \xi (\xb(x, \V(\tau))) \cdot \V(\tau)\big)\big(- \dot{\V} (\tau) \cdot   \nabla^2 \xi (\xb(x, \V(\tau)))  \dot{\V}(\tau)\big)
		\Big\},
	\end{split}\Ee\Be
	\begin{split}\label{ODE_v:nabla_xi_v}
		&	 \frac{d}{d\tau} \big(
		- \nabla \xi (\xb(x, \V(\tau)))  \cdot \V(\tau)
		\big) = - \nabla \xi(\xb(x, \V(\tau))) \cdot \dot{\V}(\tau)\\
		& \ \ \ \  \ \   +  \frac{-\tb(x, \V(\tau))}{-\nabla \xi (\xb(x, \V(\tau))) \cdot \V(\tau) }
		\Big\{\big( \dot{\V} (\tau) \cdot \nabla \xi (\xb(x, \V(\tau)))\big)
		\big(
		- \V(\tau)\cdot  \nabla^2 \xi (\xb(x, \V(\tau)))   {\V} (\tau)\big)\\
		& \ \ \ \ \ \ \ \ \ \ \ \ \ \ \   \ \ \ \ \ \ \ \ \ \ \ \ \ \ \  \ \ \       
		+\big(-\nabla \xi (\xb(x, \V(\tau))) \cdot \V(\tau)\big)\big(- \dot{\V} (\tau) \cdot   \nabla^2 \xi (\xb(x, \V(\tau)))   {\V}(\tau)\big)
		\Big\}.
	\end{split}
	\Ee
		These are the direct outcome of the  basic computation (\cite{CKL1, GKTT1})
		\Be\label{nabla_x_bv_b}
		\begin{split}
			&	\nabla_{x} \tb = \frac{ \nabla \xi(\xb)}{ \nabla \xi (\xb) \cdot v}
			,
			\  \ \
			\nabla_{v} \tb =
			- \tb 	\nabla_{x} \tb,
			\\
			&		\nabla_{x} x_{\mathbf{b}} =  I - \frac{v\otimes n(\xb)}{v\cdot n(\xb)},
			\  \ \ \nabla_{v} x_{\mathbf{b}} = - \tb  \nabla_{x} x_{\mathbf{b}},  \\
			&		\nabla_{x} x_{\mathbf{b}} =  I - \frac{v\otimes n(\xb)}{v\cdot n(\xb)},
			\  \ \ \nabla_{v} x_{\mathbf{b}} = - \tb  \nabla_{x} x_{\mathbf{b}},  \\
			& \nabla_{x} n(\xb) = \frac{1}{|\nabla\xi(\xb)|} \Big( I - n(\xb)\otimes n(\xb) \Big)\nabla^{2}\xi(\xb).  \\
		\end{split}
		\Ee


The next differential inequalities are crucially used to prove Proposition \ref{prop_unif H}.

	\begin{lemma}[ODE for Specular Singularity]\label{lemma:ODE}Suppose the domain is given as in Definition \ref{def:domain} and \eqref{convex_xi}. 
	
	(i)	Recall $\tilde{x} =\tilde{x} (x, \bar{x}, v) $ in Definition \ref{def_tilde} under assumption \eqref{assume_x} and \eqref{assume_x2}. 
We also	recall $\mathfrak{S}_{sp}(\tau; x, \tilde{x}, v)$ in \eqref{def:S_x} and $\X(\tau)$ in \eqref{xv para}. For $\tau \in (\tau_-(x, \bar{x}, v), \tau_+(x, \bar{x}, v))$,
		\Be\label{ODE:S_x}
		\begin{split}
			\frac{d \mathfrak{S}_{sp}(\tau;x,\tilde{x}, v)}{d\tau} 
			&
			\geq  
			\frac{1}{\mathfrak{S}_{sp}(\tau; x, \tilde{x}, v)} \frac{ \theta_\O   |\dot{\X}|^{2} }{ |\dot{\X} \cdot \nabla \xi (\xb(\X(\tau), v))| }  \big(  |v|^2 + \mathfrak{S}^{2}_{sp}(\tau; x, \tilde{x}, v)  \big)  .
		\end{split}
		\Ee	
		(ii) We recall $\tilde{v} =\tilde{v} (v, \bar{v}, \zeta) $ in Definition \ref{def_tilde} under assumption \eqref{assume_v} and \eqref{assume_v2}. Recall $\mathfrak{S}_{vel}(\tau; x, v, \tilde{v}, \zeta)$ in \eqref{def:S_v} and $\V(\tau)$ in \eqref{xv para}. Define 
			\Be \label{def:S vel tilde}
		\widetilde{\mathfrak{S}}_{vel}(\tau; x, v, \tilde{v}, \zeta) 
		:= \Big| \frac{\V(\tau)\cdot\nabla\xi(\xb(x, \V(\tau)))}{\dot{\V}(\tau)\cdot\nabla\xi(\xb(x, \V(\tau)))} \Big|
		= \frac{\tb(x, \V(\tau))}{|\dot{\V}(\tau)|} \mathfrak{S}_{vel}(\tau; x, v, \tilde{v}, \zeta).
		\Ee
	For $\tau \in (\tau_-(x, v, \bar{v}, \zeta), \tau_+(x,  v, \bar{v}, \zeta))$,  
		\Be\label{ODE:S_v}
		\begin{split}
			\frac{d \widetilde{\mathfrak{S}}_{vel}(\tau;x  , v, \tilde v, \zeta)}{d\tau} 
			&\geq   1 +\frac{ \theta_{\O} |\V(\tau)|^{2}  \tb(x, \V(\tau))  }{ \widetilde{\mathfrak{S}}_{vel}(\tau; x, v, \tilde{v}, \zeta) |\dot\V(\tau) \cdot\nabla\xi(\xb(x, \V(\tau)))| } 
			\Big( 1 + \widetilde{\mathfrak{S}}^{2}_{vel}(\tau; x, v, \tilde{v}, \zeta)  \Big).
		\end{split}
		\Ee	
		
		\hide
		\Be\begin{split}
			\frac{d}{d\tau} \mathfrak{S}(\tau)
			&=   |\dot v| +
			\frac{1}{\mathfrak{S}(\tau)}
			\frac{\tb|\dot v|}{\big| \frac{\dot v}{|\dot v|} \cdot \nabla \xi\big|^3 }
			\Big[
			\Big(\frac{\dot v}{|\dot v|} \cdot \nabla \xi\Big) v - (\nabla \xi \cdot v) \frac{\dot v}{|\dot v|}
			\Big] \cdot (- \nabla^2 \xi) \cdot\Big[
			\Big(\frac{\dot v}{|\dot v|} \cdot \nabla \xi\Big) v - (\nabla \xi \cdot v) \frac{\dot v}{|\dot v|}
			\Big]\\
			&\geq    |\dot v| +
			\frac{1}{\mathfrak{S}(\tau)}
			\frac{\tb|\dot v|}{\big| \frac{\dot v}{|\dot v|} \cdot \nabla \xi\big|^3 }
			\theta_\O \bigg(
			\Big|
			(\nabla \xi \cdot v) - \Big(\nabla \xi \cdot \frac{\dot v}{|\dot v|}\Big) \Big(\frac{\dot v}{|\dot v|} \cdot v\Big)
			\Big|^2
			+ \Big| \frac{\dot v}{|\dot v|} \cdot \nabla \xi\Big|^2 \Big|\Big(I-\frac{\dot v}{|\dot v|} \otimes \frac{\dot v}{|\dot v|} \Big) v\Big|^2
			\bigg)
			\\
			& \geq  |\dot v| +
			\frac{1}{\mathfrak{S}(\tau)}
			\frac{\theta_\O\tb }{\big|  \dot v  \cdot \nabla \xi\big|  } 
			\bigg(
			\frac{  \big|
				\nabla \xi   \cdot  \big( 
				I - \frac{\dot v }{|\dot v |} \otimes \frac{\dot v }{|\dot v |} 
				\big) v
				\big|^2}{ \big| \frac{\dot v}{|\dot v|} \cdot \nabla \xi\big|^2}+
			\Big|\Big(I-\frac{\dot v}{|\dot v|} \otimes \frac{\dot v}{|\dot v|} \Big) v\Big|^2
			\bigg)
		\end{split}\Ee	\unhide\end{lemma}
		\begin{remark}
		Actually for the differential inequalities \eqref{ODE:S_x} and \eqref{ODE:S_v} we do not need the whole setting of Definition \ref{def_tilde}, but only arbitrary $x, \tilde{x}, v, \tilde{v}, \zeta$ satisfy \eqref{perp}. 
		\end{remark}

	\begin{proof} 
		\textbf{Step 1.} First we prove \eqref{ODE:S_x}. Recall $\tau_0(x, \tilde{x}, v)$ in \eqref{tau_0} and let us consider the case $\tau \in [\tau_-(x, \tilde{x}, v), \tau_0(x, \tilde{x}, v)]$. Simply let us write $\mathfrak{S}_{sp}(\tau) = \mathfrak{S}_{sp}(\tau;x,\tilde{x}, v)$ here. Using \eqref{nabla_x_bv_b}, \eqref{ODE:dot_x_n}, and \eqref{ODE_x:nabla_xi_v},
		\begin{align}
		\frac{d}{d\tau} \mathfrak{S}_{sp}(\tau) 
		&=  \frac{1}{\mathfrak{S}_{sp}(\tau)} \frac{ |\dot{\X}|^{2} }{\dot{\X} \cdot \nabla \xi (\xb( \X(\tau), v))}
		\eqref{dS_force}_{*}, \label{dS_force} 
		\\
		\text{where} \ \eqref{dS_force}_{*} &=  \frac{1}{\big|\frac{\dot \X}{|\dot \X|} \cdot \nabla \xi\big|^2 }
		\Big[
		\Big(\frac{\dot \X}{|\dot \X|} \cdot \nabla \xi \Big) v - (\nabla \xi \cdot v) \frac{\dot \X}{|\dot \X|}
		\Big]	\cdot (- \nabla^2 \xi) \cdot \Big[
		\Big(\frac{\dot \X}{|\dot \X|} \cdot \nabla \xi \Big) v - (\nabla \xi \cdot v) \frac{\dot \X}{|\dot \X|}
		\Big]\notag\\
		&\geq \frac{\theta_\O}{\big|\frac{\dot \X}{|\dot \X|} \cdot \nabla \xi\big|^2 }\Big| 
		\Big(\frac{\dot{\X}}{|\dot{\X}|} \cdot \nabla \xi \Big) v - (\nabla \xi \cdot v) \frac{\dot{\X}}{|\dot{\X}|}
		\Big|^2.\label{lower:dS_force}
\end{align}
	Here, we have used the convexity \eqref{convex_xi} to derive the lower bound estimate in \eqref{lower:dS_force}. Here, we abbreviated $\nabla\xi = \nabla\xi(\xb(\X(\tau),v))$  and $\nabla^{2}\xi = \nabla^{2}\xi(\xb(\X(\tau),v))$ for notational simplicity. Now, using the decomposition

		\Be \notag
		v= \Big(v \cdot \frac{\dot\X}{|\dot\X|}\Big) \frac{\dot\X}{|\dot\X|} + \Big( I - \frac{\dot \X}{|\dot \X|} \otimes \frac{\dot \X}{|\dot \X|} \Big)v, 
		\Ee	
		and $\dot{\X}\cdot v = 0$ from \eqref{perp},
		\begin{align}
		\eqref{lower:dS_force}&=
		\theta_\O 
		\Big|
		\Big( 
		I - \frac{\dot \X }{|\dot \X |} \otimes \frac{\dot \X }{|\dot \X |} 
		\Big)v
		\Big|^2+ 
		\frac{\theta_\O}{\big|\frac{\dot \X}{|\dot \X|} \cdot \nabla \xi\big|^2 } 
		\Big|
		\nabla \xi \cdot v - \Big(\nabla \xi\cdot \frac{\dot \X}{|\dot \X|}  \Big) \Big(\frac{\dot \X}{|\dot \X|}  \cdot v \Big)
		\Big|^2\notag\\
		&= \theta_\O 
		|v|^2+  \frac{\theta_\O}{\big|\frac{\dot \X}{|\dot \X|} \cdot \nabla \xi\big|^2 } 
		\Big|
		\nabla \xi (\xb(\X(\tau), v))  \cdot v \Big|^2  = \theta_\O 
		|v|^2 + \theta_\O  \mathfrak{S}^{2}_{sp}(\tau),   \notag   
		\end{align}
		since $\dot{\X}\cdot v = 0$ from \eqref{perp}.   \\
		
		From the above equality, combining with \eqref{dS_force} and \eqref{lower:dS_force}, we conclude \eqref{ODE:S_x}. Proof for $\tau \in [\tau_0(x, \tilde{x}, v), \tau_+(x, \tilde{x}, v)]$ is same. \\

		\smallskip 
		
		\textbf{Step 2.} Next we prove \eqref{ODE:S_v}. Recall $\tau_0(x, v, \bv, \zeta)$ in \eqref{tau_0_v} and let us consider $\tau \in [\tau_- (x,v,\bar{v}, \zeta), \tau_0 (x,v,\bar{v}, \zeta)] $. Simply let us write $\mathfrak{S}_{vel}(\tau) = \mathfrak{S}_{vel}(\tau; x, v, \tilde{v}, \zeta)$ here. 
		%
		Using \eqref{ODE:dot_v_n}, and \eqref{ODE_v:nabla_xi_v},
		
		\hide
		\begin{align}
		& \frac{d}{d\tau}  \left(
		\frac{- \nabla \xi(\xb(x, v(\tau))) \cdot v(\tau)  }{
			- \frac{\dot{v}(\tau)}{|\dot v(\tau)|}\cdot \nabla \xi (\xb(x , v(\tau)))  
		}\right)  \\ 
		&=   |\dot v(\tau)|+
		\left(
		\frac{- \nabla \xi(\xb(x, v(\tau))) \cdot v(\tau)  }{
			- \frac{\dot{v}(\tau)}{|\dot v(\tau)|}\cdot \nabla \xi (\xb(x , v(\tau)))  
		}\right)^{-1}
		\frac{\tb(x, v(\tau))|\dot v(\tau)|}{\big| \frac{\dot v(\tau)}{|\dot v(\tau)|} \cdot \nabla \xi(\xb(x, v(\tau)))\big|^3 } \notag\\
		&\quad\quad		 
		\times 
		\Big[
		\Big(\frac{\dot v(\tau)}{|\dot v(\tau)|} \cdot \nabla \xi(\xb(x, (\tau))) \Big) v(\tau) - (\nabla \xi(\xb(x, (\tau))) \cdot v(\tau)) \frac{\dot v(\tau)}{|\dot v(\tau)|}
		\Big] \cdot (- \nabla^2 \xi) \\
		&\quad\quad	\quad\quad\quad\quad	\quad\quad	\quad\quad		 \cdot\Big[
		\Big(\frac{\dot v(\tau)}{|\dot v(\tau)|} \cdot \nabla \xi(\xb(x, (\tau)))\Big) v(\tau) - (\nabla \xi (\xb(x, (\tau)))\cdot v(\tau)) \frac{\dot v(\tau)}{|\dot v(\tau)|}
		\Big]. \label{D_Sv}
		\end{align} 	
		
		Using the above identity together with \eqref{nabla_x_bv_b} and (\ref{convex_xi}), we derive that 
		\Be \notag
		\begin{split}
			\frac{d}{d \tau} \mathfrak{S}_{vel}(\tau;x,  v,  \bar {v}, u)  
			& =\frac{1}{\tb(x, v(\tau))}
			\Big\{|\dot v(\tau)| + (\ref{D_Sv}) +
			\frac{ \dot v(\tau) \cdot \nabla \xi(\xb(x, v(\tau)))  }{ \nabla \xi(\xb(x, v(\tau))) \cdot v(\tau)}
			\frac{ \nabla \xi(\xb(x, v(\tau))) \cdot v(\tau)  }{ \frac{\dot{v}(\tau)}{|\dot v(\tau)|}\cdot \nabla \xi (\xb(x , v(\tau)))  
			}
			\Big\} \\
			& = \frac{2|\dot{v}(\tau)|}{\tb(x, v(\tau))} + \frac{1}{\tb(x, v(\tau))} \eqref{D_Sv}. 
		\end{split}
		\Ee
		On the other hand, from \eqref{convex_xi}, 
		\Be\begin{split} \notag
			\eqref{D_Sv}	&\geq    
			\frac{\theta_\O }{\mathfrak{S}_{vel}(\tau; x, v, \bv, u)}
			\frac{ |\dot v(\tau)|}{\big| \frac{\dot v(\tau)}{|\dot v(\tau)|} \cdot \nabla \xi (\xb(x, v(\tau)))\big|^3 }  \\
			&\quad \times \bigg\{
			\Big|
			(\nabla \xi (\xb(x, v(\tau))) \cdot v(\tau)) - \Big(\nabla \xi(\xb(x, v(\tau)))\cdot \frac{\dot v(\tau)}{|\dot v(\tau)|}\Big) \Big(\frac{\dot v(\tau)}{|\dot v(\tau)|} \cdot v(\tau)\Big)
			\Big|^2  \\
			&\quad\quad\quad\quad
			+ \Big| \frac{\dot v(\tau)}{|\dot v(\tau)|} \cdot \nabla \xi(\xb(x, v(\tau))) \Big|^2 \Big|\Big(I-\frac{\dot v(\tau)}{|\dot v(\tau)|} \otimes \frac{\dot v(\tau)}{|\dot v(\tau)|} \Big) v(\tau) \Big|^2
			\bigg\}
			\\
			& \geq   
			\frac{1}{\mathfrak{S}_{vel}(\tau; x, v, \bv, u)}
			\frac{\theta_\O  |\dot{v }(\tau)|  }{\big|  \dot v(\tau)  \cdot \nabla \xi(\xb(x, v(\tau))) \big|  } 
			\bigg\{
			\frac{  \big|
				\nabla \xi(\xb(x, v(\tau)))   \cdot  \big( 
				I - \frac{\dot v(\tau) }{|\dot v(\tau) |} \otimes \frac{\dot v(\tau) }{|\dot v(\tau) |} 
				\big) v(\tau)
				\big|^2}{ \big| \frac{\dot v(\tau)}{|\dot v(\tau)|} \cdot \nabla \xi(\xb(x, v(\tau))) \big|^2}  \\
			&\quad\quad\quad\quad +
			\Big|\Big(I-\frac{\dot v(\tau)}{|\dot v(\tau)|} \otimes \frac{\dot v(\tau)}{|\dot v(\tau)|} \Big) v(\tau) \Big|^2
			\bigg\}.
		\end{split}\Ee	
		The above two results would imply \eqref{ODE:S_v} immediately. 
		We compute 
		\unhide
		
		\Be \label{v_frac:est}
		\begin{split} 
			&\frac{d}{d\tau}\Big(  \frac{\V(\tau)\cdot\nabla\xi(\xb(x, \V(\tau)))}{\dot{\V}(\tau)\cdot\nabla\xi(\xb(x, \V(\tau)))}\Big)   \\
			&=
			\frac{1}{\dot{\V}(\tau)\cdot\nabla\xi(\xb(x, \V(\tau)))} 
			\Big[ \dot{\V}(\tau)\cdot\nabla\xi(\xb(x, \V(\tau)))  \\
			&\quad\quad + \frac{\tb(x, \V(\tau))}{\V(\tau)\cdot \nabla\xi(\xb(x, \V(\tau)))} 
			\Big\{ (\dot{\V}(\tau)\cdot \nabla\xi(\xb(x, \V(\tau)))) \V(\tau)\cdot \nabla^{2}\xi(\xb(x, \V(\tau))) \V(\tau)   \\
			&\hspace{6cm}  -
			({\V}(\tau)\cdot \nabla\xi(\xb(x, \V(\tau)))) \dot\V(\tau)\cdot \nabla^{2}\xi(\xb(x, \V(\tau))) \V(\tau) 
			\Big\} \Big]  \\
			&\quad - 
			\frac{ \V(\tau)\cdot\nabla\xi(\xb(x, \V(\tau))) }{ |\dot{\V}(\tau)\cdot\nabla\xi(\xb(x, \V(\tau)))|^{2} } 
			\Big[ \ddot{\V}(\tau)\cdot\nabla\xi(\xb(x, \V(\tau)))  \\
			&\quad\quad + \frac{\tb(x, \V(\tau))}{\V(\tau)\cdot \nabla\xi(\xb(x, \V(\tau)))} 
			\Big\{ (\dot{\V}(\tau)\cdot \nabla\xi(\xb(x, \V(\tau)))) \V(\tau)\cdot \nabla^{2}\xi(\xb(x, \V(\tau))) \dot\V(\tau)   \\
			&\hspace{6cm}  -
			({\V}(\tau)\cdot \nabla\xi(\xb(x, \V(\tau)))) \dot\V(\tau)\cdot \nabla^{2}\xi(\xb(x, \V(\tau))) \dot\V(\tau) 
			\Big\} \Big].  \\
		\end{split}\Ee	
		Since $\ddot\V(\tau) = -\theta^{2}\V(\tau)$ from \eqref{perp}, $- 
		\frac{ \V(\tau)\cdot\nabla\xi(\xb(x, \V(\tau))) }{ |\dot{\V}(\tau)\cdot\nabla\xi(\xb(x, \V(\tau)))|^{2} }  \ddot{\V}(\tau)\cdot\nabla\xi(\xb(x, \V(\tau))) > 0$ and then
		we use \eqref{v_frac:est} to obtain
		\Be  \label{lower:dS_force v}
		\begin{split} 
			& \frac{d}{d\tau} \widetilde{\mathfrak{S}}_{vel}(\tau; x, v, \tilde{v}, \zeta)   \\ 
			&\geq 1 + \tb(x,\V(\tau)) \frac{1}{ AB^{2} } (B\V(\tau) - A\dot{\V}(\tau)) \cdot \nabla^{2}\xi(\xb(x, \V(\tau))) (B \V(\tau) - A\dot{\V}(\tau))  \\
			&\geq 1 - \tb(x, \V(\tau)) \frac{\theta_{\O}}{AB^{2}} |(\dot{\V}(\tau)\cdot\nabla\xi(\xb(x, \V(\tau)))) \V(\tau) - ( \V(\tau)\cdot\nabla\xi(\xb(x, \V(\tau))))\dot{\V}(\tau)|^{2} \\
			&= 1 -  \frac{\theta_{\O}\tb(x, \V(\tau)) }{(\V(\tau)\cdot\nabla\xi(\xb(x, \V(\tau))))} \\
			&\quad \times 
			\Big\{
			\frac{  \big|
				\nabla \xi(\xb(x, \V(\tau)))   \cdot  \big( 
				I - \frac{\dot \V(\tau) }{|\dot \V(\tau) |} \otimes \frac{\dot \V(\tau) }{|\dot \V(\tau) |} 
				\big) \V(\tau)
				\big|^2}{ \big| \frac{\dot \V(\tau)}{|\dot \V(\tau)|} \cdot \nabla \xi(\xb(x, \V(\tau))) \big|^2}+
			\Big|\Big(I-\frac{\dot \V(\tau)}{|\dot \V(\tau)|} \otimes \frac{\dot \V(\tau)}{|\dot \V(\tau)|} \Big)  \V(\tau) \Big|^2
			\Big\} \\
			&= 1 + \tb(x, \V(\tau)) \frac{ \theta_{\O} |\V(\tau)|^{2} }{ \widetilde{\mathfrak{S}}_{vel}(\tau; x, v, \tilde{v}, \zeta) |\dot\V(\tau) \cdot\nabla\xi(\xb(x, \V(\tau)))| } 
			\Big( 1 + \widetilde{\mathfrak{S}}^{2}_{vel}(\tau; x, v, \tilde{v}, \zeta)  \Big),    \\
		\end{split}\Ee	
		where we used \eqref{convex_xi}, $\dot{\V}(\tau) \cdot \V(\tau) = 0$ by \eqref{perp}, and notation,
		\[
			A := \V(\tau)\cdot\nabla\xi(\xb(x, \V(\tau))) < 0,\quad B:= \dot{\V}(\tau)\cdot\nabla\xi(\xb(x, \V(\tau))) < 0.
		\]
		Sign of $B$ comes from \eqref{tau_0_v}. Proof for $\tau \in [\tau_0, \tau_+)$ is same.
		\hide
		\Be
		\begin{split} 
			&\frac{d}{d\tau}\mathfrak{S}_{vel}(\tau; x, v, \bv, \zeta) \\
			&=  |\dot{v}(\tau)| \frac{d}{d\tau}\Big( \frac{v(\tau)\cdot\nabla\xi(\xb(x, v(\tau)))}{\tb(x, v(\tau)) \dot{v}(\tau)\cdot\nabla\xi(\xb(x, v(\tau)))} \Big)  \\
			&= |\dot{v}(\tau)|\Big\{ -\frac{1}{\tb^{2}(x, v(\tau))}(\nabla_{v}\tb(x, v(\tau)) \dot{v}(\tau)) \frac{v(\tau)\cdot\nabla\xi(\xb(x, v(\tau)))}{\dot{v}(\tau)\cdot\nabla\xi(\xb(x, v(\tau)))} + \frac{1}{\tb(x, v(\tau))}\frac{d}{d\tau}\big( \frac{v(\tau)\cdot\nabla\xi(\xb(x, v(\tau)))}{\dot{v}(\tau)\cdot\nabla\xi(\xb(x, v(\tau)))} \big) \Big\}  \\
			&= |\dot{v}(\tau)|\Big\{ -\frac{1}{\tb^{2}(x, v(\tau))}( -\tb(x, v(\tau)) \frac{\dot{v}(\tau)\cdot\nabla\xi(\xb(x, v(\tau)))}{v(\tau)\cdot\nabla\xi(\xb(x, v(\tau))) } ) \frac{v(\tau)\cdot\nabla\xi(\xb(x, v(\tau))) }{\dot{v}(\tau)\cdot\nabla\xi(\xb(x, v(\tau))) } \\
			&\quad\quad\quad\quad\quad\quad
			 + \frac{1}{\tb(x, v(\tau))}\frac{d}{d\tau}\big( \frac{v(\tau)\cdot\nabla\xi(\xb(x, v(\tau))) }{\dot{v}(\tau)\cdot\nabla\xi(\xb(x, v(\tau))) } \big) \Big\}  \\
			&= \frac{|\dot{v}(\tau)|}{\tb(x, v(\tau))}\Big\{ 1 + \frac{d}{d\tau}\big( \frac{v(\tau)\cdot\nabla\xi(\xb(x, v(\tau)))}{\dot{v}(\tau)\cdot\nabla\xi(\xb(x, v(\tau))) } \big) \Big\}
		\end{split}\Ee
		
		Hence using \eqref{v_frac:est},	
		\Be
		\begin{split} 
			&\frac{d}{d\tau}\mathfrak{S}_{vel}(\tau; x, v, \bv, \zeta) \\
			&\geq  \frac{|\dot{v}(\tau)|}{\tb(x,v(\tau))}\Big\{ 3 + \tb(x,v(\tau)) \frac{\theta_{\O}}{|v(\tau)\cdot\nabla\xi(\xb(x, v(\tau)))|} 
			\Big\{ |v(\tau)|^{2} + \mathfrak{S}_{vel}^{2}(\tau; x, v, \bv, \zeta)
			\Big\} \Big\}  \\
			&\geq  \frac{|\dot{v}(\tau)|}{\tb(x,v(\tau))} \Big[ 3 + \frac{\theta_{\O} }{ \mathfrak{S}_{vel}(\tau; x, v, \bv, \zeta) |\widehat{\dot{v}(\tau)}\cdot \nabla\xi(\xb(x, v(\tau)))| } 
			\Big\{ |v(\tau)|^{2} + \mathfrak{S}_{vel}^{2}(\tau; x, v, \bv, \zeta) 
			\Big\}  \Big] \\
			&\geq  \frac{3|\dot{v}(\tau)|}{\tb(x,v(\tau))}  + \frac{\theta_{\O} |\dot{v}(\tau)| }{ \mathfrak{S}_{vel}(\tau; x, v, \bv, \zeta) \tb(x,v(\tau)) |\widehat{\dot{v}(\tau)}\cdot \nabla\xi(\xb(x, v(\tau)))| } 
			\Big\{ |v(\tau)|^{2} + \mathfrak{S}_{vel}^{2}(\tau; x, v, \bv, \zeta) 
			\Big\}   \\
		\end{split}
		\Ee
		\unhide
	\end{proof}

\begin{proposition}\label{prop_avg S} 
	Suppose the domain is given as in Definition \ref{def:domain} and \eqref{convex_xi}.  \\
	\noindent (i) Assume $\tb(x(\tau_*), v) < \infty$ and $\tau_{*}\in [\tau_{-}(x, \bx, v), \tau_{+}(x, \bx, v)]$. Then we have 
	\begin{equation} \label{int:Sx}
		\begin{split}
			\int_{\tau_{-}(x, \bar{x}, v)}^{\tau_{*} }  
			\frac{ \dd \tau 	}{\mathfrak{S}_{sp}(\tau; x,\bar{x}, v)} 
			&\lesssim C_{\Omega} 
			\frac{\tau_{*} - \tau_{-} (x, \bx, v) }{|v\cdot\nabla\xi(\xb(x(\tau_{*}), v)) |}. 
		\end{split}
	\end{equation} 
	\noindent (ii) Assume $\tb(x, v(\tau)) < \infty$ and $\tau_{*}\in [\tau_{-}(x, v, \bv, \zeta), \tau_{+}(x, v, \bv, \zeta)]$. Then we have 
	\begin{equation} \label{int:Sv}
		\begin{split}
			\int_{\tau_{-}(x, v, \bar{v}, \zeta)}^{\tau_{*}}  
			\frac{ \dd \tau 	}{\mathfrak{S}_{vel}(\tau; x,v, \bar{v}, \zeta)}
			&\lesssim C_{\Omega}  \frac{ (\tau_{*} - \tau_{-}) }{ |\V(\tau_{*})\cdot \nabla\xi(\xb(x, \V(\tau_{*})))| }	
			\frac{1}{|\V(\tau_*)|} \big( 1 + \min_{\tau} (|\V(\tau)|\tb(x, \V(\tau))) \big).
		\end{split}
	\end{equation} 
	(Remind that $|\V(\tau)| = |v+\zeta|$ for all $\tau$. )  \\
\end{proposition}
	\begin{proof}
		\textbf{Step 1.} We first prove \eqref{int:Sx} when $ \tau \in [\tau_-, \tau_0]$, where $\tau_{-}(x,\bx, v)$ and $\tau_0=\tau_0(x,\bx, v)$ are defined in \eqref{tau_pm} and \eqref{tau_0}.
		From the ODE \eqref{ODE:S_x}, 
		\begin{equation} \label{ode_sp}
		\begin{split}
		\frac{d}{dt} G_{sp}(\tau; x, \tx, v) \geq \frac{2\theta_{\O}|\dot{\X}|^{2} }{|\dot{\X}\cdot\nabla\xi(\xb( \X(\tau), v))|  } G_{sp}(\tau; x, \tx, v),\quad G_{sp}(\tau; x, \tx, v) := |v|^{2} + \mathfrak{S}^{2}_{sp}(\tau; x, \tx, v).
		\end{split}
		\end{equation}
	 From $\mathfrak{S}_{sp}(\tau_{-})=0$, we derive an upper bound of $G_{sp}(\tau; x, \tx, v) $ by applying the Gronwall's inequality to \eqref{ode_sp}. Then we derive that, in terms of $\mathfrak{S}_{sp}(\tau; x, \tx, v)$,
		\begin{equation*} 
		\begin{split}
		\frac{1}{\mathfrak{S}_{sp}(\tau; x, \tx, v)} 
		&\leq \frac{1}{|v|}\Big( e^{\int_{\tau_{-}}^{\tau} \frac{2\theta_{\O}|\dot{\X}|^{2}}{ |\dot{\X}\cdot\nabla\xi(\xb( \X(s), v))| } ds    } -1 \Big)^{-\frac{1}{2}} \\
		&\leq \frac{1}{|v|}\Big[ \int_{\tau_{-}}^{\tau} \frac{2\theta_{\O}|\dot{\X}|^{2}}{ \max_{\tau_{-}\leq s\leq \tau}|\dot{\X}\cdot\nabla\xi(\xb( \X(s), v))| }     \Big]^{-\frac{1}{2}} \\
		&\lesssim_{\O} \frac{1}{|v||\dot{\X}|} \sqrt{\frac{ |\dot{\X} \cdot\nabla\xi(\xb( \X(\tau_{-} (x, \bar x, v)  ), v))| }{ (\tau-\tau_{-}  (x, \bar x, v) ) }}.
		\end{split}
		\end{equation*}
Here, we have used the fact $0\leq |{\dot{\X}}\cdot\nabla\xi(\xb( \X(\tau), v))| \leq C_{\O} |{\dot{\X}}\cdot\nabla\xi(\xb( \X(\tau_{-}), v))|$ for $\tau\in[\tau_{-}, \tau_{0}]$ by \eqref{mono Cst 1}. Hence, for $\tau_{*} \leq \tau_{0}$, 
		\begin{equation} \label{int root x} 
		\begin{split}
		\int^{\tau_{*}}_{\tau_{-}} \frac{1}{\mathfrak{S}_{sp}(s; x, \tx, v)} ds  
		&\leq C_{\O}\frac{ \sqrt{ {\dot{\X}}\cdot\nabla\xi( \xb( \X(\tau_{-}), v ) )}}{|v||\dot{\X}|}   \sqrt{\tau_{*} - \tau_{-} }\\
		&\leq	\frac{C_\O   |\widehat{\dot{\X}}\cdot\nabla\xi(\xb( \X(\tau_{-}), v))|}{|v| \sqrt{|\dot\X|}}	\sqrt{\tau_*- \tau_- (x, \bar x, v)}
		.
		\end{split}
		\end{equation}
	
	Next we claim that  
	\Be\label{upper_vn}
	\frac{1}{|v|\sqrt{|\dot{\X}|}}
	\leq  \frac{C_\O  \| \nabla \xi \|_{L^\infty (\p\O)} \sqrt{\tau_*-\tau_-(x, \bar x, v)}}{	|v\cdot\nabla\xi(\xb(\X(\tau_*), v))| }
 \ \ 	\text{for any} \ 
	\tau_*  \in [ \tau_-(x, \bar x, v) ,\tau_0(x, \bar x, v)].
	\Ee	
Combining \eqref{upper_vn} and \eqref{int root x}, we can prove \eqref{int:Sx}.

The proof of \eqref{upper_vn} comes from \eqref{ODE_x:nabla_xi_v}:
	\begin{equation*}
	\begin{split}
		\frac{d}{d\tau}(v\cdot\nabla\xi(\xb(\X(\tau), v)))^2 &= 2v\cdot\nabla^2 \xi(\xb( \X(\tau), v)) \big[ (v\cdot\nabla\xi(\xb(\X(\tau), v))) {\dot{\X}} - ( {\dot{\X}}\cdot\nabla\xi(\xb( \X(\tau), v)))v \big] \\
		&\leq C_{\O}|v|^2 | {\dot{\X}}\cdot\nabla\xi(\xb( \X(\tau_{-}), v)) |\\
		& \leq C_{\O} \| \nabla \xi \|_{L^\infty(\p\O)} |v|^2 |\dot \X|. 
	\end{split}
\end{equation*}
We integrate the above inequality from $\tau_-$ to $\tau_*$ and use $v\cdot\nabla\xi(\xb(\X(\tau), v))|_{\tau = \tau_-} =0$ from \eqref{tau_pm}. Then we can prove the claim \eqref{upper_vn}. 
	

\hide	Now	let us obtain upper bound of $v\cdot \nabla \xi(\xb(\X(\tau), v))$. Using \eqref{nabla_x_bv_b}, we get
		\[
			\frac{d}{d\tau}(v\cdot\nabla\xi(\xb(\X(\tau), v))) = v\cdot\nabla^{2}\xi (\xb( \X(\tau), v)) \dot{\X} - v\cdot \nabla^2\xi(\xb(\X(\tau), v)) v \Big(\frac{\dot{\X}\cdot\nabla\xi(\xb( \X(\tau), v))}{v\cdot\nabla\xi(\xb( \X(\tau), v))}\Big),
		\]
		and so obtain
		\begin{equation*}
		\begin{split}
			\frac{d}{d\tau}(v\cdot\nabla\xi(\xb(\X(\tau), v)))^2 &= 2v\cdot\nabla^2 \xi(\xb( \X(\tau), v)) \big[ (v\cdot\nabla\xi(\xb(\X(\tau), v))) {\dot{\X}} - ( {\dot{\X}}\cdot\nabla\xi(\xb( \X(\tau), v)))v \big] \\
			&\leq C_{\O}|v|^2 | {\dot{\X}}\cdot\nabla\xi(\xb( \X(\tau_{-}), v)) |,
		\end{split}
		\end{equation*}
		since $\hat{v}\cdot\nabla\xi(\xb(\X(\tau), v)) \lesssim_{\O} | \widehat{\dot{\X}}\cdot\nabla\xi(\xb( \X(\tau_{-}), v)) |$ can be easily proved mimicking \eqref{dot x opt}. 
		Integrating from $\tau_{-}$ to $\tau$, 
		\[
			|v\cdot\nabla\xi(\xb(\X(\tau), v))| \leq C_{\O} |v| \sqrt{{\dot{\X}}\cdot\nabla\xi(\xb( \X(\tau_{-}), v))} \sqrt{\tau-\tau_{-}}
		\]
		Combining with \eqref{int root x}, we obtain \eqref{int:Sx} for $\tau_* \in [\tau_-, \tau_0]$. \unhide
		\hide
		\Be \label{est:tau-0}
		\int^{\tau_{*}}_{\tau_{-}} \frac{1}{\mathfrak{S}_{sp}(s; x, \tx, v)} ds  
		\leq C_{\O}  |\widehat{\dot{\X}}\cdot\nabla\xi(\xb( \X(\tau_{-}), v))|
		\frac{\tau_{*}-\tau_{-}}{|v\cdot\nabla\xi(\xb( \X(\tau_{*}), v))|}.
		\Ee
		\unhide
		
		Next, let us consider the case $\tau_*  \in [\tau_0 , \tau_+]$. Following the same argument to prove \eqref{int:Sx}, using \eqref{mono Cst 2} instead of \eqref{mono Cst 1}, we can derive that for $\tau_* \in [\tau_-, \tau_0]$
	%
		\Be \label{est:tau0+}
			\int^{\tau_{+}}_{\tau_{*}} \frac{\dd \tau }{\mathfrak{S}_{sp}(\tau; x, \tx, v)} 
			\lesssim C_{\O}  
			\frac{\tau_{+}-\tau_{*}}{|v\cdot\nabla\xi(\xb( \X(\tau_{*}), v))|}.
		\Ee
		Now we split $\int_{\tau_-}^{\tau_*} = \int_{\tau_-}^{\tau_0} + \int_{\tau_0}^{\tau_+} - \int_{\tau_*}^{\tau_+} \leq |\int_{\tau_-}^{\tau_0}| + |\int_{\tau_0}^{\tau_+} | +| \int_{\tau_*}^{\tau_+}|$ and apply \eqref{int:Sx} with $\tau_* \in [\tau_-, \tau_0]$ and \eqref{est:tau0+} to derive that 
		\Be \notag
		\begin{split}
	 			\int_{\tau_{-}}^{\tau_{*} }  
	 		\frac{ \dd \tau 	}{\mathfrak{S}_{sp}(\tau; x,\bar{x}, v)} 
	 		&\lesssim_{\Omega} 
	 		\frac{(\tau_{0} - \tau_{-}) }{|v\cdot\nabla\xi(\xb(x(\tau_{0}), v)) |}
	 		+ \frac{(\tau_{+} - \tau_{0}) }{|v\cdot\nabla\xi(\xb(x(\tau_{0}), v)) |}
	 		+ \frac{(\tau_{+} - \tau_{*}) }{|v\cdot\nabla\xi(\xb(x(\tau_{*}), v)) |}  . 
		\end{split}
		\Ee
		Note that from \eqref{mono Cst 3}, we have $\frac{1 }{|v\cdot\nabla\xi(\xb(x(\tau_{0}), v)) |} \leq \frac{C_\O}{|v\cdot\nabla\xi(\xb(x(\tau_{*}), v)) |}$ 
	. On the other hand, from \eqref{tau ratio x}, we easily obtain $\tau_+ - \tau_0 \lesssim \tau_0 - \tau_-\lesssim \tau_* - \tau_-$ and $\tau_+ - \tau_* \lesssim \tau_+ - \tau_0\lesssim \tau_* - \tau_-$. Hence all three terms above (on the RHS) are bounded by  
		\[
		 \frac{C(\tau_{*} - \tau_{-}) }{|v\cdot\nabla\xi(\xb(x(\tau_{*}), v)) |}.
		\]
		Therefore we prove \eqref{int:Sx} for $\tau_* \in [\tau_0, \tau_+]$.

		\smallskip
		
		\hide
		\textbf{Step 2.}	Next we prove \eqref{Desing_x} when $ \tau \in [\tau_0, \tau_+]$.

		For $\tau \in [\tau_0, \tau_+]$ we redefine $\tilde\tau= -(\tau - \tau_+) \in [0, \tau_+- \tau_0]$. Then $\mathfrak{S}(\tilde{\tau})= \mathfrak{S} (\tau_+-  \tau)$ solve 
		\Be
		\frac{d }{d \tilde{\tau}}
		\mathfrak{S}(\tilde{\tau};x,\bar{x}, v)
		=	\eqref{ODE_S}_*|_{\tau= \tau_+ - \tilde\tau}
		\times \frac{1}{-\dot x \cdot \nabla \xi(\xb(x(\tau), v))  }	\frac{1}{\mathfrak{S}(\tilde\tau;x,\bar{x}, v)}
		\Ee
		Then 
		\Be
		\frac{d }{ d \tilde\tau } \Big(\frac{\mathfrak{S}_{sp} (\tilde\tau)^2}{2}\Big)\geq  \frac{ \big|
			\big( 
			I - \frac{\dot x }{|\dot x |} \otimes \frac{\dot x }{|\dot x |} 
			\big)v
			\big|^2 }{-\dot x  \cdot \nabla \xi(\xb(x(\tau), v))  } \ \ \text{for} \ 
		\tilde\tau \in [0, \tau_+ -\tau_0],  \ \  \  \Big(\frac{\mathfrak{S}(\tilde\tau)^2}{2}\Big)\Big|_{\tilde\tau= 0}=0.
		\Ee\Be
		\int^{\tau_+}_{\tau_0} 
		\bigg(
		\frac{- \nabla \xi(\xb(x(\tau), v)) \cdot v  }{
			- \frac{\dot{x}  }{|\dot{x}  |} \cdot \nabla \xi (\xb(x(\tau), v))
		}
		\bigg)^{-1}  \dd \tau 	
		\Ee
		\unhide
		
		\textbf{Step 2.} We prove \eqref{int:Sv} first when $ \tau \in [\tau_-, \tau_0]$, where $\tau_{-}(x,v, \bv, \zeta)$ and $\tau_0=\tau_0(x, v, \bv, \zeta)$ are defined in \eqref{tau_pm} and \eqref{tau_0_v}. 
		From  the differential inequality \eqref{ODE:S_v},  
		\begin{equation} \label{ode_vel}
		\begin{split}
		\frac{d}{dt} G_{vel}(\tau; x, v, \tv, \zeta) \geq \tb(x, \V(\tau))  \frac{ 2\theta_{\O}|{\V}(\tau)|^{2} }{ |\dot{\V}(\tau) \cdot\nabla\xi(\xb(x, \V(\tau)))|} G_{vel}(\tau; x, v, \tv, \zeta) ,
		\end{split}
		\end{equation}
		where $G_{vel}(\tau; x, v, \tv, \zeta) := 1 + \widetilde{\mathfrak{S}}^{2}_{vel}(\tau; x, v, \tv, \zeta)$. We apply the Gronwall's inequality to \eqref{ode_vel} using $\widetilde{\mathfrak{S}}_{vel}(\tau_{-})=0$. Then in terms of $\widetilde{\mathfrak{S}}_{vel}(\tau; x, v, \tv, \zeta)$, we have that for $\tau\in[\tau_{-}, \tau_{0}]$, 
		\begin{equation*} 
		\begin{split}
		\frac{1}{\widetilde{\mathfrak{S}}_{vel}(\tau; x, v, \tv, \zeta)} &\leq \Big( e^{\int_{\tau_{-}}^{\tau} \frac{ 2\theta_{\O}|{\V}(\tau)|^{2} \tb(x, \V(s)) }{ |\dot{\V}(s)\cdot\nabla\xi(\xb(x, \V(s)))|} ds    } -1 \Big)^{-\frac{1}{2}} \\
		&\leq \Big[ \int_{\tau_{-}}^{\tau} \frac{2\theta_{\O}|{\V}(\tau)|^{2} \min_{\tau_{-}\leq s\leq \tau} \tb(x, \V(s)) }{ \max_{\tau_{-}\leq s\leq \tau}|\dot{\V}(\tau)\cdot\nabla\xi(\xb( x, \V(s)))| }     \Big]^{-\frac{1}{2}} \\
		&\lesssim_{\O} \frac{  \sqrt{ | {\dot{\V}(\tau_-)}\cdot\nabla\xi(\xb(x, \V(\tau_-)))| } }{ |{\V}(\tau)|\sqrt{\tb(x, \V(\tau))}\sqrt{ \tau-\tau_{-}}} ,
		\end{split}
		\end{equation*}
		where we have used \eqref{min tb x-} and \eqref{mono Cst 1 v}. 
		Hence, from definition \eqref{def:S vel tilde}, we have that for $\tau_{*} \leq \tau_{0}$
		\begin{equation} \label{S vel est}
		\begin{split}
		\int_{\tau_{-}}^{\tau_{*}}  
		\frac{ \dd \tau 	}{\mathfrak{S}_{vel}(\tau; x,v, \bar{v}, \zeta)} &
		\lesssim_{\O} 
		\sqrt{\tb(x, \V(\tau_-))}
		\frac{  \sqrt{ | {\dot{\V}(\tau_-)}\cdot\nabla\xi(\xb(x, \V(\tau_-)))| } }{ |\dot{\V}(\tau_*)| |{\V}(\tau_*)| } 
		\sqrt{\tau_{*}-\tau_{-}}  .
		\end{split}
		\end{equation}
	
		Next we claim that 
		\begin{equation} \label{upper_vn2}
		\begin{split}
			&|\V(\tau)\cdot\nabla\xi(\xb(x, \V(\tau)))|\\
			& \lesssim \sqrt{ {\theta}  |\widehat{\dot{\V}(\tau_{-})}\cdot \nabla\xi(\xb(x, \V(\tau_{-})))| |\V(\tau)|^{2} \Big( |\widehat{\dot{\V}(\tau_{-})}\cdot \nabla\xi(\xb(x, \V(\tau_{-})))| + \tb(x, \V(\tau_-))|\V(\tau)|\Big) } \sqrt{\tau - \tau_{-}} \\
			&\lesssim_{\O} \sqrt{ |\dot\V(\tau)||\V(\tau)| (1 + \tb(x, \V(\tau_-)) | \V(\tau)|)	} \sqrt{\tau - \tau_{-}} .
		\end{split}
	\end{equation}

	From \eqref{S vel est}, \eqref{upper_vn2}, and \eqref{tb minmax}, we can easily get \eqref{int:Sv} for $\tau_* \in [\tau_-, \tau_0]$. For the case $\tau \in [\tau_0, \tau_+]$, we follow the same argument of the last part in \textbf{Step 1}, using \eqref{mono Cst 3 v} and \eqref{tau ratio v}. This finishes the proof.

Now we only need to prove the claim \eqref{upper_vn2}. From \eqref{ODE_v:nabla_xi_v}, \eqref{mono Cst 1 v} and \eqref{min tb v-}, 
 	\begin{equation} \notag %
		\begin{split}
		&\frac{d}{d\tau}(\V(\tau)\cdot\nabla\xi(\xb(x, \V(\tau)))) \\
		&= \dot{\V}(\tau)\cdot\nabla\xi(\xb(x, \V(\tau))) -\tb(x, \V(\tau)) \V(\tau) \cdot \nabla^{2}\xi(\xb(x, \V(\tau)))
		\Big( I - \frac{\V(\tau)\otimes \nabla\xi(\xb(x, \V(\tau)))}{\V(\tau)\cdot\nabla\xi(\xb(x, \V(\tau)))} \Big)\dot{\V}(\tau).  \\
		\end{split}
		\end{equation} 
		Therefore,
		\begin{equation} \notag %
		\begin{split}
		&\frac{1}{2}\frac{d}{d\tau}(\V(\tau)\cdot\nabla\xi(\xb(x, \V(\tau))))^{2} \\
		&= (\dot{\V}(\tau)\cdot\nabla\xi(\xb(x, \V(\tau)))) (\V(\tau)\cdot\nabla\xi(\xb(x, \V(\tau)))) \\
		&\quad -\tb(x, \V(\tau)) ( \V(\tau)\cdot\nabla\xi(\xb(x, \V(\tau)))) \V(\tau)\cdot \nabla^{2}\xi(\xb(x, \V(\tau))) \dot{\V}(\tau) \\
		&\quad + \tb(x, \V(\tau)) (\dot{\V}(\tau)\cdot\nabla\xi(\xb(x, \V(\tau)))) \V(\tau)\cdot \nabla^{2}\xi(\xb(x, \V(\tau))) \V(\tau)  \\
		&\leq C_{\O} {\theta}  |\widehat{\dot{\V}(\tau_{-})}\cdot \nabla\xi(\xb(x, \V(\tau_{-})))| |\V(\tau)|^{2} \Big( |\widehat{\dot{\V}(\tau_{-})}\cdot \nabla\xi(\xb(x, \V(\tau_{-})))| + \tb(x, \V(\tau_-))|\V(\tau)|\Big),
		\end{split}
		\end{equation}
		where we perform the following estimate for the second term,
			\begin{equation*}
			\begin{split}
			&|\tb(x, \V(\tau)) ( \V(\tau)\cdot\nabla\xi(\xb(x, \V(\tau)))) \V(\tau)\cdot \nabla^{2}\xi(\xb(x, \V(\tau))) \dot{\V}(\tau)| \\
			&\leq  C_{\O} \theta \tb(x, \V(\tau_-))|\V(\tau)|^{3} |\widehat{\V(\tau)}\cdot n_{\parallel}(\xb(x, \V(\tau))) |   \\
			&\leq  C_{\O} \theta \tb(x, \V(\tau_-))|\V(\tau)|^{3}  |\widehat{\dot{\V}(\tau_{-})}\cdot \widehat{n_{\parallel}}(\xb(x, \V(\tau_{-}))) |   
			|n_{\parallel}(\xb(x, \V(\tau)))| \\
			&\leq  C_{\O} \theta \tb(x, \V(\tau_-))|\V(\tau)|^{3} |\widehat{\dot{\V}(\tau_{-})}\cdot \nabla\xi(\xb(x, \V(\tau_{-}))) |.  \\
			\end{split}
			\end{equation*}
	Here, we have used the fact $|\dot \V (\tau)| = \theta |\V(\tau)|= \cos^{-1} (\widehat{v+ \zeta} \cdot \widehat{\bar v+ \zeta} )$ in \eqref{perp} and \eqref{unif n} in Lemma \ref{lem_unif n}. 
	
	\hide
		Integrating from $\tau_{-}$ to $\tau$,

		
		Now, combining with \eqref{S vel est} and using \eqref{min tb v-} again,
		\begin{equation} 
		\begin{split}
		&\int_{\tau_{-}}^{\tau_{*}}  
		\frac{ \dd \tau 	}{\mathfrak{S}_{vel}(\tau; x,v, \bar{v}, \zeta)}  \\
		&\lesssim \frac{ (\tau_{*} - \tau_{-}) }{ |\V(\tau_{*})\cdot \nabla\xi(\xb(x, \V(\tau_{*})))| }	
		\frac{1}{|\V(\tau_*)|}
		\max_{\tau_- \leq \tau \leq \tau_*} \{ |\V(\tau)|\tb(x, \V(\tau)) , \sqrt{|\V(\tau)|\tb(x, \V(\tau))}   \}
		\end{split}
		\end{equation}
		Applying \eqref{tb minmax}, we get \eqref{int:Sv} for $\tau_* \in [\tau_-, \tau_0]$.  
		\unhide
	\end{proof}
	
	\hide
	For the sake of briefness we introduce the following notations
	\Be\label{def:V}
	\begin{split}
		\mathcal{V}(t,s, X , V ) 
		=
		\mathcal{V}(t,s,X(s ), V(s ))   := e^{- \int^t_ s \nu(f) (\tau, X(\tau ), V(\tau )) \dd \tau},  \\
		\Ga (t,s, X , V )    := \Gamma_{\text{gain}}(f,f)(s,X(s ), V(s )) .
	\end{split}
	\Ee 
	A solution $f$ to \eqref{f_eqtn} has the Duhamel's formula of  
	\Be\label{f_expan_V}
	\begin{split}
		f(t,x,v+\zeta) &= \mathcal{V}(t,0,X,V) f(0, X(0), V(0)) 
		+  \int^t_0 \mathcal{V}(t,s, X, V) \Ga(t,s,X,V) \dd s,\\ 
		f(t,\bar x,\bar v+\zeta) &= \mathcal{V}(t,0,\bar X,\bar V) f(0, \bar X(0), \bar V(0)) 
		+  \int^t_0 \mathcal{V}(t,s, \bar X, \bar V) \Ga(t,s,\bar X,\bar V) \dd s,
	\end{split}\Ee
	with the notations
	\Be\label{bar X}
	\begin{split}
		( X  ,   V  )=( X (s),   V (s))&= (X(s;t,  x,  v+ \zeta), V(s;t,  x,  v+ \zeta)) ,\\
		( \bar X ,   \bar V )=(\bar X (s),  \bar V (s))&= (X(s;t,\bar x, \bar v+ \zeta), V(s;t,\bar x, \bar v+ \zeta)),  \\
	\end{split}
	\Ee
	
	\subsection{Specular Singularity and $\mathfrak{H}^1$-seminorm}
	In this section we prove Proposition \ref{lem:1/2_to_H}, which relates the H\"older seminorm of the solutions
	\Be\begin{split}\label{def:holder}
		[f(t,\cdot, v)]_{sp}
		&: = \sup_{\substack{
				x, \bar x \in \bar \O \\
				0 < |x-\bar x| \leq 1} }
		\langle v\rangle^{-1}   e^{-\varpi \langle v  \rangle^2 t} 
		\frac{|f(t,x,v) - f(t, \bar x, v)|}{|x-\bar x|^{1/2}},  \\ 
		[f(t,x, \cdot)]_{vel}&: = \sup_{\substack{
				v, \bar v \in \R^3 \\
				0 < |v-\bar v| \leq 1} }
		\langle v\rangle^{-1}  e^{-\varpi \langle v  \rangle^2 t} 
		\frac{|f(t,x,v) - f(t,  x, \bar v)|}{|v-\bar v|^{1/2}},
	\end{split}\Ee
	to a version of the H\"older seminorm of a special moment associated to $\mathbf{k}$ in \eqref{bar k}: 
	\begin{definition} Recall $\mathbf{k}$ in \eqref{bar k}. We define the $\mathfrak{H}^\beta$-seminorms, for $0 \leq \beta \leq 1$ and $\varpi\geq 0$, 
		\Be\label{def:H}
		\begin{split}
			\mathfrak H^\beta_{sp} (t,v)&: =\sup_{\substack{
					x, \bar x \in \bar \O \\
					0 < |x-\bar x| \leq 1} }
			\Big( \int_{\R^3} 
			\mathbf{k}
			(v,   v, u) 
			{\color{white} e^{-\varpi \langle v+u \rangle^2 t}}
			\frac{ |f (t,x,v+u) - f (t,\bar{x}  ,v+u)| }{ |x-\bar{x} |^{\b} } \dd u\Big),\\ 
			\mathfrak H^\beta_{vel}(t,x) &: =
			\sup_{\substack{
					v, \bar v \in \bar \R^3 \\
					0 < |v-\bar v| \leq 1} }
			\Big( \int_{\R^3}   
			\mathbf{k}
			(v,   \bar v, u)  
			{\color{white} e^{-\varpi \langle v+u \rangle^2 t}}
			\frac{ |f (x,v+u) - f ( x  ,\bar v+u)| }{ |v-\bar{v} |^{\b} } \dd u\Big). 
		\end{split} \Ee
		The natural time-dependent weight in $\mathfrak H^\beta_{vel}(t,x) $ would be $e^{-\varpi  \max\{\langle v+u \rangle^2,  \langle \bar v+u \rangle^2 \}t}$. But the difference of these weights is bounded by $|v-\bar v|t$, and therefore the difference of these difference seminorms might be an order of $t \| w f \|_\infty$. For the sake of simplicity we choose the one in \eqref{def:H} and will endure this error.
	\end{definition}

	We list some necessary lemmas to prove Proposition \ref{lem:1/2_to_H}, while proving them are adjourned. 
	
	First we express the difference of Duhamel's formulas of $f$ solving \eqref{f_eqtn} and \eqref{specular}. To count several situations conveniently, including the case that one trajectory hits the boundary while other one misses it, we introduce a following notation:
	\begin{definition}Let $(\bar{x}, \bar{v} ), (x,v) \in \O \times \R^3$
		, and $\zeta \in \R^3$. We define 
		\begin{align}\label{I_int} 
		I_{\textit{int}} &= I_{\textit{int}}(t,x,v, \bar x, \bar v, \zeta)  \\
		&:=\notag
		\begin{cases} 
		\big[0,  t_1(t,\bar x, \bar v+\zeta)   \big] \cup   \big[  {t}_1(t,   x,  v+\zeta) , t\big], 
		\ \  \text{for} \ \ \tb(x, v+ \zeta) \leq  \tb(\bar x, \bar v+ \zeta)<\infty,
		\\
		\big[0,  t_1(t,x,v+\zeta)   \big] \cup   \big[  {t}_1(t, \bar x, \bar v+\zeta) , t\big], 
		\ \   \text{for} \ \ \tb(x, \bar v+ \zeta) \leq  \tb(x,  v+ \zeta)<\infty,\\
		\big[0, t\big] , 
		\ \  \text{otherwise.}  
		\end{cases} 
		\end{align}
	\end{definition}
	In this time interval two characteristics $X(s;t,x,v)$ and $X(s;t,\bar x, \bar v)$ are either both before hitting the boundary or both after hitting it.  \\
	\unhide
	
	\hide
	\begin{lemma}Let $(\bar{x}, \bar{v} ), (x,v) \in \O \times \R^3$. 
		\begin{align}
		& |f(t,x,v+\zeta) - f(t,\bar x,   v+\zeta)| \lesssim |
		f(0,X(0), V(0)) - f(0,\bar X(0),\bar  V(0)) 
		|    
		\notag%
		\\
		& + \mathcal{P}(\| w f \|_\infty) \bigg\{ 
		\mathbf{1}_{\max\{\tb(x, v+\eta), \tb(\bar x, \bar v+\eta)\}<\infty}
		\frac{\langle v+\zeta\rangle  }{w (v +\zeta)}  
		|\tb(x,v+\zeta) - \tb (\bar x,   v+\zeta)|
		\label{diff:f1}
		+  
		\frac{1}{ \langle v +\zeta\rangle }
		\int_{I_{\textit{int}}} 
		|\bar{V}(s) - {V}(s)  |^\beta  \dd s
		\\
		%
		%
		%
		%
		& \  \ \ \ \ \ \ \ \ \ \ \  \ \ \ \  \ \ \   + 
		\int_{I_{\textit{int}}} |X(s)-\bar{X}(s)|^{ \b} \int_{\R^3} 
		\mathbf{k}
		(V(s),   V(s), u)   \frac{ |f (X(s),V(s)+u) - f (\bar{X}(s),V(s)+u)| }{ |X(s)-\bar{X}(s)|^{ \b} } \dd u \dd s
		\label{diff:f4}\\
		& \  \ \ \ \ \ \ \ \ \ \ \  \ \ \ \  \ \ \   + \int_{I_{\textit{int}}}
		|V(s)- \bar V(s)|^\beta \int_{\R^3} 
		\mathbf{k}
		(V(s), \bar V(s), u)   
		\frac{ | f(\bar X(s), V(s)+u) - f(\bar X(s), \bar{V}(s)+u) | }{ |V(s)-\bar{V}(s)|^{\b} }  \dd u \dd s
		\bigg\},
		\label{diff:f5}
		\end{align}

	\end{lemma}
	\unhide

\section{$\mathfrak{H}^{2\b}_{sp, vel}$ estimates} 

\subsection{Difference estimates}
When \eqref{assume_x} with $v+\zeta$ and \eqref{assume_v} hold, we split
\begin{align}
& f(s, X(s;t,x,v+\zeta), V(s;t,x,v+\zeta)) - f(s, X(s;t, \bar{x}, \bar{v}+\zeta), V(s;t, \bar{x}, \bar{v}+\zeta))  \label{diff f}  \\
&\leq  f(s, X(s;t,x,v+\zeta), V(s;t,x,v+\zeta)) - f(s, X(s;t, \tilde{x}, v+\zeta ), V(s;t, \tilde{x}, v+\zeta ))   \label{split1}   \\
&\quad + f(s, X(s;t, \tilde{x}, v+\zeta ), V(s;t, \tilde{x}, v+\zeta ))  - f(s, X(s;t, \bar{x}, v+\zeta ), V(s;t, \bar{x}, v+\zeta ))    \label{split2}   \\
&\quad + f(s, X(s;t, \bar{x}, v+\zeta ), V(s;t, \bar{x}, v+\zeta )) - f(s, X(s;t, \bar{x}, \tv+\zeta ), V(s;t, \bar{x}, \tv+\zeta ))   \label{split3}    \\
&\quad + f(s, X(s;t, \bar{x}, \tv+\zeta ), V(s;t, \bar{x}, \tv+\zeta ))   - f(s, X(s;t, \bar{x}, \bv+\zeta ), V(s;t, \bar{x}, \bv+\zeta )),   \label{split4}        
\end{align}
\hide
\begin{align}
& f(s, X(s;t,x,v), V(s;t,x,v)) - f(s, X(s;t, \bar{x}, \bar{v}), V(s;t, \bar{x}, \bar{v}))  \\
&\leq  f(s, X(s;t,x,v), V(s;t,x,v)) - f(s, X(s;t, \tilde{x}, v ), V(s;t, \tilde{x}, v ))      \\
&\quad + f(s, X(s;t, \tilde{x}, v ), V(s;t, \tilde{x}, v ))  - f(s, X(s;t, \bar{x}, v ), V(s;t, \bar{x}, v ))       \\
&\quad + f(s, X(s;t, \bar{x}, v ), V(s;t, \bar{x}, v )) - f(s, X(s;t, \bar{x}, \tilde{v} ), V(s;t, \bar{x}, \tilde{v} ))       \\
&\quad + f(s, X(s;t, \bar{x}, \tilde{v} ), V(s;t, \bar{x}, \tilde{v} ))   - f(s, X(s;t, \bar{x}, \bar{v} ), V(s;t, \bar{x}, \bar{v} ))       
\end{align}
\unhide
where $\tilde{x} = \tx(x, \bx, v+\zeta)$ and $\tv = \tv(v, \bv, \zeta)$ are defined in Definition \ref{def_tilde}. We estimate each \eqref{split1}--\eqref{split4}.  \\

\begin{lemma} \label{f-f split}
	Let $f : \overline{\O}\times \R^{3} \rightarrow \R_{+} \cup \{0\}$ be a function which satisfies specular reflection \eqref{specular}, where $\O$ is a domain as in Definition \ref{def:domain}. Let $w(v) = e^{\vartheta|v|^{2}}$ for some $0 < \vartheta$.  \\
	If \eqref{assume_x} holds with $v+\zeta$, then \eqref{split1} and \eqref{split2} enjoy the following estimates.
	\begin{eqnarray}
	\frac{\eqref{split1}}{ |x - \bx|^{\g} } &\lesssim&    \Big[ 1 + |v+\zeta|(t-s)  + ( |v+\zeta| + |v+\zeta|^{2}(t-s) ) \mathcal{T}_{sp}(x, \tx, v+\zeta)
	\Big]^{2\b}  
	\notag  \\
	&&\quad 
	\times  
		\Big[ \frac{ e^{\varpi \langle v+\zeta \rangle^2 s}  }{ \langle v+\zeta \rangle^{s_{1}}}  
		\sup_{ \substack{ v\in \R^{3} \\ 0 < |x - \bx|\leq 1   } } e^{-\varpi \langle v \rangle^2 s}  
		\langle v \rangle^{s_{1}} \frac{|f(s, x, v ) - f(s, \bx, v)|}{|x - \bx|^{\g}}  + \frac{\|w f(s)\|_{\infty}}{w(v+\zeta)} \Big] 
	\notag  \\
	&& + \Big[1 + |v+\zeta| +  (|v+\zeta| + |v+\zeta|^{2}) \mathcal{T}_{sp}(x, \tx, v+\zeta)
	\Big]^{2\b}  \notag \\
	&&\quad
	\times
	\Big[ 
		\frac{ e^{\varpi \langle v+\zeta \rangle^2 s} }{ \langle v+\zeta \rangle^{s_{2}} }	
		\sup_{ \substack{ x\in \overline{\O} \\ 0 < |v - \bv|\leq 1   } } e^{\varpi \langle v \rangle^2 s}  \langle v \rangle^{s_{2}} \frac{|f(s, x, v ) - f(s, x, \bv)|}{|v - \bv|^{\g}} + \frac{\|wf(s)\|_{\infty}}{w(v+\zeta)}   \Big],
	 \label{est:split1} \\
	\frac{\eqref{split2}}{ |x - \bx|^{\g} }  &\lesssim& \Big[ \frac{e^{\varpi \langle v+\zeta \rangle^2}s}{\langle v+\zeta \rangle^{s_{1}}}  \sup_{ \substack{ v\in \R^{3} \\ 0 < |x - \bx|\leq 1   } } \Big( e^{-\varpi \langle v \rangle^2 s} \langle v \rangle^{s_{1}} \frac{|f(s, x, v ) - f(s, \bar{x}, v)|}{|x - \bx|^{\g}} \Big) 
	+ \frac{\|wf(s)\|_{\infty}}{w(v+\zeta)}  \Big] .     \label{est:split2}  
	\end{eqnarray}
	Similarly, if \eqref{assume_v} holds, then \eqref{split3} and \eqref{split4} enjoy the following estimates.
	\begin{eqnarray}
	\frac{\eqref{split3}}{|v-\bv|^{\g}} &\lesssim& 
	\Big[ (t-s) + |v+\zeta|(t-s)^{2} + ( |v+\zeta| + |v+\zeta|^{2}(t-s) ) \mathcal{T}_{vel}(\bx, v, \tv, \zeta; t, s)
	\Big]^{2\b}  
	\notag  \\
	&&\quad\quad 
	\times  
	\Big[ \frac{ e^{\varpi \langle v+\zeta \rangle^2 s}  }{ \langle v+\zeta \rangle^{s_{1}}}  
		\sup_{ \substack{ v\in \R^{3} \\ 0 < |x - \bx|\leq 1   } }  e^{-\varpi \langle v \rangle^2 s}  
		\langle v \rangle^{s_{1}} \frac{|f(s, x, v ) - f(s, \bx, v)|}{|x - \bx|^{\g}}  + \frac{\|wf(s)\|_{\infty}}{w(v+\zeta)}  \Big] 
	\notag  \\
	&& + \Big[1 + |v+\zeta|(t-s) +  |v+\zeta|^{2} \mathcal{T}_{vel}(\bx, v, \tv, \zeta; t,s)
	\Big]^{2\b}  \notag \\
	&&\quad\quad 
	\times
	\Big[ 
		\frac{ e^{\varpi \langle v+\zeta \rangle^2 s} }{ \langle v+\zeta \rangle^{s_{2}} }	
		\sup_{ \substack{ x\in \overline{\O} \\ 0 < |v - \bv|\leq 1   } } e^{-\varpi \langle v \rangle^2 s}  \langle v \rangle^{s_{2}} \frac{|f(s, x, v ) - f(s, x, \bv)|}{|v - \bv|^{\g}} + \frac{\|wf(s)\|_{\infty}}{w(v+\zeta)}   \Big],
	\label{est:split3}   \\
	\frac{\eqref{split4}}{|v-\bv|^{\g}} &\lesssim& \Big[ \frac{ e^{\varpi \langle v+\zeta \rangle^2 s}  }{\langle v+\zeta \rangle^{s_{1}}}  \sup_{ \substack{ v\in \R^{3} \\ 0 < |x - \bx|\leq 1   } } \Big( e^{-\varpi \langle v \rangle^2 s}  \langle v \rangle^{s_{1}} \frac{|f(s, x, v ) - f(s, \bar{x}, v)|}{|x - \bx|^{\g}} \Big) + \frac{\|w f(s)\|_{\infty}}{w(v+\zeta)}   \Big]   (t-s)^{\g} 	\notag   \\
	&&\quad + \Big[ \frac{ e^{\varpi \langle v+\zeta \rangle^2 s}  }{\langle v +\zeta\rangle^{s_{2}}}  \sup_{ \substack{ x\in \overline{\O} \\ 0 < |v - \bv|\leq 1   } } \Big( e^{-\varpi \langle v \rangle^2 s}  \langle v \rangle^{s_{2}} \frac{|f(s, x, v ) - f(s, x, \bv)|}{|v - \bv|^{\g}} \Big)
	+ \frac{\|w f(s)\|_{\infty}}{w(v+\zeta)}   \Big] . 	\label{est:split4}  
	\end{eqnarray}
	Here $\mathcal{T}_{sp}$ and $\mathcal{T}_{vel}$ are defined as
	\Be \label{def_Tsp}
	\begin{split}
		\mathcal{T}_{sp}(x, \tx, v+\zeta)  &:= \Big[  \fint_{0}^{1} \frac{1}{\mathfrak{S}_{sp}(\tau; x, \tx, v+\zeta)} d\tau  
		\mathbf{1}_{ \{  
			\tb(\X(\tau), v+\zeta) < \infty, \ 0\leq \tau \leq 1  
			\} } 
		\\
		&\quad +  \fint_{\tau_{-}}^{1} \frac{1}{\mathfrak{S}_{sp}(\tau; x, \tx, v+\zeta)}   d\tau  
		\mathbf{1}_{  \{  
			\tb( \X(\tau), v+\zeta) < \infty, \ \tau_{-}\leq \tau \leq 1  
			\} } 
		\\
		&\quad +  \fint_{0}^{\tau_{-}} \frac{1}{\mathfrak{S}_{sp}(\tau; x, \tx, v+\zeta)}  d\tau
		\mathbf{1}_{  \{  
			\tb( \X(\tau), v+\zeta) < \infty, \ 0\leq \tau \leq \tau_{-}  
			\} } \Big]  \mathbf{1}_{\zeta\in P^{1}_{(x, \bx, v+\zeta)} },
	\end{split}
	\Ee 
	where 
	\Be \notag 
		P^{1}_{(x, \bx, v+\zeta)} := \{\zeta\in\R^{3} : \text{\eqref{assume_x} and \eqref{assume_x2} with $S_{(x, \bx, v+\zeta)}$ hold} \},
	\Ee
	and
	\Be \label{def_Tvel}
	\begin{split}
		\mathcal{T}_{vel}(\bx, v, \tv, \zeta;t,s)  &:=  \Big[ \fint_{0}^{1} \frac{1}{\mathfrak{S}_{vel}(\tau; \bx, v, \tv, \zeta)} d\tau  
		\mathbf{1}_{ \Big\{ \substack{
				\tb(\bx, \V(\tau)) < \infty, \ 0\leq \tau \leq 1  
				\\ 
				\min_{0\leq \tau \leq 1}\tb(\bx, \V(\tau)) \leq t-s }  \Big\} } 
		\\
		&\quad 
		+  \fint_{\tau_{-}}^{1} \frac{1}{\mathfrak{S}_{vel}(\tau; \bx, v, \tv, \zeta)} d\tau  
		\mathbf{1}_{ \Big\{ \substack{
				\tb(\bx, \V(\tau)) < \infty, \ \tau_{-}\leq \tau \leq 1  
				\\ 
				\min_{\tau_{-} \leq \tau \leq 1}\tb(\bx, \V(\tau)) \leq t-s  }  \Big\} } 
		\\
		&\quad +  \fint_{0}^{\tau_{-}} \frac{1}{\mathfrak{S}_{vel}(\tau; \bx, v, \tv, \zeta)} d\tau
		\mathbf{1}_{ \Big\{ \substack{
				\tb(\bx, \V(\tau)) < \infty, \ 0\leq \tau \leq \tau_{-}  
				\\ 
				\min_{0\leq \tau \leq \tau_{-}  }\tb(\bx, \V(\tau)) \leq t-s  }  \Big\} }  \Big]  \mathbf{1}_{\zeta\in P^{2}_{(\bx, v, \bv, \zeta)} },
	\end{split}
	\Ee  
	where 
	\Be \notag 
	P^{2}_{(\bx, v, \bv, \zeta)} := \{\zeta\in\R^{3} : \text{\eqref{assume_v} and \eqref{assume_v2} with $S_{(\bx, v, \bv, \zeta)}$ hold} \}.
	\Ee
	Here, we used notation, $\fint_{a}^{b} := \frac{1}{b-a} \int_{a}^{b}$.   \\
\end{lemma}

\begin{remark}
	In the definition $\mathcal{T}_{vel}$ \eqref{def_Tvel}, we put extra condition of the following type
	\[
	\min_{0\leq \tau \leq 1}\tb(x, \V(\tau)) \leq t-s,
	\]
	which is missing in the definition of $\mathcal{T}_{sp}$ in \eqref{def_Tsp}. In fact, we can put similar condition in the definition of $\mathcal{T}_{sp}$ also. However, above condition is not important for $\mathcal{T}_{sp}$ case as we see in \eqref{int:Sx} : unlike to \eqref{int:Sv}, it does not contain any $\tb(x, \V(\tau))$-related terms. \\
\end{remark}	

\begin{proof}
	In this proof, we drop $\zeta$ in \eqref{split1}--\eqref{split4} for notational convenience. \\
	{\bf Step 0} (Trivial dynamics) Assume \eqref{assume_x2} with $S_{(x, \bx, v)}$ ($S_{(x, \bx, v+\zeta)}$ in fact) and \eqref{assume_v2} with $S_{(\bx, v, \bv, 0)}$ ($S_{(\bx, v, \bv, \zeta)}$ in fact) do not hold. (The definition of $S_{(x, \bx, v)}$ and $S_{(\bx, v, \bv, 0)}$ are given in \eqref{S_x}, \eqref{S_v}.) Then, all the backward in time trajectories do not hit $\p\O$, hence \eqref{est:split1}--\eqref{est:split4} hold obviously, using the following trivial trajectory estimates,
	\Be \notag 
	\begin{split}
		\frac{ | X(s; t, x, v) - X(s; t, \tx, v)| }{ |x - \tx| } &= 1, \quad 
		\frac{ | X(s; t, \tx, v) - X(s; t, \bx, v)| }{ |\tx - \bx| } = 1,  \\
		\frac{ | V(s; t, x, v) - V(s; t, \tx, v)| }{ |x - \tx| } &= 0, \quad 
		\frac{ | V(s; t, \tx, v) - V(s; t, \bx, v)| }{ |\tx - \bx| } = 0,  \\
		\frac{ | X(s; t, \bx, v) - X(s; t, \bx, \tv)| }{ |v - \tv| } &= (t-s), \quad 
		\frac{ | X(s; t, \bx, \tv) - X(s; t, \bx, \bv)| }{ |\tv - \bv| } = (t-s),  \\
		\frac{ | V(s; t, \bx, v) - V(s; t, \bx, \tv)| }{ |v - \tv| } &= 1, \quad 
		\frac{ | V(s; t, \bx, \tv) - V(s; t, \bx, \bv)| }{ |\tv - \bv| } = 1.  \\
	\end{split}
	\Ee
	We omit the details. \\
	
	Now let us consider nontrivial cases. In the following \textbf{Step 1} and \textbf{Step 2}, we assume \eqref{assume_x2} with $S_{(x, \bx, v)}$ ($S_{(x, \bx, v+\zeta)}$ in fact) and \eqref{assume_v2} with $S_{(\bx, v, \bv, 0)}$ ($S_{(\bx, v, \bv, \zeta)}$ in fact), in addition to \eqref{assume_x} and \eqref{assume_v}.  \\
	
	\noindent {\bf Step 1} (Nonsingular parts) Let us treat \eqref{split2} first. Since $\tx-\bx$ is parallel to $v$ ($v+\zeta$ in fact), we do not see any specular singularity in \eqref{split2}. \\

	\begin{equation} \label{est:split2 raw}
	\begin{split}
	\eqref{split2} &\leq  \frac{ |f(s, X(s;t, \tilde{x}, v ), V(s;t, \tilde{x}, v ))  - f(s, X(s;t, \bar{x}, v ), V(s;t, \bar{x}, v ))| }{ | X(s;t, \tilde{x}, v ) -  X(s;t, \bar{x}, v )|^{\gamma} } | X(s;t, \tilde{x}, v ) -  X(s;t, \bar{x}, v )|^{\gamma}  \\
	&\quad \quad \times \big( \mathbf{1}_{s > \max\{  t^{1}(\tilde{x},v),  t^{1}(\bar{x},v) \}} + \mathbf{1}_{s \leq \min\{ t^{1}(\tilde{x},v), t^{1}(\bar{x},v) \}} \big)  \\
	&\quad + |f(s, X(s;t, \tilde{x}, v ), V(s;t, \tilde{x}, v ))  - f(s, X(s;t, \bar{x}, v ), V(s;t, \bar{x}, v ))|  \\
	&\quad \quad \times  \mathbf{1}_{\min\{  t^{1}(\tilde{x},v),  t^{1}(\bar{x},v) \}< s \leq \max\{  t^{1}(\tilde{x},v),  t^{1}(\bar{x},v) \}}   \\
	&\leq  \frac{ |f(s, X(s;t, \tilde{x}, v ), V(s;t, \tilde{x}, v ))  - f(s, X(s;t, \bar{x}, v ), V(s;t, \bar{x}, v ))| }{ | X(s;t, \tilde{x}, v ) -  X(s;t, \bar{x}, v )|^{\gamma} } | \tilde{x} - \bar{x} |^{\gamma}  \\
	&\quad \quad \times \big( \mathbf{1}_{s > \max\{  t^{1}(\tilde{x},v),  t^{1}(\bar{x},v) \}} + \mathbf{1}_{s \leq \min\{  t^{1}(\tilde{x},v),  t^{1}(\bar{x},v) \}} \big)  \\
	&\quad +  \big( |f(s, X(s;t, \tilde{x}, v ), v )  - f(s, \xb(\bar{x},v), v)|   \\
	&\quad\quad  + |f(s, \xb(\bar{x},v), R_{\xb(\bar{x},v)}v )  - f(s, X(s;t, \bar{x}, v ), R_{\xb(\bar{x},v)}v )|  \big) \times  \mathbf{1}_{  t^{1}(\tilde{x},v) < s \leq  t^{1}(\bar{x},v) }  \\
	&\quad +  \big( |f(s, X(s;t, \tilde{x}, v ), R_{\xb(\tilde{x},v)}v )  - f(s, \xb(\tilde{x},v), R_{\xb(\tilde{x},v)}v )|   \\
	&\quad\quad  + |f(s, \xb(\tilde{x},v), v)  - f(s, X(s;t, \bar{x}, v ), v )|  \big) \times  \mathbf{1}_{  t^{1}(\bar{x},v) < s \leq  t^{1}(\tilde{x},v) }  \\
	&\leq  \frac{ |f(s, X(s;t, \tilde{x}, v ), V(s;t, \tilde{x}, v ))  - f(s, X(s;t, \bar{x}, v ), V(s;t, \bar{x}, v ))| }{ | X(s;t, \tilde{x}, v ) -  X(s;t, \bar{x}, v )|^{\gamma} } | \tilde{x} - \bar{x} |^{\gamma}  \\
	&\quad \quad \times \big( \mathbf{1}_{s > \max\{  t^{1}(\tilde{x},v),  t^{1}(\bar{x},v) \}} + \mathbf{1}_{s \leq \min\{  t^{1}(\tilde{x},v),  t^{1}(\bar{x},v) \}} \big)  \\
	&\quad +  \Big( \frac{ |f(s, X(s;t, \tilde{x}, v ), v )  - f(s,  \xb(\bar{x},v), v)|  }{ |X(s;t, \tilde{x}, v ) - \xb(\bar{x},v)|^{\gamma} }  
	\\
	&\quad\quad\quad + \frac{ |f(s, \xb(\bar{x},v), R_{\xb(\bar{x},v)}v )  - f(s, X(s;t, \bar{x}, v ), R_{\xb(\bar{x},v)}v )| }{ |\xb(\bar{x},v) - X(s;t, \bar{x}, v )|^{\gamma} } 
	\Big) 
	|\bar{x}-\tilde{x}|^{\gamma} \mathbf{1}_{  t^{1}(\tilde{x},v) < s \leq  t^{1}(\bar{x},v) }  \\
	&\quad +  \Big( \frac{ |f(s, X(s;t, \tilde{x}, v ), R_{\xb(\tilde{x},v)}v )  - f(s,  \xb(\tilde{x},v), R_{\xb(\tilde{x},v)}v )| }{|X(s;t, \tilde{x}, v ) - \xb(\tilde{x},v)|^{\gamma}} \\
	&\quad\quad\quad
	+ \frac{|f(s, \xb(\tilde{x},v), v)  - f(s, X(s;t, \bar{x}, v ), v )| }{|\xb(\tilde{x},v) - X(s;t, \bar{x},v) |^{\gamma}} \Big) 
	|\bar{x}-\tilde{x}|^{\gamma} \mathbf{1}_{  t^{1}(\bar{x},v) < s \leq  t^{1}(\tilde{x},v) },  \\
	\end{split}
	\end{equation}
	where we used the facts that $(\tx - \bx) \parallel v$ ($v+\zeta$ in fact) and $|\tx - \bx| \leq |x - \bx|$ by \eqref{def_tildex}. When a denominator is larger than $1$,  the second term of the RHS in \eqref{est:split2} controls LHS. Therefore, \eqref{est:split2 raw} gives \eqref{est:split2}. 	\\
	
	Let us treat \eqref{split4}. Since $\bv$ and $\tv$ are parallel to each other ($\bv+\zeta$ and $\tv+\zeta$ in fact), we do not see any specular singularity in \eqref{split4}, neither. \\
	\begin{equation} \label{est:split4 raw}
	\begin{split}
	\eqref{split4} 
	&\leq  \Big( \frac{ |f(s, X(s;t, \bar{x}, \tilde{v} ), V(s;t, \bar{x}, \tilde{v} ))  - f(s, X(s;t, \bar{x}, \tilde{v} ), V(s;t, \bar{x}, \bar{v} ))| }{ | V(s;t, \bar{x}, \tilde{v} ) -  V(s;t, \bar{x}, \bar{v} )|^{\g} }
	|\bar{v}-\tilde{v}|^{\gamma}  \\
	&\quad + \frac{ | f(s, X(s;t, \bar{x}, \tilde{v} ), V(s;t, \bar{x}, \bar{v} )) - f(s, X(s;t, \bar{x}, \bar{v} ), V(s;t, \bar{x}, \bar{v} ))| }{ | X(s;t, \bar{x}, \tilde{v} ) -  X(s;t, \bar{x}, \bar{v} )|^{\g} }
	|(\bar{v}-\tilde{v})(t-s) |^{\gamma}  \Big) \\
	&\quad \quad \times \big( \mathbf{1}_{s > \max\{  t^{1}(\bar{x}, \tilde{v}),  t^{1}(\bar{x}, \bar{v}  ) \}} + \mathbf{1}_{s \leq \min\{ t^{1}(\bar{x}, \tilde{v}), t^{1}(\bar{x}, \bar{v} ) \}} \big)   \\
	&\quad +  \Big( \frac{ |f(s, X(s;t, \bar{x}, \tilde{v} ), \tilde{v} )  - f(s, \xb(\bar{x},\bar{v}), \tilde{v})|  }{ |X(s;t, \bar{x}, \tilde{v} ) - \xb(\bar{x},\bar{v})|^{\gamma} } |(\bar{v} - \tilde{v})(t-s)|^{\gamma}  \\
	&\quad\quad  + \frac{ |f(s, \xb(\bar{x},\bar{v}), R_{\xb(\bar{x},\bar{v})}\tilde{v} )  - f(s, \xb(\bar{x},\bar{v}), R_{\xb(\bar{x},\bar{v})}\bv )| }{ |R_{\xb(\bar{x},\bar{v})}\tilde{v} - R_{\xb(\bar{x},\bar{v})}\bv|^{\gamma} } |\bv - \tv|^{\gamma}   \\
	&\quad\quad  + \frac{ |f(s, \xb(\bar{x},\bar{v}), R_{\xb(\bar{x},\bar{v})}\bv )  - f(s, X(s;t, \bar{x}, \bar{v} ), R_{\xb(\bar{x},\bar{v})}\bar{v} )| }{ |\xb(\bar{x},\bar{v}) - X(s;t, \bar{x}, \bar{v} )|^{\gamma} } |(\bar{v} - \tilde{v})(t-s)|^{\gamma} \Big) \mathbf{1}_{  t^{1}(\bar{x}, \tilde{v}) < s \leq  t^{1}(\bar{x},\bar{v}) }  \\ 
	&\quad +  \Big( \frac{ |f(s, X(s;t, \bar{x}, \tilde{v} ), R_{\xb(\bar{x}, \tilde{v})}\tilde{v} )  - f(s, \xb(\bar{x}, \tilde{v}), R_{\xb(\bar{x}, \tilde{v})}\tilde{v} )| }{|X(s;t, \bar{x}, \tilde{v} ) - \xb(\bar{x}, \tilde{v})|^{\gamma}} |(\bar{v} - \tilde{v})(t-s)|^{\gamma}   \\
	&\quad\quad  + \frac{ |f(s, \xb(\bar{x}, \tilde{v}), \tilde{v} )  - f(s, \xb(\bar{x}, \tilde{v}), \bar{v} ) | }{|\bar{v} - \tilde{v}|^{\gamma}}
	|\bar{v} - \tilde{v}|^{\gamma}   \\
	&\quad\quad  + \frac{ |f(s, \xb(\bar{x}, \tilde{v}), \bar{v} ) - f(s, X(s;t, \bar{x}, \bar{v} ), \bar{v} ) | }{|\xb(\bar{x}, \tilde{v}) - X(s;t, \bar{x}, \bar{v} ) |^{\gamma}}
	|(\bar{v} - \tilde{v})(t-s)|^{\gamma} \Big) \mathbf{1}_{  t^{1}(\bar{x},\bv) < s \leq  t^{1}(\bx,\tv) },  \\
	\end{split}
	\end{equation}
	where we used specular boundary condition \eqref{specular}, $|\tv - \bv| \lesssim |v - \bv|$ by \eqref{def_tildev}. So we get \eqref{est:split4} from \eqref{est:split4 raw}. \\
	
	\noindent {\bf Step 2} (Singular parts)
	Now, we treat main contributions :\eqref{est:split1} and \eqref{est:split3}. We will see specular singularity in these estimates. Using specular condition \eqref{specular}, we split \eqref{split1} into
	\begin{equation} \label{est:split1 raw}
	\begin{split}
	\eqref{split1} 
	&\leq \Big( \frac{ |f(s, X(s;t,x,v), V(s;t,x,v)) - f(s, X(s;t, \tilde{x}, v ), V(s;t, x, v ))| }{ | X(s;t, x, v ) -  X(s;t, \tx, v )|^{\gamma} } \\
	&\quad\quad\quad \times | X(s;t, x, v ) -  X(s;t, \tx, v )| ^{\gamma} 
	\mathbf{1}_{s \leq \min\{ t^{1}(x,v), t^{1}(\tilde{x},v) \}} \Big)  \\
	&\quad + \Big( \frac{ |f(s, X(s;t, \tx, v), V(s;t,x,v)) - f(s, X(s;t, \tilde{x}, v ), V(s;t, \tilde{x}, v ))| }{ | V(s;t, x, v ) -  V(s;t, \tx, v )|^{\gamma} } \\
	&\quad\quad\quad\quad \times  | V(s;t, x, v ) -  V(s;t, \tx, v )| ^{\gamma} 
	\mathbf{1}_{s \leq \min\{ t^{1}(x,v), t^{1}(\tilde{x},v) \}} \Big) \\
	&\quad + \frac{ |f(s, X(s;t,x,v), V(s;t,x,v)) - f(s, X(s;t, \tilde{x}, v ), V(s;t, \tilde{x}, v ))| }{ | X(s;t, x, v ) -  X(s;t, \tx, v )|^{\gamma} } | x - \tx |^{\gamma}  
	\mathbf{1}_{s > \max\{  t^{1}(x,v),  t^{1}(\tilde{x},v) \}}  \\
	&\quad + A_{1} + A_{2} + A_{3} + A_{4} \\
	&\quad + B_{1} + B_{2} + B_{3} + B_{4} ,\\
	\end{split}
	\end{equation}
	where
	\Be \notag 
	\begin{split}
		A_{1} &:=  \Big( \frac{ |f(s, X(s;t, x, v ), R_{\xb(x,v)}v )  - f(s,  \xb(x,v), R_{\xb(x,v)}v)|  }{ |X(s;t, x, v ) - \xb(x,v)|^{\gamma} }  | v(t^{1}(x,v) - s)|^{\gamma} \\
		&\quad\quad + \frac{ |f(s, \xb(x,v),  v )  - f(s, X(s;t, \tx, v ),  v )| }{ |\xb(x,v) - X(s;t, \tilde{x}, v )|^{\gamma} } ( |x-\tilde{x}| + |v|(t^{1}(x,v) - s)  )^{\gamma}  \Big)  \\
		&\quad\quad\quad\times \mathbf{1}_{ t^{1}(\tilde{x},v) < s \leq  t^{1}(x,v) } \mathbf{1}_{t^{1}(\tx,v) > -\infty} ,  \\
	\end{split}
	\Ee
	\Be\notag 
	\begin{split}
		A_{2} &:=  \Big( \frac{ |f(s, X(s;t, x, v ), V(s;t, x, v ) )  -   f(s, X(s;t, \X(\tau_{-}) , v ),  V(s;t, x, v) ) |  }{ |X(s;t, x, v ) - X(s;t, \X(\tau_{-}) , v )|^{\gamma} }  \\
		&\quad\quad\quad \times  |X(s;t, x, v ) - X(s;t, \X(\tau_{-}) , v )|^{\gamma}    \\
		&\quad\quad +   \frac{ | f(s, X(s;t, x(\tau_{-}) , v ),  V(s;t, x, v) )   - f(s, X(s;t, \X(\tau_{-}) , v ),  V(s;t, \X(\tau_{-}) , v) )| }{ |V(s;t, x, v) - V(s;t, \X(\tau_{-}) , v) |^{\gamma} } \\
		&\quad\quad\quad\quad \times |V(s;t, x, v) - V(s;t, \X(\tau_{-}) , v) |^{\gamma} \Big) 
		\mathbf{1}_{  s\leq \min\{t^{1}(\X(\tau_{-}),v), t^{1}(x,v)\} }  \mathbf{1}_{t^{1}(\tx,v) = -\infty} ,  \\
	\end{split}
	\Ee
	\Be\notag 
	\begin{split}
		A_{3} &:=  \Big( \frac{ |f(s, X(s;t, x, v ), R_{\xb(x,v)}v )  -   f(s, \xb(x,v), R_{\xb(x,v)}v ) |  }{ |X(s;t, x, v ) - \xb(x,v)|^{\gamma} }   |v(t^{1}(x,v) - s)|^{\gamma}     \\
		&\quad\quad +   \frac{ | f(s, \xb(x,v), v)  - f(s, X(s;t, \X(\tau_{-}) , v ),  V(s;t, \X(\tau_{-}) , v) )| }{ |V(s;t, x, v) - V(s;t, \X(\tau_{-}) , v) |^{\gamma} } \\
		&\quad\quad\quad \times \big( |x - \X(\tau_{-})| + |v|(t^{1}(x,v) - s)  \big)^{\gamma}  \Big)  
		\mathbf{1}_{  t^{1}(\X(\tau_{-}),v) < s \leq  t^{1}(x,v) }  \mathbf{1}_{t^{1}(\tx,v) = -\infty} ,  \\
	\end{split}
	\Ee
	\Be \notag 
	\begin{split}
		A_{4} &:= \frac{ |f(s, X(s;t, \X(\tau_{-}) , v ),  v )  - f(s, X(s;t, \tx, v ),  v )| }{ |X(s;t, \X(\tau_{-}) , v ) - X(s;t, \tilde{x}, v )|^{\gamma} } \\
		&\quad\quad \times |\X(\tau_{-}) -\tilde{x}|^{\gamma} \times  \mathbf{1}_{  t^{1}(\tilde{x},v) < s \leq  t^{1}(x,v) }  \mathbf{1}_{t^{1}(\tx,v) = -\infty},   \\
	\end{split}
	\Ee
	and
	\Be \notag 
	\begin{split}
		B_{1,2,3,4} &:= \text{interchanging $x$ and $\tx$ in $A_{1,2,3,4}$ },  \ \ \text{respectively}.  \\
	\end{split}
	\Ee
	Applying \eqref{est:V/x} and \eqref{est:X/x} in Lemma \ref{frac sim S}, and \eqref{est:Lx} in Lemma \ref{L sim S}, to \eqref{est:split1 raw}, we obtain \eqref{est:split1}.    \\
	
	\hide
	\begin{eqnarray} \notag
	&&\frac{\eqref{split1}}{|x-\bx|^{\g}}  \\
	&&\lesssim \Big[ 1 + |v+\zeta|t  + ( |v+\zeta| + |v+\zeta|^{2}t ) \mathcal{T}_{sp}(x, \tx, v+\zeta)
	\Big]^{2\b}  
	\notag  \\
	&&\quad\quad 
	\times  
	{\color{blue}  \Big[ \frac{ e^{\varpi \langle v+\zeta \rangle^2 s}  }{ \langle v+\zeta \rangle^{s_{1}}}  
		\sup_{v, |x - \bx|\leq 1} e^{-\varpi \langle v \rangle^2 s}  
		\langle v \rangle^{s_{1}} \frac{|f(s, x, v ) - f(s, \bx, v)|}{|x - \bx|^{2\b}}  + \frac{1}{w(v+\zeta)}\|w f(s)\|_{\infty} \Big] 
	}      \notag  \\
	&& + \Big[1 + |v+\zeta| +  (|v+\zeta| + |v+\zeta|^{2}) \mathcal{T}_{sp}(x, \tx, v+\zeta)
	\Big]^{2\b}  \notag \\
	&&\quad\quad 
	\times
	{\color{blue}  \Big[ 
		\frac{ e^{\varpi \langle v+\zeta \rangle^2 s} }{ \langle v+\zeta \rangle^{s_{2}} }	
		\sup_{x, |v - \bv|\leq 1} e^{\varpi \langle v \rangle^2 s}  \langle v \rangle^{s_{2}} \frac{|f(s, x, v ) - f(s, x, \bv)|}{|v - \bv|^{2\b}} + \frac{1}{w(v+\zeta)} \|wf(s)\|_{\infty}  \Big]
	}    \notag  
	\end{eqnarray}
	where we defined
	
	\Be  
	\begin{split}
		\mathcal{T}_{sp}(x, \tx, v+\zeta)  &:=  \fint_{0}^{1} \frac{1}{\mathfrak{S}_{sp}(\tau; x, \tx, v+\zeta)} d\tau  
		\mathbf{1}_{ \{  
				\tb(x, v(\tau)) < \infty, \ 0\leq \tau \leq 1  
				\} } 
		\\
		&\quad +  \fint_{\tau_{-}}^{1} \frac{1}{\mathfrak{S}_{sp}(\tau; x, \tx, v+\zeta)}   d\tau  
		\mathbf{1}_{  \{  
				\tb(x, v(\tau)) < \infty, \ \tau_{-}\leq \tau \leq 1  
				\} } 
		\\
		&\quad +  \fint_{0}^{\tau_{-}} \frac{1}{\mathfrak{S}_{sp}(\tau; x, \tx, v+\zeta)}  d\tau
		\mathbf{1}_{  \{  
				\tb(x, v(\tau)) < \infty, \ 0\leq \tau \leq \tau_{-}  
				\} } 
	\end{split}
	\Ee \\
	where $\fint_{a}^{b} := \frac{1}{b-a} \int_{a}^{b}$. 
	\unhide
	
	For \eqref{split3},
	\begin{equation} \label{est:split3 raw}
	\begin{split}
	\eqref{split3} 
	&\lesssim \frac{ |f(s, X(s;t,\bx,v), V(s;t, \bx, v)) - f(s, X(s;t, \bx, \tv ), V(s;t, \bx, v ))| }{ | X(s;t, \bx, v ) -  X(s;t, \bx, \tv )|^{\gamma} } \\
	&\quad\quad \times | X(s;t, \bx, v ) -  X(s;t, \bx, \tv ) | ^{\gamma} 
	\mathbf{1}_{s \leq \min\{ t^{1}(\bx,v), t^{1}(\bar{x}, \tv) \}} \\
	&\quad + \frac{ |f(s, X(s;t,\bx, \tv), V(s;t, \bx, v)) - f(s, X(s;t, \bx, \tv ), V(s;t, \bx, \tv ))| }{ | V(s;t, \bx, v ) -  V(s;t, \bx, \tv )|^{\gamma} } \\
	&\quad\quad \times | V(s;t, \bx, v ) -  V(s;t, \bx, \tv )|^{\gamma} 
	\mathbf{1}_{s \leq \min\{t^{1}(\bx,v), t^{1}(\bar{x}, \tv) \}} \\
	&\quad + \frac{ |f(s, X(s;t,\bx,v), V(s;t, \bx, v)) - f(s, X(s;t, \bx, \tv ), V(s;t, \bx, v ))| }{ | X(s;t, \bx, v ) -  X(s;t, \bx, \tv )|^{\gamma} } \\
	&\quad\quad \times| (v - \tv)(t-s) | ^{\gamma} 
	\mathbf{1}_{s > \max\{  t^{1}(\bx,v), t^{1}(\bar{x}, \tv)  \}}  \\
	&\quad + \frac{ |f(s, X(s;t,\bx, \tv), V(s;t, \bx, v)) - f(s, X(s;t, \bx, \tv ), V(s;t, \bx, \tv ))| }{ | V(s;t, \bx, v ) -  V(s;t, \bx, \tv )|^{\gamma} } \\
	&\quad\quad \times| v - \tv |^{\gamma}   
	\mathbf{1}_{s > \max\{ t^{1}(\bx,v), t^{1}(\bar{x}, \tv)  \}}  \\
	&\quad + C_{1} + C_{2} + C_{3} + C_{4}  \\
	&\quad + D_{1} + D_{2} + D_{3} + D_{4},  \\
	\end{split}
	\end{equation}
	where
	\Be \notag 
	\begin{split}
		&C_{1} :=  \Big( \frac{ |f(s, X(s;t, \bar{x}, v ), R_{\xb(\bar{x},v)}v )  - f(s, \xb(\bar{x}, v), R_{\xb(\bar{x},v)}v  )| }{|X(s;t, \bar{x}, v ) - \xb(\bar{x}, v)|^{\gamma}} ( |v|( t^{1}(\bar{x}, v) - s ) )^{\gamma}   \\
		&\quad\quad  + \frac{ |f(s, \xb(\bar{x}, v), R_{\xb(\bar{x},v)}v  )  - f(s, \xb(\bar{x}, v), R_{\xb(\bar{x}, v)}\tv  ) | }{|R_{\xb(\bar{x},v)}v  - R_{\xb(\bar{x}, v)}\tv|^{\gamma}}
		|v - \tv|^{\gamma}   \\
		&\quad\quad  + \frac{ |f(s, \xb(\bar{x},v), \tv  ) - f(s, X(s;t, \bar{x}, \tv ), \tv ) | }{|\xb(\bar{x}, v) - X(s;t, \bar{x}, \tv )  |^{\gamma}}
		( |v - \tv|\tb(\bar{x}, v) + |v|( t^{1}(\bar{x}, v) - s ) )^{\gamma} \Big) \\
		&\quad\quad\quad \times  \mathbf{1}_{ t^{1}(\bx, \tv) < s \leq  t^{1}(\bx, v) } \mathbf{1}_{t^{1}(\bx, \tv) > -\infty},   \\
	\end{split}
	\Ee
	\Be \notag 
	\begin{split}
		C_{2} &:=  \Big( \frac{ |f(s, X(s;t, \bx, v ), V(s;t, \bx, v ) )  -   f(s, X(s;t, \bx , \V(\tau_{-}) ),  V(s;t, \bx, v ) ) |  }{ |X(s;t, \bx, v ) - X(s;t, \bx , \V(\tau_{-}) )|^{\gamma} }   \\
		&\quad\quad\quad \times |X(s;t, \bx, v ) - X(s;t, \bx , \V(\tau_{-}) )|^{\gamma}   \\
		&\quad\quad +   \frac{ | f(s, X(s;t, \bx , \V(\tau_{-}) ),  V(s;t, \bx, v ) )   -  f(s, X(s;t, \bx , \V(\tau_{-}) ),  V(s;t, \bx , \V(\tau_{-}) )   ) | }{ |V(s;t, \bx, v ) - V(s;t, \bx , \V(\tau_{-}) ) |^{\gamma} } \\
		&\quad\quad\quad\quad \times |V(s;t, \bx, v ) - V(s;t, \bx , \V(\tau_{-}) ) |^{\gamma}   \Big) 
		\mathbf{1}_{ s\leq \min\{ t^{1}(\bx, \V(\tau_{-})),  t^{1}(\bx, v) \} } \mathbf{1}_{t^{1}(\bx, \tv) = -\infty}  , \\
	\end{split}
	\Ee
	\Be \notag 
	\begin{split}
		C_{3} &:=  \Big( \frac{ |f(s, X(s;t, \bar{x}, v ), R_{\xb(\bar{x},v)}v )  - f(s, \xb(\bar{x}, v), R_{\xb(\bar{x},v)}v  )| }{|X(s;t, \bar{x}, v ) - \xb(\bar{x}, v)|^{\gamma}} ( |v|( t^{1}(\bar{x}, v) - s ) )^{\gamma}   \\
		&\quad\quad  + \frac{ |f(s, \xb(\bar{x}, v), R_{\xb(\bar{x},v)}v  )  - f(s, \xb(\bar{x}, v), R_{\xb(\bar{x}, v)}\V(\tau_{-})  ) | }{|R_{\xb(\bar{x},v)}v  - R_{\xb(\bar{x}, v)}\V(\tau_{-})|^{\gamma}}
		|v - \V(\tau_{-})|^{\gamma}   \\
		&\quad\quad  + \frac{ |f(s, \xb(\bar{x},v), \V(\tau_{-})  ) - f(s, X(s;t, \bar{x}, \V(\tau_{-}) ), \V(\tau_{-}) ) | }{|\xb(\bar{x}, v) - X(s;t, \bar{x}, \V(\tau_{-}) )  |^{\gamma}}  \\
		&\quad\quad\quad\quad \times
		\big( |v - \V(\tau_{-})|\tb(\bar{x}, v) + |v|( t^{1}(\bar{x}, v) - s ) \big)^{\gamma} \Big) 
		 \mathbf{1}_{ t^{1}(\bx, \V(\tau_{-})) < s \leq  t^{1}(\bx, v) \} } \mathbf{1}_{t^{1}(\bx, \tv) = -\infty} ,  \\
	\end{split}
	\Ee
	\Be \notag 
	\begin{split}
		C_{4} &:= \Big( \frac{ |f(s, X(s;t, \bx , \V(\tau_{-}) ),  \V(\tau_{-}) )  - f(s, X(s;t, \bx, \tv ),  \V(\tau_{-})  )| }{ |X(s;t, \bx , \V(\tau_{-}) ) - X(s;t, \bx, \tv ) |^{\gamma} }  |(\V(\tau_{-}) - \tv )(t-s)|^{\gamma}  \\
		&\quad\quad +  \frac{ |f(s, X(s;t, \bx, \tv ),  \V(\tau_{-})  )  - f(s, X(s;t, \bx, \tv ),  \tv )| }{ | \V(\tau_{-}) - \tv |^{\gamma} }  | \V(\tau_{-}) - \tv |^{\gamma}  \Big) \\
		&\quad\quad  \times  \mathbf{1}_{ t^{1}(\bx, \tv) < s \leq  t^{1}(\bx, v) } \mathbf{1}_{t^{1}(\bx, \tv) = -\infty},   \\
	\end{split}
	\Ee			
	and
	\Be\notag 
	\begin{split}
		D_{1,2,3,4} &:= \text{interchanging $v$ and $\tv$ in $C_{1,2,3,4}$, respectively. } \\
	\end{split}
	\Ee
	Applying \eqref{est:V/v} and \eqref{est:X/v} in Lemma \ref{frac sim S}, and \eqref{est:Lv} in Lemma \ref{L sim S}, to \eqref{est:split3 raw}, we obtain \eqref{est:split3}.    \\
	
	\hide
	\begin{eqnarray} \notag
		&&\frac{\eqref{split3}}{|v-\bv|^{\g}}  \\
		&&\lesssim \Big[ t + |v+\zeta|t^{2} + ( |v+\zeta| + |v+\zeta|^{2}t ) \mathcal{T}_{vel}(\bx, v, \tv, \zeta)
		\Big]^{2\b}  
		\notag  \\
		&&\quad\quad 
		\times  
		{\color{blue}  \Big[ \frac{ e^{\varpi \langle v+\zeta \rangle^2 s}  }{ \langle v+\zeta \rangle^{s_{1}}}  
			\sup_{v, |x - \bx|\leq 1} e^{-\varpi \langle v \rangle^2 s}  
			\langle v \rangle^{s_{1}} \frac{|f(s, x, v ) - f(s, \bx, v)|}{|x - \bx|^{2\b}}  + \frac{1}{w(v+\zeta)}\|w f(s)\|_{\infty} \Big] 
		}      \notag  \\
		&& + \Big[1 + |v+\zeta|t +  |v+\zeta|^{2} \mathcal{T}_{vel}(\bx, v, \tv, \zeta)
		\Big]^{2\b}  \notag \\
		&&\quad\quad 
		\times
		{\color{blue}  \Big[ 
		\frac{ e^{\varpi \langle v+\zeta \rangle^2 s} }{ \langle v+\zeta \rangle^{s_{2}} }	
			\sup_{x, |v - \bv|\leq 1} e^{\varpi \langle v \rangle^2 s}  \langle v \rangle^{s_{2}} \frac{|f(s, x, v ) - f(s, x, \bv)|}{|v - \bv|^{2\b}} + \frac{1}{w(v+\zeta)} \|wf(s)\|_{\infty}  \Big]
		}    \notag  
	\end{eqnarray}
	where we defined
	\Be  
	\begin{split}
		\mathcal{T}_{vel}(\bx, v, \tv, \zeta)  &:=  \fint_{0}^{1} \frac{1}{\mathfrak{S}_{vel}(\tau; \bx, v, \tv, \zeta)} d\tau  
			\mathbf{1}_{ \Big\{ \tiny \begin{aligned}
			&\tb(x, v(\tau)) < \infty, \ 0\leq \tau \leq 1  
			\\ 
			&\min_{0\leq \tau \leq 1}\tb(x, v(\tau)) \leq t-s \end{aligned}  \Big\} } 
		\\
		&\quad +  \fint_{\tau_{-}}^{1} \frac{1}{\mathfrak{S}_{vel}(\tau; \bx, v, \tv, \zeta)} d\tau  
		 	\mathbf{1}_{ \Big\{ \tiny \begin{aligned}
		 			&\tb(x, v(\tau)) < \infty, \ \tau_{-}\leq \tau \leq 1  
		 			\\ 
		 			&\min_{\tau_{-} \leq \tau \leq 1}\tb(x, v(\tau)) \leq t-s \end{aligned}  \Big\} } 
	 	\\
		&\quad +  \fint_{0}^{\tau_{-}} \frac{1}{\mathfrak{S}_{vel}(\tau; \bx, v, \tv, \zeta)} d\tau
			\mathbf{1}_{ \Big\{ \tiny \begin{aligned}
					&\tb(x, v(\tau)) < \infty, \ 0\leq \tau \leq \tau_{-}  
					\\ 
					&\min_{0\leq \tau \leq \tau_{-}  }\tb(x, v(\tau)) \leq t-s \end{aligned}  \Big\} } 
	\end{split}
	\Ee \\
	\unhide
\end{proof}
\begin{lemma} \label{nu G split}
	In \eqref{split1}--\eqref{split4}, let us replace $f$ into $\nu(f)$ or $\Gamma_{\text{gain} }(f,f)$. Corresponding Lemma \ref{f-f split} (with $\nu(f)$ or $\Gamma_{\text{gain} }(f,f)$, instead of $f$) satisfies the same estimates as \eqref{est:split1}--\eqref{est:split4}, except that we replace
	\[
		\frac{1}{w(v+\zeta)} \|wf(s)\|_{\infty} 
	\]
	on the RHS of each \eqref{est:split1}--\eqref{est:split4}, into
	\[
		\frac{\langle v+\zeta \rangle}{w(v+\zeta)} \|wf(s)\|^{2}_{\infty},\quad \text{for} \quad \Gamma_{\text{gain}}(f,f) \ \text{case}
	\]	
	and
	\[
		\frac{\langle v+\zeta \rangle}{w(v+\zeta)} \|wf(s)\|_{\infty},\quad \text{for} \quad \nu(f) \ \text{case},
	\]
	respectively.
\end{lemma}
\begin{proof}
	It is obvious because both $\nu(f)$ and $\Gamma_{\text{gain} }(f,f)$ also satisfy \eqref{specular} by Lemma \ref{lem_specular G}. We omit the proof. 
\end{proof}

\hide
specular reflection condition for $\Gamma$ original form
\Be
\begin{split}
	&\Gamma_{\text{gain}}(f,f)( t-\tb, X(t-\tb), V(t-\tb) )  \\
	&= \iint |(V(t-\tb)-u)\cdot\sigma| \sqrt{\mu}(u)
	f(t-\tb, X(t-\tb), u + ((V(t-\tb)-u)\cdot\sigma)\sigma )  \\
	&\quad \times f(t-\tb, X(t-\tb), V(t-\tb) - ((V(t-\tb)-u)\cdot\sigma)\sigma )   \\
	&= \iint |(R_{X(t-\tb)}V(t-\tb) - R_{X(t-\tb)}u)\cdot R_{X(t-\tb)}\sigma| \sqrt{\mu}(R_{X(t-\tb)}u)
	f(t-\tb, X(t-\tb), R_{X(t-\tb)}u + (R_{X(t-\tb)}(V(t-\tb)-u)\cdot R_{X(t-\tb)}\sigma)R_{X(t-\tb)}\sigma )  \\
	&\quad \times
	f(t-\tb, X(t-\tb), R_{X(t-\tb)}V(t-\tb) - (R_{X(t-\tb)}(V(t-\tb)-u)\cdot R_{X(t-\tb)}\sigma)R_{X(t-\tb)}\sigma )  d\sigma du \\
	&= \iint |(v - R_{X(t-\tb)}u)\cdot R_{X(t-\tb)}\sigma| \sqrt{\mu}(R_{X(t-\tb)}u)
	f(t-\tb, X(t-\tb), R_{X(t-\tb)}u + ( (v-R_{X(t-\tb)}u)\cdot R_{X(t-\tb)}\sigma)R_{X(t-\tb)}\sigma )  \\
	&\quad \times
	f(t-\tb, X(t-\tb), v - ( (v - R_{X(t-\tb)}u)\cdot R_{X(t-\tb)}\sigma)R_{X(t-\tb)}\sigma )  d R_{X(t-\tb)}\sigma d R_{X(t-\tb)}u \\
	&= \iint |(v - u)\cdot \sigma| \sqrt{\mu}(u)  
	f(t-\tb, X(t-\tb), u + ( (v-u)\cdot \sigma)\sigma )  
	f(t-\tb, X(t-\tb), v - ( (v - u)\cdot \sigma)\sigma )  d\sigma du   \\
	&\quad \quad \text{by}\quad R_{X(t-\tb)}u \rightarrow u, \ \ R_{X(t-\tb)}\sigma \rightarrow \sigma \\
	&= \Gamma_{\text{gain}}(f,f)( t-\tb, X(t-\tb), v )  \\
\end{split}
\Ee
Note that 
\Be
\begin{split}
	& f(t-\tb, X(t-\tb), u + ((V(t-\tb)-u)\cdot\sigma)\sigma )  \\
	&=  f(t-\tb, X(t-\tb), R_{X(t-\tb)}u + (R_{X(t-\tb)}(V(t-\tb)-u)\cdot R_{X(t-\tb)}\sigma)R_{X(t-\tb)}\sigma )  \\
\end{split}
\Ee	
\unhide

\subsection{Integrability for $\b < \frac{1}{2}$}

\begin{lemma} \label{curvature 1}
	Consider $\O$ as in Definition \ref{def:domain}. We $x \in \O$ and choose a unique $\hat{z}$ whose backward trajectory hit $\p\O$ verically, i.e., $\hat{z}\cdot \hat{n}(\xb(x, \hat{z})) = -1$. For any plane $S$ which includes $x$ and $x+\hat{z}$, curvature at any point $y\in\p\O\cap S$ is uniformly nonzero of which lower/upper bounds depend only on $\O$. 
\end{lemma}
\begin{proof}
	Imagine a plane which is perpendicular to $\hat{z}$. Then using an angle on the plane (as cylindrical coordinate or spherical coordinate), we can parametrize all possible planes $S$ as $S_{\varphi}$ with $[0,2\pi)$ wher $S_{0}=S_{2\pi}$. Each cross section $\p\O\cap S_{\varphi}$ is uniformly convex, so there are finite maximum and minimum of curvature on the curve depending on $\O$. Now using compactness of $\varphi\in[0,2\pi]$, we finish the proof.  
\end{proof}

\begin{lemma}[Integrability] \label{lem_int sing}
	\noindent (i)When $\b < \frac{1}{2}$	
	\Be \label{est : integ1}
		\int_{\{\zeta :  \xb(x, v+\zeta)\in\p\O\}} \frac{e^{-c|\zeta|^{2}}}{|\zeta|} 
		\frac{ \langle v+\zeta \rangle^{r} }{ |(v+\zeta)\cdot\nabla\xi(\xb(x, v+\zeta))|^{2\b} } 
		d\zeta \leq C_{\b} \langle v \rangle^{r +1 - 2\b}.
	\Ee
	\noindent (ii) When $0 < \b \leq  \frac{1}{4}$,
	\Be \label{est : integ2}
	\int_{\{\zeta :  \xb(x, v+\zeta)\in\p\O\}} \frac{e^{-c|\zeta|^{2}}}{|\zeta|} 
	\frac{ \langle v+\zeta \rangle^{r} }{ |(v+\zeta)\cdot\nabla\xi(\xb(x, v+\zeta))|^{2\b} } \frac{1}{|v+\zeta|^{2\b}} d\zeta \leq C_{\b} \langle v \rangle^{r+1-4\b}.
	\Ee
	\noindent (iii) When $\frac{1}{4} < \b < \frac{1}{2}$,
	\Be \label{est : integ3}
	\int_{\{\zeta :  \xb(x, v+\zeta)\in\p\O\}} \frac{e^{-c|\zeta|^{2}}}{|\zeta|} 
	\frac{ \langle v+\zeta \rangle^{r} }{ |(v+\zeta)\cdot\nabla\xi(\xb(x, v+\zeta))|^{2\b} } \frac{1}{|v+\zeta|^{2\b}} d\zeta \leq C_{\b} \langle v \rangle^{r}.
	\Ee
	\hide
	When $|v-\bv| \leq 1$,
	\Be 
		\int_{\{\zeta :  \xb(x, v+\zeta)\in\p\O\}} \frac{e^{-c|\zeta|^{2}}}{|\zeta|} 
		\frac{ \langle v+\zeta \rangle^{r}}{|(v+\zeta)\cdot\nabla\xi|^{2\b}} 
		\frac{1}{|v+\zeta|^{2\b}}d\zeta
	\Ee
	\unhide
\end{lemma}
\begin{proof}
	Let us prove \eqref{est : integ1} first. We consider a fixed point $x$ and $\zeta\in\R^{3}$ such that $\xb(x, \zeta)\in\p\O$ is well defined. There exist a unique $\hat{z}$ whose backward trajectory hit $\p\O$ verically, i.e., $\hat{z}\cdot \hat{n}(\xb(x, \hat{z})) = -1$. We consider spherical coordinate of $\zeta\in\R^{3}$ whose $\hat{z}$ is z-axis. Then for each $\varphi\in[0,2\pi)$, there exists $\theta_{g} = \theta_{g}(\varphi) \in [0, \frac{\pi}{2}]$ such that $\zeta_{g,\varphi}\cdot \nabla\xi(\xb(x, \zeta_{g,\varphi})) = 0$, where 
	\[
		\zeta_{g,\varphi} := |\zeta|(\sin\theta_{g}\cos\varphi, \sin\theta_{g}\sin\varphi, \cos\theta_{g}),
	\]  
	whose spherical component is $(|\zeta|, \theta_{g}(\varphi), \varphi)$ so that its trajectory grazes on $\p\O$.  \\
	
	Now, let us use $S_{\varphi}$ to denote the $\varphi$-plane in above coordinate. Since the cross section $\p\O\cap S_{\varphi}$ with a fixed $\varphi$ is a two dimensional uniformly convex curve on the cross section,
	\Be \label{2D angle}
		|\zeta \cdot \widehat{n_{\parallel}}(\xb(x, \zeta_{g,\varphi}))| \leq |\zeta \cdot \widehat{n_{\parallel}}(\xb(x, \zeta))|
	\Ee
	is obvious, where $\zeta\in\p\O\cap S_{\varphi}$ and $n_{\parallel}$ is projection of $\frac{\nabla\xi(\xb(x, \zeta))}{|\nabla\xi(\xb(x, \zeta))|}$ onto $\p\O\cap S_{\varphi}$. Combining \eqref{2D angle} and Lemma \ref{lem_unif n}, we can derive 
	\[
		|\zeta \cdot\nabla\xi(\xb(x, \zeta_{g,\varphi}))| \lesssim_{\O} |\zeta \cdot\nabla\xi(\xb(x, \zeta))|,
	\]
	using similar argument as \eqref{dot x opt}, where $\zeta$ has spherical coordinates $(|\zeta|, \theta, \varphi)$, $0\leq \theta \leq \theta_{g}(\varphi)$ for fixed $\varphi$. Then, for given $|v|>0$,
	\begin{eqnarray}
		&& \int_{\{\zeta :  \xb(x, \zeta)\in\p\O\}} \frac{e^{-c|v-\zeta|^{2}}}{|v-\zeta|} \frac{ \langle \zeta \rangle^{r} }{|\zeta \cdot\nabla\xi(\xb(x, \zeta))|^{2\b}} 
		d\zeta  \notag \\
		&\lesssim& \int_{0}^{2\pi}   \int_{0}^{\infty} \int_{0}^{\theta_{g}(\varphi)} 
		\frac{e^{-c|v-\zeta|^{2}}}{|v-\zeta|} \frac{  \langle \zeta \rangle^{r} }{|\zeta \cdot\nabla\xi(\xb(x, \zeta_{g,\varphi}))|^{2\b}} 
		|\zeta|^{2} \sin\theta d|\zeta|d\theta d\varphi    \notag   \\
		&\lesssim& \sup_{\substack{ \hat{v} = (\theta_{v}, \varphi_{v}) \\
		0\leq \theta_{v} \leq \theta_{g}(\varphi)}}
		\int_{0}^{2\pi}   \int_{0}^{\infty} \int_{0}^{\theta_{g}(\varphi)} 
		\frac{e^{-c|v-\zeta|^{2}}}{ |v-\zeta| } \frac{  \langle \zeta \rangle^{r} }{ \sin^{2\b}(\theta_{g}(\varphi) - \theta)} \frac{1}{|\zeta|^{2\b}} |\zeta|^{2}\sin\theta d|\zeta| d\theta  d\varphi,    \label{integ:bdry}
	\end{eqnarray}
	where we also used $|n_{\parallel}(\xb(x, \zeta_{g,\varphi}))| \gtrsim_{\O} 1$ in the last step, which is true by Lemma \ref{lem_unif n} and Lemma \ref{curvature 1}. (If $v=0$, \eqref{est : integ1} is obvious.) \\
	Notice that we get optimal $\theta_{g}(\varphi) =\frac{\pi}{2}$ for all $\varphi$ only when $x\in\p\O$. Therefore, \eqref{integ:bdry} is optimal when $x \in \p\O$ and 
	\begin{eqnarray}
	\text{LHS of } \ \eqref{est : integ1} 
	&\lesssim& \sup_{\substack{ \hat{v}\in\mathbb{S}^{2} \\ \hat{v}\cdot\nabla\xi(x) \leq 0}}
	\int_{ \{\zeta\cdot \nabla\xi(x) \leq 0\} }
	\frac{e^{-c|v-\zeta|^{2}}}{ |v-\zeta| } \frac{  \langle \zeta \rangle^{r} }{ |\zeta\cdot\nabla\xi(x)|^{2\b} } \frac{1}{|\zeta|^{2\b}} d\zeta, \quad x\in\p\O,
	\label{est : integ pre} 
	\end{eqnarray}
	which is integration on half space $\{\zeta\in\R^{3} : \zeta\cdot \nabla\xi(x) \leq 0 \}$ when $x\in\p\O$. To make estimate easier, let us change axis of spherical coordinate. We assign a direction vector in tangential plane of $x\in\p\O$ to $\hat{z}$ axis (so that $\hat{z}\cdot\nabla\xi(x)=0$) and also assign $-\frac{\nabla\xi(x)}{ |\nabla\xi(x)| }$ to $\hat{y}$. Then, 
	\Be \label{new coord}
		\{\zeta : \zeta\cdot \nabla\xi(x) \leq 0 \} = \{ \zeta = (|\zeta|, \theta, \varphi)  \ : \ 0\leq |\zeta| <\infty, \ \ 0\leq \theta \leq \pi, \ \ 0\leq \varphi \leq \pi \},
	\Ee
	in spherical coordinate and  
	\[
		\Big| \zeta\cdot\frac{\nabla\xi(x)}{ |\nabla\xi(x)| } \Big|= |\zeta|\sin\theta\sin\varphi.
	\]
	We write $v = |v|(\sin\theta_{v}\cos\varphi_{v}, \sin\theta_{v}\sin\varphi_{v}, \cos\varphi_{v})$, then  
	\begin{equation} \notag %
	\begin{split}
		|v-\zeta|_{\varphi_{v}} &:= \sqrt{|v|^{2} +|\zeta|^{2} - 2|v||\zeta|\cos(\theta - \theta_{v})} \leq |v-\zeta|,
	\end{split}
	\end{equation}
	where $|v-\zeta|_{\varphi_{v}}$ is 2D distance in a fixed $\varphi_{v}$ plane, when both $v$ and $\zeta$ have coordinate $(\theta_{v}, \varphi_{v})$ and $(\theta, \varphi_{v})$, respectively. Now we treat $v$ and $\zeta$ as like 2D vectors in fixed $\varphi_{v}$ plane. Therefore, applying \eqref{new coord} to \eqref{est : integ pre},
	\begin{eqnarray}
		\eqref{est : integ1}
		&\lesssim& \int_{0}^{\pi} \frac{1}{\sin^{2\b}\varphi} d\varphi \Big[ \int_{0}^{\infty} \int_{0}^{\pi} 
		\frac{ e^{-c|v-\zeta|^{2}_{\varphi_{v}}} }{ |v-\zeta|_{\varphi_{v}} }    \langle \zeta \rangle^{r} \sin^{1-2\b}\theta \frac{1}{|\zeta|^{2\b-1}} dA \Big]  \notag \\
		&\lesssim& C_{\b} \int_{0}^{\infty} \int_{0}^{\pi} 
		\frac{ e^{-c|v-\zeta|^{2}_{\varphi_{v}}} }{ |v-\zeta|_{\varphi_{v}} }    \langle \zeta \rangle^{r} \sin^{1-2\b}\theta \frac{1}{|\zeta|^{2\b-1}} dA,\quad 2\b-1 < 0,  \label{integ grow} \\
		&\lesssim&\langle v \rangle^{r +1-2\b} \notag, 
	\end{eqnarray}
	where $dA = |\zeta|d|\zeta| d\theta$ is 2D measure in $\varphi_{v}$ plane. This proves \eqref{est : integ1}. \\
	
	Proof for \eqref{est : integ2} is nearly same as \eqref{est : integ1} since $4\b-1 \leq 0$. We modify \eqref{integ grow} to get
	\Be \label{integ : mid}
		C_{\b} \int_{0}^{\infty} \int_{0}^{\pi} 
		\frac{ e^{-c|v-\zeta|^{2}_{\varphi_{v}}} }{ |v-\zeta|_{\varphi_{v}} }    \langle \zeta \rangle^{r} \sin^{1-2\b}\theta \frac{1}{|\zeta|^{4\b-1}} dA 
		\lesssim \langle v \rangle^{r + 1 -4\b}.
	\Ee
	To prove \eqref{est : integ3}, from the LHS of \eqref{integ : mid} with $4\b - 1>0$,
	we use H\"older inequality with $p < 2$ and $q >2$ (to be justified below) so that
	\begin{eqnarray}  
		\eqref{est : integ1}
		&\lesssim& C_{\b} \int_{0}^{\infty} \int_{0}^{\pi} 
		\frac{ e^{-c|v-\zeta|^{2}_{\varphi_{v}}} }{ |v-\zeta|_{\varphi_{v}} }   \langle \zeta \rangle^{r} \sin^{1-2\b}\theta \frac{1}{|\zeta|^{4\b-1}} dA   \notag \\
		&\lesssim& C_{\b}  \langle v \rangle^{r} \Big[ \iint e^{-\frac{cq}{2}|v-\zeta|_{\varphi_{v}}^{2}}    \frac{1}{|\zeta|^{(4\b-1)q}} dA \Big]^{\frac{1}{q}}  \notag \\
		&\lesssim& C_{\b}  \langle v \rangle^{r}, \notag 
	\end{eqnarray}
	where $q=q(\b) >2$ and $p=p(\b) < 2$ can be chosen depending on $\b < \frac{1}{2}$ so that
	\Be	\notag 
		(4\b-1)q < 2.		
	\Ee
\end{proof}

\begin{corollary} \label{lem_int T}
	Let $\b < \frac{1}{2}$	 and $|v-\bv| \leq 1$. We also assume \eqref{assume_x} and \eqref{assume_x2} with $S_{(x, \bx, v+\zeta)}$, and \eqref{assume_v} and \eqref{assume_v2} with $S_{(\bx, v, \bv, \zeta)}$. Then, we get the followings,
	\begin{eqnarray} 
		\int_{\zeta} k_{c}(v, v+\zeta)   \langle v+\zeta \rangle^{r} \mathcal{T}^{2\b}_{sp}(x, \tx, v+\zeta) d\zeta  
		&\lesssim& C_{\b} \langle v \rangle^{r+1-2\b},  \label{est : int Tsp}  
		\\
		\int_{\zeta} \mathbf{k}_{c}(v, \bv, \zeta)   \langle v+\zeta \rangle^{r} \mathcal{T}^{2\b}_{vel}(\bx, v, \tv, \zeta; t,s) d\zeta  
		&\lesssim& C_{\b} \langle v \rangle^{r+\max\{1-4\b,0\}} \big( 1 +  \langle v \rangle (t-s) \big),  \label{est : int Tvel}  
	\end{eqnarray}
	where $\mathcal{T}_{sp}$ and $\mathcal{T}_{vel}$ are defned in \eqref{def_Tsp} and \eqref{def_Tvel}, respectively. \\ 
\end{corollary}
\begin{proof}
	For \eqref{est : int Tsp}, we apply \eqref{int:Sx} to \eqref{def_Tsp} and then use \eqref{est : integ1}. For \eqref{est : int Tvel}, we apply \eqref{int:Sx} to \eqref{def_Tvel} and then use \eqref{est : integ2} and \eqref{est : integ3}.   
\end{proof}

\subsection{Uniform estimates for $\mathfrak{H}^{2\b}_{sp, vel}$}
We start from the following 
\begin{equation}  \notag
\begin{split}
& |f(t,x,v+ \zeta)-f(t,\bar{x}, \bar{v}+ \zeta)|  \\
&\leq   
e^{- \int^t_ 0 \nu(f) (\tau, X(\tau ), V(\tau )) d \tau}
|f(0,X(0 ), V(0 ))-f(0,\bar{X}(0 ), \bar{V}(0 ))| 
\\
&\quad + \int^t_0 
e^{- \int^t_ s \nu(f) (\tau, X(\tau ), V(\tau )) d \tau} 	
|\Gamma_{\text{gain}}(f,f)(s,X(s ), V(s ))
-\Gamma_{\text{gain}}(f,f)(s,\bar{X}(s ), \bar{V}(s ))|   
\\
&\quad +  
\big| e^{- \int^t_ 0 \nu(f) (\tau, X(\tau ), V(\tau )) d \tau}  -  
e^{- \int^t_0 \nu(f) (\tau, \bar{X}(\tau ), \bar{V}(\tau )) d \tau} \big|	
|f(0, \bar{X}(0), \bar{V}(0))|
\\
&\quad +\int^t_0
\big| e^{- \int^t_ s \nu(f) (\tau, X(\tau ), V(\tau )) d \tau}  -  
e^{- \int^t_s \nu(f) (\tau, \bar{X}(\tau ), \bar{V}(\tau )) d \tau} \big|
|	\Gamma_{\text{gain}}(f,f)(s,\bar{X}(s ), \bar{V}(s ))| ds, \\
\end{split}
\end{equation}
which is trivial by \eqref{f_expan}. Since, $|e^{-a}-e^{-b}| \leq |a-b|$ for $a\geq b \geq 0$, we obtain the following basic esitmate,
\begin{eqnarray} 
&& |f(t,x,v+ \zeta)-f(t,\bar{x}, \bar{v}+ \zeta)|   \label{basic f-f} \\
&&\leq   
|f(0,X(0 ), V(0 ))-f(0,\bar{X}(0 ), \bar{V}(0 ))| 
 \label{basic f-f1} \\
&&\quad + \int^t_0 
|\Gamma_{\text{gain}}(f,f)(s,X(s ), V(s ))
-\Gamma_{\text{gain}}(f,f)(s,\bar{X}(s ), \bar{V}(s ))|   
 \label{basic f-f2} \\
&&\quad +  \|w_{0}f_{0}\|_{\infty} \frac{1}{w_{0}(\bv+\zeta)} 
\int_{0}^{t } | \nu(f) (s, X(s), V(s ))  - \nu(f) (s, \bar{X}(s), \bar{V}(s)) | ds
 \label{basic f-f3} \\
&&\quad + t \sup_{0\leq s \leq t} \|wf(s)\|_{\infty}
\frac{1}{\sqrt{w(\bv+\zeta)}}  
\int_{0}^{t } | \nu(f) (s, X(s), V(s ))  - \nu(f) (s, \bar{X}(s), \bar{V}(s)) | ds , \label{basic f-f4}  
\end{eqnarray}
where 
\[
	X(s) := X(s;t, x, v+\zeta),\quad V(s) := V(s;t, x, v+\zeta),\quad \bar{X}(s) := \bar{X}(s;t, \bx, \bv+\zeta), \quad \bar{V}(s) := \bar{V}(s;t, \bx, \bv+\zeta).
\]


\begin{proposition}[Seminorm estimate]\label{prop_unif H} 
Suppose the domain is given as in Definition \ref{def:domain} and \eqref{convex_xi}. For $0 < |(x,v)-(\bar x, \bar v)|\leq 1$ and $\zeta \in \R^3$, there exists $\varpi \gg_{f_{0}, \b} 1$ such that
\Be \label{est : H}
\begin{split}
	&\big[ \sup_{0\leq s \leq T}\mathfrak{H}_{sp}^{2\b}(s) + \sup_{0\leq s \leq T}\mathfrak{H}_{vel}^{2\b}(s) \big]  \\
	&\lesssim_{\b} \sup_{\substack{v\in\R^{3} \\ 0 < |x - \bx|\leq 1}} 
	\langle v \rangle \frac{|f_{0}( x, v ) - f_{0}(\bx, v)|}{|x - \bx|^{2\b}}  
	+ \sup_{\substack{x\in\bar{\O} \\ 0 < |v - \bv|\leq 1}}  \langle v \rangle^{2} \frac{|f_{0}( x, v ) - f_{0}( x, \bv)|}{|v - \bv|^{2\b}}  
	+ \|w_{0} f_{0}\|_{\infty}. \\
\end{split}
\Ee
for sufficiently small $T>0$ such that $\varpi T \ll 1$.
\end{proposition}

\hide
{\color{blue}  
\begin{remark}
	For optimal weight, we change RHS (v-directional initial data) into
	\[
		\sup_{\substack{x\in\bar{\O} \\ |v - \bv|\leq 1}}  \langle v \rangle^{1+2\b} \frac{|f_{0}( x, v ) - f_{0}( x, \bv)|}{|v - \bv|^{2\b}}  ,\quad 1+2\b < 2.
	\]
\end{remark}
} 
\unhide

\begin{proof}
		{\bf Step 1} First, we estimate $\mathfrak{H}^{2\b}_{vel}$.  From definition \ref{def_H}, we estimate (we use $\bx$ instead of $x$ to use Lemma \ref{f-f split} directly).
		\Be \label{f-f : x v bv}
			e^{-\varpi \langle v \rangle^{2} t} 
			\int_{\R^{3}_{\zeta}} \mathbf{k}_{c}(v, \bv, \zeta) \frac{ |f(t, \bx, v + \zeta) - f(t, \bx, \bv + \zeta)| }{|v-\bv|^{2\b}} d\zeta,\quad\text{for}\quad  |v-\bv| \leq 1.
		\Ee
		 We note that $\tv+\zeta$ is well-defined only when \eqref{assume_v} holds. For given $v$ and $\bv$, however, the set of $\zeta\in\R^{3}$, where \eqref{assume_v} does not hold is of measure zero. So we assume \eqref{assume_v} without loss of generality throughout {\bf Step 1} and can use \eqref{est:split3} and \eqref{est:split4} in Lemma \ref{f-f split}.  \\
		
		\noindent For $f(t, \bx, v + \zeta) - f(t, \bx, \bv + \zeta)$ in integrand of \eqref{f-f : x v bv}, we use expansion \eqref{basic f-f} by replaincg
		\Be\notag 
		\begin{split}
			( \bar X ,   \bar V ) := (\bar X (s),  \bar V (s))&= (X(s;t, \bx,  v+ \zeta), V(s;t, \bx, \bar v+ \zeta)),  \quad  |v-\bv| \leq 1, \\
		\end{split}
		\Ee
		Now, we apply  $e^{-\varpi \langle v \rangle^{2} t} \langle v \rangle^{2\b} \int_{\R^{3}_{\zeta}} \mathbf{k}(v, \bv, \zeta) \frac{ \cdot }{|v-\bv|^{2\b}} d\zeta$ to each \eqref{basic f-f1} -- \eqref{basic f-f4}. \\
		
		{\it Substep 1-1} In this substep, we consider \eqref{basic f-f1}. To consider \eqref{basic f-f1}, we put
		\Be \label{replace1-1}
			 s=0,  \quad x = \bx.
		\Ee
		in \eqref{split1}--\eqref{split4}. Then, it is sufficient to consider only \eqref{split3} and \eqref{split4} only. Since we are dealing with difference of $f$, let us use notation $\eqref{split3}_{f}$ and $\eqref{split4}_{f}$ to stress the function $f$. Using \eqref{est:split3}, we get 
		\begin{equation} \label{est : f-f1-split3}
		\begin{split}
			&e^{-\varpi \langle v \rangle^{2} t} 
			 \int_{\R_{\zeta}^{3}} \mathbf{k}_{c}(v, \bv, \zeta) \frac{ |\eqref{split3}_{f} | }{|v-\bv|^{2\b}} d\zeta \\
			&\lesssim e^{-\varpi \langle v \rangle^{2} t} 
			\int_{\R_{\zeta}^{3}} \mathbf{k}_{c}(v, \bv, \zeta)  
			\Big[ t + |v+\zeta|t^{2} + ( |v+\zeta| + |v+\zeta|^{2}t ) \mathcal{T}_{vel}(\bx, v, \tv, \zeta; t,0)
			\Big]^{2\b}  
			  \\
			&\quad\quad 
			\times  
				\Big[ \frac{ 1 }{ \langle v+\zeta \rangle}  
				\sup_{\substack{v\in\R^{3} \\ 0 < |x - \bx|\leq 1}} 
				\langle v \rangle \frac{|f_{0}( x, v ) - f_{0}(\bx, v)|}{|x - \bx|^{2\b}}  + \frac{1}{w_{0}(v+\zeta)}\|w_{0} f_{0}\|_{\infty} \Big] 
			d\zeta 
			  \\
			& + e^{-\varpi \langle v \rangle^{2} t} 
			\int_{\R_{\zeta}^{3}} \mathbf{k}_{c}(v, \bv, \zeta) 
			\Big[1 + |v+\zeta|t +  |v+\zeta|^{2} \mathcal{T}_{vel}(\bx, v, \tv, \zeta; t,0)
			\Big]^{2\b}   \\
			&\quad\quad 
			\times
				\Big[ 
				\frac{ 1}{ \langle v+\zeta \rangle^{2} }	
				\sup_{ \substack{ x\in \overline{\O} \\ 0 < |v - \bv|\leq 1   } }  \langle v \rangle^{2} \frac{|f_{0}( x, v ) - f_{0}( x, \bv)|}{|v - \bv|^{2\b}} + \frac{1}{w_{0}(v+\zeta)} \|w_{0}f_{0}\|_{\infty}  \Big]
			d\zeta   \\
			&\lesssim_{\b}  
			\sup_{\substack{v\in\R^{3} \\ 0 < |x - \bx|\leq 1}} 
			\langle v \rangle \frac{|f_{0}( x, v ) - f_{0}(\bx, v)|}{|x - \bx|^{2\b}}  
			+ \sup_{ \substack{ x\in \overline{\O} \\ 0 < |v - \bv|\leq 1   } }  \langle v \rangle^{2} \frac{|f_{0}( x, v ) - f_{0}( x, \bv)|}{|v - \bv|^{2\b}}  
			+ \|w_{0} f_{0}\|_{\infty},
		\end{split}
		\end{equation}
		where we have used Corollary \ref{lem_int T}.   \\
		
		We similarly apply \eqref{replace1-1} and use $\eqref{est:split4}$ to get estimate for $\eqref{split4}_{f}$,
		\begin{equation} \label{est : f-f1-split4}
		\begin{split}
		&e^{-\varpi \langle v \rangle^{2} t} 
		\int_{\R_{\zeta}^{3}} \mathbf{k}_{c}(v, \bv, \zeta) \frac{ |\eqref{split4}_{f} | }{|v-\bv|^{2\b}} d\zeta \\
		&\lesssim e^{-\varpi \langle v \rangle^{2} t} 
		\int_{\R_{\zeta}^{3}} \mathbf{k}_{c}(v, \bv, \zeta) t^{2\b}
		\Big[ \frac{ 1}{\langle v +\zeta\rangle}  \sup_{\substack{v\in\R^{3} \\ 0 < |x - \bx|\leq 1}}  \Big(  \langle v \rangle \frac{|f_{0}(x, v ) - f_{0}(\bar{x}, v)|}{|x - \bx|^{2\b}} \Big) + \frac{1}{w_{0}(v+\zeta)}  \|w_{0} f_{0}\|_{\infty} \Big]  d\zeta    \\
		&\quad + e^{-\varpi \langle v \rangle^{2} t} 
		\int_{\R_{\zeta}^{3}} \mathbf{k}_{c}(v, \bv, \zeta) 
		\Big[ \frac{ 1 }{\langle v+\zeta \rangle^{2}}  \sup_{ \substack{ x\in \overline{\O} \\ 0 < |v - \bv|\leq 1   } } \Big(  \langle v \rangle^{2} \frac{|f_{0}( x, v ) - f_{0}( x, \bv)|}{|v - \bv|^{2\b}} \Big)
		+ \frac{1}{w_{0}(v+\zeta)}  \|w_{0} f_{0}\|_{\infty} \Big] d\zeta  \\
		&\lesssim_{\b}
		\sup_{\substack{v\in\R^{3} \\ 0 < |x - \bx|\leq 1}} 
		\langle v \rangle \frac{|f_{0}( x, v ) - f_{0}(\bx, v)|}{|x - \bx|^{2\b}}  
		+ \sup_{ \substack{ x\in \overline{\O} \\ 0 < |v - \bv|\leq 1   } }  \langle v \rangle^{2} \frac{|f_{0}( x, v ) - f_{0}( x, \bv)|}{|v - \bv|^{2\b}}  
		+ \|w_{0} f_{0}\|_{\infty}.
		\end{split}
		\end{equation}
		Note that the bound of  \eqref{est : f-f1-split3} also control \eqref{est : f-f1-split4}.  \\	 
		 Hence, 
		\Be \label{est : basic f-f1-dv}
		\begin{split}
			&e^{-\varpi \langle v \rangle^{2} t} 
			\int_{\R_{\zeta}^{3}} \mathbf{k}_{c}(v, \bv, \zeta) \frac{ |\eqref{basic f-f1} | }{|v-\bv|^{2\b}} d\zeta   \lesssim \text{RHS of} \ \eqref{est : f-f1-split3}.
		\end{split}
		\Ee
		
		{\it Substep 1-2} In this substep, we consider  \eqref{basic f-f2}. We put
		\Be \label{replace1-2}
			x=\bx,
		\Ee
		and use \eqref{est:split3} replacing $f$ into $\Gamma_{\text{gain}}(f,f)$. Using Lemma \ref{nu G split}, Lemma \ref{lem_Gamma}, and \eqref{bf nega} of Lemma \ref{lem_nega}, we obtain
		\begin{equation} \notag 
		\begin{split}
		&\int_{0}^{t} e^{-\varpi \langle v \rangle^{2} t} 
		\int_{\R_{\zeta}^{3}} \mathbf{k}_{c}(v, \bv, \zeta) \frac{ |\eqref{basic f-f2} | }{|v-\bv|^{2\b}} d\zeta ds   \\
		&\lesssim \int_{0}^{t} e^{-\varpi \langle v \rangle^{2} (t-s)} 
		\int_{\R_{\zeta}^{3}} \mathbf{k}_{c}(v, \bv, \zeta)
		e^{- \varpi \langle v \rangle^{2} s}  
		\Big[ t + |v+\zeta|t^{2} + ( |v+\zeta| + |v+\zeta|^{2}t ) \mathcal{T}_{vel}(\bx, v, \tv, \zeta; t, s)
		\Big]^{2\b}  
		 \\
		&\quad\quad 
		\times  
		\Big[  e^{\varpi \langle v+\zeta \rangle^2 s} 
			\sup_{\substack{v\in\R^{3} \\ 0 < |x - \bx|\leq 1}}  e^{-\varpi \langle v \rangle^2 s}  
			\frac{|\Gamma_{\text{gain} }(s, x, v ) - \Gamma_{\text{gain} }(s, \bx, v)|}{|x - \bx|^{2\b}}  + \frac{\langle v+\zeta \rangle}{w(v+\zeta)}\|w f(s)\|^{2}_{\infty} \Big] 
		d\zeta   \\
		& + \int_{0}^{t} e^{-\varpi \langle v \rangle^{2} (t-s)} 
		\int_{\R_{\zeta}^{3}} \mathbf{k}_{c}(v, \bv, \zeta)
		e^{-\varpi \langle v \rangle^{2} s} 
		\Big[1 + |v+\zeta|t +  |v+\zeta|^{2} \mathcal{T}_{vel}(\bx, v, \tv, \zeta; t,s)
		\Big]^{2\b}    \\
		&\quad\quad 
		\times
		\Big[ 
			e^{\varpi \langle v+\zeta \rangle^2 s}  
			\sup_{ \substack{ x\in \overline{\O} \\ 0 < |v - \bv|\leq 1   } } e^{-\varpi \langle v \rangle^2 s}  
			\frac{|\Gamma_{\text{gain} }(s, x, v ) - \Gamma_{\text{gain} }(s, x, \bv)|}{|v - \bv|^{2\b}} + \frac{\langle v+\zeta \rangle}{w(v+\zeta)} \|wf(s)\|^{2}_{\infty}  \Big]
		d\zeta
		\\
		&\lesssim \int_{0}^{t} e^{-\varpi \langle v \rangle^{2} (t-s)} 
		\int_{\R_{\zeta}^{3}} \mathbf{k}_{\frac{c}{2}}(v, \bv, \zeta)
		\Big[ t + |v+\zeta|t^{2} + ( |v+\zeta| + |v+\zeta|^{2}t ) \mathcal{T}_{vel}(\bx, v, \tv, \zeta; t, s)
		\Big]^{2\b}  
		\\
		&\quad\quad 
		\times  
		\Big[ 
			\Big( 
			\|wf(s)\|_{\infty}
			\sup_{\substack{v\in\R^{3} \\ 0 < |x - \bx|\leq 1}}  e^{-\varpi \langle v \rangle^2 s}  
			 \int_{\R^{3}_{u}} k_{c}(v, v+u) \frac{|f(s, x, v+u ) - f(s, \bx, v+u)|}{|x - \bx|^{2\b}} du
			\Big) 
			\\
			&\quad\quad\quad \quad
			  + \frac{\langle v+\zeta \rangle}{w(v+\zeta)}\|w f(s)\|^{2}_{\infty} \Big] 
		d\zeta   \\
		& + \int_{0}^{t} e^{-\varpi \langle v \rangle^{2} (t-s)} 
		\int_{\R_{\zeta}^{3}} \mathbf{k}_{\frac{c}{2}}(v, \bv, \zeta)
		\Big[1 + |v+\zeta|t +  |v+\zeta|^{2} \mathcal{T}_{vel}(\bx, v, \tv, \zeta; t,s)
		\Big]^{2\b}    \\
		&\quad\quad 
		\times
		 \Big[ 
			\Big(
			\|wf(s)\|_{\infty}
			\sup_{ \substack{ x\in \overline{\O} \\ 0 < |v - \bv|\leq 1   } } e^{-\varpi \langle v \rangle^2 s} 
			\int_{\R_{\zeta}^{3}} \mathbf{k}_{\frac{c}{2}}(v, \bv, u)
			\frac{|f(s, x, v+u ) - f(s, x, \bv+u)|}{|v - \bv|^{2\b}}
			\Big) 
		\\
		&\quad\quad\quad \quad
			 + \big( \frac{1}{\langle v+\zeta \rangle} + \frac{\langle v+\zeta \rangle}{w(v+\zeta)} \big) \|wf(s)\|^{2}_{\infty}  \Big]
		d\zeta. 
		\end{split}
		\end{equation}
		Applying Corollary \ref{lem_int T} and Lemma \ref{lem_Gamma},
		\begin{equation} \label{est : basic f-f2-dv}
		\begin{split}
		&\int_{0}^{t} e^{-\varpi \langle v \rangle^{2} t} 
		\int_{\R_{\zeta}^{3}} \mathbf{k}_{c}(v, \bv, \zeta) \frac{ |\eqref{basic f-f2} | }{|v-\bv|^{2\b}} d\zeta ds   \\
		&\lesssim_{\b} \int_{0}^{t} e^{-\frac{\varpi}{2} \langle v \rangle^{2} (t-s)} 
		ds \|w_{0}f_{0}\|_{\infty} \big[ \sup_{0\leq s \leq T}\mathfrak{H}_{sp}^{2\b}(s) + \sup_{0\leq s \leq T}\mathfrak{H}_{vel}^{2\b}(s) \big] 
		+ \|w_{0}f_{0}\|_{\infty}^{2}  \\
		&\lesssim_{\b} \frac{1}{\varpi} 
		\big[ \sup_{0\leq s \leq T}\mathfrak{H}_{sp}^{2\b}(s) + \sup_{0\leq s \leq T}\mathfrak{H}_{vel}^{2\b}(s) \big] 
		\mathcal{P}_{2}( \|w_{0}f_{0}\|_{\infty}) ,
		\end{split}
		\end{equation}
		where $\mathcal{P}_{2}(\cdot) = 1 + |\cdot| +|\cdot|^{2}$. Note that we should consider $\eqref{split3}_{\Gamma}$ and $\eqref{split4}_{\Gamma}$ separately. However, the bound of $\eqref{split3}_{\Gamma}$ also control $\eqref{split4}_{\Gamma}$ similar as {\it Substep 1-1}. 
		
		{\it Substep 1-3}  In this substep, we consider  \eqref{basic f-f3} and \eqref{basic f-f4}. We use \eqref{replace1-2} and then, from Lemma \ref{lem_nu}, we have same bound as \eqref{est : basic f-f2-dv}, but order of $\|wf(s)\|_{\infty}$ is reduced by $1$. ($\nu(f)$ is linear.)
		\begin{equation} \label{est : basic f-f34-dv}
		\begin{split}
		&\int_{0}^{t} e^{-\varpi \langle v \rangle^{2} t} 
		\int_{\R_{\zeta}^{3}} \mathbf{k}_{c}(v, \bv, \zeta) \frac{ |\eqref{basic f-f3} | + |\eqref{basic f-f4} | }{|v-\bv|^{2\b}} d\zeta ds   \\
		&\lesssim \frac{1}{\varpi} 
		\big[ \sup_{0\leq s \leq T}\mathfrak{H}_{sp}^{2\b}(s) + \sup_{0\leq s \leq T}\mathfrak{H}_{vel}^{2\b}(s) \big] 
		\mathcal{P}_{1}( \|w_{0}f_{0}\|_{\infty}) .
		\end{split}
		\end{equation}
		We omit the detail. \\
		
		Putting \eqref{est : basic f-f1-dv}, \eqref{est : basic f-f2-dv}, and \eqref{est : basic f-f34-dv} altogether, we conclude 
		\Be \label{est : Hvel pre}
		\begin{split}
			 \sup_{0\leq s \leq T}\mathfrak{H}_{vel}^{2\b}(s)  
			&\lesssim_{\b} \sup_{\substack{v\in\R^{3} \\ 0 < |x - \bx|\leq 1}} 
			\langle v \rangle \frac{|f_{0}( x, v ) - f_{0}(\bx, v)|}{|x - \bx|^{2\b}}  
			+ \sup_{ \substack{ x\in \overline{\O} \\ 0 < |v - \bv|\leq 1   } }  \langle v \rangle^{2} \frac{|f_{0}( x, v ) - f_{0}( x, \bv)|}{|v - \bv|^{2\b}}  
			+ \|w_{0} f_{0}\|_{\infty} \\
			&\quad + \frac{1}{\varpi}  \big[ \sup_{0\leq s \leq T}\mathfrak{H}_{sp}^{2\b}(s) + \sup_{0\leq s \leq T}\mathfrak{H}_{vel}^{2\b}(s) \big] \mathcal{P}_{2}( \|w_{0}f_{0}\|_{\infty}) .
			\end{split}
		\Ee

	\noindent {\bf Step 2} We estimate $\mathfrak{H}^{2\b}_{sp}$. From definition \ref{def_H}, we control
	\Be \label{f-f : x bx v} 
	e^{-\varpi \langle v \rangle^{2} t} \int_{\R^{3}_{\zeta}} k_{c}(v, v+\zeta) \frac{ |f(t, x, v + \zeta) - f(t, \bx, v + \zeta)| }{|x-\bx|^{2\b}} d\zeta,\quad |x-\bx| \leq 1.
	\Ee
	 We note that $\tx$ is well-defined only when \eqref{assume_x} holds. For given $x$, $\bx$, and $v$, the set of $\zeta\in\R^{3}$, where \eqref{assume_x} does not hold is of measure zero. So we can assume \eqref{assume_x} without loss of generality throughout {\bf Step 2} and can use \eqref{est:split1} and \eqref{est:split2} in Lemma \ref{f-f split}. \\
	
	\noindent For $f(t, x, v + \zeta) - f(t, \bx, v + \zeta)$ in integrand of \eqref{f-f : x bx v}, we use \eqref{basic f-f} by replacing
	\Be \notag  
	\begin{split}
		( \bar X ,   \bar V ) := (\bar X (s),  \bar V (s))&= (X(s;t, x,  v+ \zeta), V(s;t, \bx, v+ \zeta)),  \quad  |x-\bx| \leq 1. \\
	\end{split}
	\Ee
	Now, we apply  $e^{-\varpi \langle v \rangle^{2} t} \int_{\R^{3}_{\zeta}} k_{c}(v, v+\zeta) \frac{ \cdot }{|x-\bx|^{2\b}} d\zeta$ to each \eqref{basic f-f1} -- \eqref{basic f-f4}. \\
	
	{\it Substep 2-1}	For \eqref{basic f-f1}, we replace 
	\Be \label{replace2-1}
	s=0
	\Ee
	in \eqref{split1} and \eqref{split2}. We use \eqref{est:split1} then similar as \eqref{est : basic f-f1-dv}, 
	\hide
	\begin{equation} \label{est : f-f1-split1}
	\begin{split}
	&e^{-\varpi \langle v \rangle^{2} t} 
	\int_{\R_{\zeta}^{3}} k_{c}(v, v+\zeta) \frac{ |\eqref{split1}_{f} | }{|x-\bx|^{2\b}} d\zeta \\
	&\lesssim e^{-\varpi \langle v \rangle^{2} t} 
	\int_{\R^{3}_{\zeta}} k_{c}(v, v+\zeta)
	{\color{blue} \frac{1}{ \langle v+\zeta \rangle^{s_{1}}} }
	\Big[ 1 + |v+\zeta|t
	\\
	&\quad\quad  + (  |v+\zeta| + |v+\zeta|^{2}t )
	\Big( \int_{0}^{1} \frac{1}{\mathfrak{S}_{sp}(\tau; x, \tx, v+ \zeta)} d\tau 
	+  \int_{\tau_{-}}^{1} \frac{1}{\mathfrak{S}_{sp}(\tau; x, \tx, v+\zeta)} d\tau  
	+  \int_{0}^{\tau_{-}} \frac{1}{\mathfrak{S}_{sp}(\tau; x, \tx, v+\zeta)} d\tau
	\Big) 
	\Big]^{2\b}  
	d\zeta  
	\\
	&\quad \quad
	\times  
	{\color{blue}  \Big[ \sup_{v, |x - \bx|\leq 1} \langle v \rangle^{s_{1}} \frac{|f_{0}(x, v ) - f_{0}(\bx, v)|}{|x - \bx|^{2\b}} + \|w_{0} f_{0}\|_{\infty} \Big]
	} \\
	&\quad +  e^{-\varpi \langle v \rangle^{2} t} 
	\int_{\R^{3}_{\zeta}} k_{c}(v, v+\zeta)
	{\color{blue} \frac{1}{ \langle v+\zeta \rangle^{s_{2}}} }
	\Big[  1 + |v+\zeta|  \\
	&\quad\quad 
	+ ( |v+\zeta| + |v+\zeta|^{2} )			 
	\Big( \int_{0}^{1} \frac{1}{\mathfrak{S}_{sp}(\tau; x, \tx, v+\zeta)} d\tau 
	+  \int_{\tau_{-}}^{1} \frac{1}{\mathfrak{S}_{sp}(\tau; x, \tx, v+\zeta)} d\tau  
	+  \int_{0}^{\tau_{-}} \frac{1}{\mathfrak{S}_{sp}(\tau; x, \tx, v+\zeta) } d\tau
	\Big) 
	\Big]^{2\b}   d\zeta      \\
	&\quad\quad 
	\times
	{\color{blue}  \Big[ \sup_{x, |v - \bv|\leq 1} \langle v \rangle^{s_{2}} \frac{|f_{0}(x, v ) - f_{0}(x, \bv)|}{|v - \bv|^{2\b}} + \|w_{0}f_{0}]\|_{\infty}  \Big]
	}     \\
	\end{split}
	\end{equation}
	Similarly, from \eqref{est:split2},
	\begin{equation} \label{est : f-f1-split2}
	\begin{split}
	&e^{-\varpi \langle v \rangle^{2} t} 
	\int_{\R_{\zeta}^{3}} k_{c}(v, v+ \zeta) \frac{ |\eqref{split2}_{f} | }{|x-\bx|^{2\b}} d\zeta \\
	&\lesssim e^{-\varpi \langle v \rangle^{2} t} 
	\int_{\R^{3}_{\zeta}} k_{c}(v, v+\zeta)  
	{\color{blue} \frac{1}{ \langle v+\zeta \rangle^{s_{1}}} }
	d\zeta  
	{\color{blue}  \Big[ \sup_{v, |x - \bx|\leq 1} \langle v \rangle^{s_{1}} \frac{|f_{0}(x, v ) - f_{0}(\bx, v)|}{|x - \bx|^{2\b}} + \|w_{0} f_{0}\|_{\infty} \Big]
	}   \\
	&\quad +  e^{-\varpi \langle v \rangle^{2} t} 
	\int_{\R^{3}_{\zeta}} k_{c}(v, v+\zeta)  
	{\color{blue} \frac{1}{ \langle v+\zeta \rangle^{s_{2}}} }
	d\zeta    
	{\color{blue}  \Big[ \sup_{x, |v - \bv|\leq 1} \langle v \rangle^{s_{2}} \frac{|f_{0}(x, v ) - f_{0}(x, \bv)|}{|v - \bv|^{2\b}} + \|w_{0}f_{0}]\|_{\infty}  \Big] 
	}      \\
	\end{split}
	\end{equation}	 
	From \eqref{est : f-f1-split1} and \eqref{est : f-f1-split2}, 
	\unhide
	
	\Be \label{est : basic f-f1-dx}
	\begin{split}
		&e^{-\varpi \langle v \rangle^{2} t}  
		\int_{\R_{\zeta}^{3}} k_{c}(v, v+\zeta) \frac{ |\eqref{split3}_{f} | }{|x-\bx|^{2\b}} d\zeta \\
		&\lesssim e^{-\varpi \langle v \rangle^{2} t}  
		\int_{\R_{\zeta}^{3}} k_{c}(v, v+\zeta)  
		\Big[ ( |v+\zeta| + |v+\zeta|^{2}t ) \mathcal{T}_{sp}(x, \tx, v+\zeta)
		\Big]^{2\b}  
		\\
		&\quad\quad 
		\times  
		\Big[ \frac{ 1 }{ \langle v+\zeta \rangle}  
			\sup_{\substack{v\in\R^{3} \\ 0 < |x - \bx|\leq 1}} 
			\langle v \rangle \frac{|f_{0}( x, v ) - f_{0}(\bx, v)|}{|x - \bx|^{2\b}}  + \frac{1}{w_{0}(v+\zeta)}\|w_{0} f_{0}\|_{\infty} \Big] 
		d\zeta 
		\\
		& + e^{-\varpi \langle v \rangle^{2} t}  
		\int_{\R_{\zeta}^{3}} k_{c}(v, v+\zeta) 
		\Big[ (|v+\zeta| +  |v+\zeta|^{2}) \mathcal{T}_{sp}(x, \tx, v+\zeta)
		\Big]^{2\b}   \\
		&\quad\quad 
		\times
		\Big[ 
			\frac{ 1}{ \langle v+\zeta \rangle^{2} }	
			\sup_{ \substack{ x\in \overline{\O} \\ 0 < |v - \bv|\leq 1   } }  \langle v \rangle^{2} \frac{|f_{0}( x, v ) - f_{0}( x, \bv)|}{|v - \bv|^{2\b}} + \frac{1}{w_{0}(v+\zeta)} \|w_{0}f_{0}\|_{\infty}  \Big]
		d\zeta  \\
		&\lesssim_{\b}  
		\sup_{\substack{v\in\R^{3} \\ 0 < |x - \bx|\leq 1}} 
		\langle v \rangle  \frac{|f_{0}( x, v ) - f_{0}(\bx, v)|}{|x - \bx|^{2\b}}  
		+ \sup_{ \substack{ x\in \overline{\O} \\ 0 < |v - \bv|\leq 1   } }  \langle v \rangle^{2} \frac{|f_{0}( x, v ) - f_{0}( x, \bv)|}{|v - \bv|^{2\b}}  
		+ \|w_{0} f_{0}\|_{\infty},
	\end{split}
	\Ee
	where we have used Corollary \ref{lem_int T}.   \\
	
	{\it Substep 2-2} In this substep, we consider  \eqref{basic f-f2}. We put \eqref{replace2-1} and use Lemma \ref{nu G split}, \ref{lem_Gamma}, and \ref{lem_Gamma} as we did in {\it Substep 1-2}. Then similar as \eqref{est : basic f-f2-dv},  we obtain
	\begin{equation} \label{est : basic f-f2-dx}
	\begin{split}
	&\int_{0}^{t} e^{-\varpi \langle v \rangle^{2} t} 
	\int_{\R_{\zeta}^{3}} k_{c}(v, v+\zeta) \frac{ |\eqref{basic f-f2} | }{|x-\bx|^{2\b}} d\zeta ds   \\
	&\lesssim \int_{0}^{t} e^{-\varpi \langle v \rangle^{2} (t-s)}  
	\int_{\R_{\zeta}^{3}} k_{\frac{c}{2}}(v, v+\zeta)
	\Big[  ( |v+\zeta| + |v+\zeta|^{2}t ) \mathcal{T}_{sp}(x, \tx, v+\zeta)
	\Big]^{2\b}  
	\\
	&\quad\quad 
	\times  
	\Big[  
		\Big( 
		\|wf(s)\|_{\infty}
		\sup_{\substack{v\in\R^{3} \\ 0 < |x - \bx|\leq 1}}  
		e^{-\varpi \langle v \rangle^2 s}  
		 \int_{\R^{3}_{u}} k_{c}(v, v+u) \frac{|f(s, x, v+u ) - f(s, \bx, v+u)|}{|x - \bx|^{2\b}} du
		\Big) 
	\\
	&\quad\quad\quad \quad
		+ \frac{\langle v+\zeta \rangle}{w(v+\zeta)}\|w f(s)\|^{2}_{\infty} \Big] 
	d\zeta   \\
	& + \int_{0}^{t} e^{-\varpi \langle v \rangle^{2} (t-s)}  
	\int_{\R_{\zeta}^{3}} k_{\frac{c}{2}}(v, v+\zeta)
	\Big[(|v+\zeta| +  |v+\zeta|^{2}) \mathcal{T}_{sp}(v, \tx, v+\zeta)
	\Big]^{2\b}    \\
	&\quad\quad 
	\times \Big[
		\Big(
		\|wf(s)\|_{\infty}
		\sup_{ \substack{ x\in \overline{\O} \\ 0 < |v - \bv|\leq 1   } } e^{-\varpi \langle v \rangle^2 s}  
		\int_{\R_{\zeta}^{3}} \mathbf{k}_{\frac{c}{2}}(v, \bv, u)
		\frac{|f(s, x, v+u ) - f(s, x, \bv+u)|}{|v - \bv|^{2\b}}
		\Big) 
	\\
	&\quad\quad\quad \quad
		+ \big(  \frac{1}{\langle v+\zeta \rangle}  + \frac{\langle v+\zeta \rangle}{w(v+\zeta)} \big) \|wf(s)\|^{2}_{\infty}  \Big]
	d\zeta  \\
	&\lesssim_{\b} \frac{1}{\varpi} 
	\big[ \sup_{0\leq s \leq T}\mathfrak{H}_{sp}^{2\b}(s) + \sup_{0\leq s \leq T}\mathfrak{H}_{vel}^{2\b}(s) \big] 
	\mathcal{P}_{2}( \|w_{0}f_{0}\|_{\infty}).  \\
	\end{split}
	\end{equation}
	
	{\it Substep 2-3} In this substep, we consider  \eqref{basic f-f3} and \eqref{basic f-f4}. Similar as \eqref{est : basic f-f34-dv} in {\it Substep 1-3}, the bound is same as \eqref{est : basic f-f2-dx} except that the order of $\|wf(s)\|_{\infty}$ is reduced $1$, i.e., \\
	\begin{equation} \label{est : basic f-f34-dx}
	\begin{split}
	&\int_{0}^{t} e^{-\varpi \langle v \rangle^{2} t} 
	\int_{\R_{\zeta}^{3}} k_{c}(v, \bv, \zeta) \frac{ |\eqref{basic f-f3} | + |\eqref{basic f-f4} | }{|x - \bx|^{2\b}} d\zeta ds   
	\lesssim_{\b} \frac{1}{\varpi} 
	\big[ \sup_{0\leq s \leq T}\mathfrak{H}_{sp}^{2\b}(s) + \sup_{0\leq s \leq T}\mathfrak{H}_{vel}^{2\b}(s) \big] 
	\mathcal{P}_{1}( \|w_{0}f_{0}\|_{\infty}) .  \\
	\end{split}
	\end{equation}
	Putting \eqref{est : basic f-f1-dx}, \eqref{est : basic f-f2-dx}, and \eqref{est : basic f-f34-dx} altogether, we conclude 
	\Be \label{est : Hsp pre}
	\begin{split}
		\sup_{0\leq s \leq T}\mathfrak{H}_{sp}^{2\b}(s)  
		&\lesssim_{\b} \sup_{\substack{v\in\R^{3} \\ |x - \bx|\leq 1}} 
		\langle v \rangle  \frac{|f_{0}( x, v ) - f_{0}(\bx, v)|}{|x - \bx|^{2\b}}  
		+ \sup_{ \substack{ x\in \overline{\O} \\ |v - \bv|\leq 1   } }  \langle v \rangle^{2} \frac{|f_{0}( x, v ) - f_{0}( x, \bv)|}{|v - \bv|^{2\b}}  
		+ \|w_{0} f_{0}\|_{\infty} \\
		&\quad + \frac{1}{\varpi}  \big[ \sup_{0\leq s \leq T}\mathfrak{H}_{sp}^{2\b}(s) + \sup_{0\leq s \leq T}\mathfrak{H}_{vel}^{2\b}(s) \big] \mathcal{P}_{2}( \|w_{0}f_{0}\|_{\infty}),
	\end{split}
	\Ee
	for $T \leq T^{*}$, where $T^{*}$ is local existence time in Lemma \ref{lem loc}. \\
	From \eqref{est : Hvel pre} and \eqref{est : Hsp pre}, we finish the proof by choosing sufficiently large $\varpi \gg_{f_{0}, \b} 1$.   
\end{proof}

 \section{Regularity estimate}
 
\subsection{$C^{0,\frac{1}{2}}_{x,v}$ estimates of trajectory}	 
 	The next geometric lemma asserts that the specular characteristics are basically H\"older regular with an exponent of $1/2$.  
 	\begin{lemma}\label{lem:Holder_tb}Suppose the domain is given as in Definition \ref{def:domain} and \eqref{convex_xi}. Let $(x,v), (\bar x, \bar v) \in \O \times \R^3$, and $|(x,v) - (\bar x , \bar v)|\leq 1$, and $\zeta \in \R^3$. 
 		
 		\noindent 1) We have
 		\Be\label{Holder:tb}
 		\begin{split}
 			&\max\{ |v+\zeta|, |\bv+\zeta| \}|\tb(x,v+\zeta) - \tb(\bar x, \bar v+\zeta)|\\
 			&  \lesssim 
 			\sqrt{\| \nabla \xi \|_\infty } 
 			\big\{
 			{|\bar x- x|}^{\frac{1}{2}} +
 			\max\{ \sqrt{\tb(x,v+\zeta)} ,\sqrt{\tb(\bar x, \bar v+\zeta)}\}
 			{|\bar v - v|}^{\frac{1}{2}}
 			\big\}.
 		\end{split}
 		\Ee
 		2) For $s \leq  t$  
 		\Be\label{Holder:X} 
 		|X(s;t,x,v +\zeta)) -X(s;t,\bar x, \bar v+\zeta)|
 		\lesssim
 		\big\{1+  \langle v+\zeta\rangle |t-s|  \big\}
 		\big\{
 		|x- \bar x|^{\frac{1}{2}} +  |t-s|^{\frac{1}{2}}|v- \bar v |^{\frac{1}{2}}
 		\big\}.  \Ee
 		3) For $s \leq  t $, if 
 		\[
 			s \leq \min\{ t^{1}(x,v), t^{1}(\bx, \bv) \}\quad\text{or} \quad s > \max\{ t^{1}(x,v), t^{1}(\bx, \bv) \},
 		\]
 		then
 		\Be\label{Holder:V}
 		\begin{split}
 			|V(s;t,x,v+\zeta) -V(s;t,\bar x,\bar   v+\zeta)|
 			\lesssim
 			|v- \bar v |
 			+   \langle v+\zeta\rangle  
 			\big\{
 			|x- \bar x|^{\frac{1}{2}} +  |t-s|^{\frac{1}{2}}|v- \bar v |^{\frac{1}{2}}
 			\big\}. 
 		\end{split}  \Ee
 		\hide
 		4) For $0 \leq s \leq  t \lesssim 1$ and $s \notin  I_{\textit{int}}(t,x,v, \bar x, \bar v,0)$,
 		\Be\label{Holder:V1}
 		\begin{split}
 			|R_{\xb(x,v)}V(s;t,x,v) -
 			V(s;t,\bar x, \bar v)
 			| 
 			&
 			\lesssim \max\{|v|, |\bar v|\} 
 			\big\{
 			|x- \bar x|^{\frac{1}{2}} +  |t-s|^{\frac{1}{2}}|v- \bar v |^{\frac{1}{2}}
 			\big\}\\
 			&
 			\ \    \ \   \ \   \ \   \  \        \   \ \   \ \   \ \   \ \ \text{for} \ \ 
 			s \in [ t_1(t, \bar x, \bar v ), t_1(t,x,v )],  \\
 			|V(s;t,x,v) -R_{\xb(\bar x,\bar v)}V(s;t,\bar x, \bar v)| 
 			&
 			\lesssim \max\{|v|, |\bar v|\} 
 			\big\{
 			|x- \bar x|^{\frac{1}{2}} +  |t-s|^{\frac{1}{2}}|v- \bar v |^{\frac{1}{2}}
 			\big\}\\
 			&    \ \    \ \   \ \   \ \   \  \        \   \ \   \ \   \ \   \ \  \text{for} \ \ 
 			s \in [t_1(t,x,v ), t_1(t, \bar x, \bar v  )].  
 		\end{split}\Ee\unhide
 	\end{lemma}	
 	\begin{proof}
 		In the course of the proof we set $\zeta=0$ without loss of generality. 
 		
 		\textbf{Step 1. Proof of \eqref{Holder:tb}.} For simplicity's sake we tentatively denote $\tb= \tb(x,v)$ and $\bar{t}_{\mathbf{b}} = \tb(\bar x, \bar v)$. Without loss of generality we may assume $\tb \geq    \bar{t}_{\mathbf{b}}$. For $\tb \leq    \bar{t}_{\mathbf{b}}$ we can follow the same argument with obvious modification. We note that $x-\bar{t}_{\mathbf{b}}  v \in \O$ and therefore $\xi(x-\tb v)- \xi(x-\bar{t}_{\mathbf{b}}  v)  \geq 0$. Using $\xi(\bar x-\bar{t}_{\mathbf{b}}  \bar v)=0=\xi(x-\tb v)$, we derive that  
 		\Be \label{diff:xi}
 		0\leq  \xi(x-\tb v)- \xi(x-\bar{t}_{\mathbf{b}}  v)    =  \xi(\bar x-\bar{t}_{\mathbf{b}}  \bar v)- \xi(x-\bar{t}_{\mathbf{b}}  v) 
 		\leq \| \nabla \xi \|_\infty 
 		\{
 		|\bar x- x| + |\bar{t}_{\mathbf{b}}| |\bar v - v|
 		\} 
 		. 
 		\Ee
 		On the other hand, by an expansion, 
 		\Be
 		\begin{split}
 			\xi(x-\tb v)- \xi(x-\bar{t}_{\mathbf{b}}  v)    &=
 			\int^{\tb}_{\bar{t}_{\mathbf{b}}}
 			-v\cdot \nabla \xi (x-s v)
 			\dd s \\
 			&=
 			\int^{\tb}_{\bar{t}_{\mathbf{b}}}
 			\Big\{
 			-v\cdot \nabla \xi (x-\tb v)
 			+ \int^s_{\tb} v\cdot \nabla^2 \xi (x- \tau v)  v \dd \tau 
 			\Big\}
 			\dd s \\
 			&=( -v\cdot \nabla \xi ( \xb  )) (\tb - \bar{t}_{\mathbf{b}})
 			+  \int^{\tb}_{\bar{t}_{\mathbf{b}}}\int^{\tb}_s - v\cdot \nabla^2 \xi (x- \tau v)   v \dd \tau  
 			\dd s .\label{exp:xi}
 		\end{split}
 		\Ee
 		Together with $-v\cdot \nabla \xi (x-\tb v)\geq0$ and the convexity \eqref{convex_xi}, we conclude that 
 		\Be\label{lower_diff:xi}
 		\xi(x-\tb v)- \xi(x-\bar{t}_{\mathbf{b}}  v)  \geq \frac{\theta_\O}{2} |v|^2 |\tb- \bar{t}_{\mathbf{b}}|^2.
 		\Ee
 		From \eqref{exp:xi} and \eqref{lower_diff:xi}, we derive that, when $\tb \geq \bar{t}_{\mathbf{b}}$, 
 		\Be\label{diff:tb1}
 		|v||\tb -  \bar{t}_\mathbf{b}|\lesssim  \sqrt{\| \nabla \xi \|_\infty}
 		\Big\{
 		\sqrt{|\bar x- x|} +
 		\max\{ \sqrt{\tb} ,\sqrt{\bar{t}_\mathbf{b}}\}
 		\sqrt{|\bar v - v|}
 		\Big\}. 
 		\Ee
 		
 		Now we repeat the procedure \eqref{diff:xi}-\eqref{diff:tb1} with some change: For $\tb \geq \bar{t}_{\mathbf{b}}$, as \eqref{diff:xi}, we have $0 \leq\xi(\bar x-  t_{\mathbf{b}} \bar v) -  \xi (\bar x -\bar  t_{\mathbf{b}}\bar v)
 		=\xi(\bar x-  t_{\mathbf{b}} \bar v) -  \xi ( x -  t_{\mathbf{b}} v)
 		\leq \| \nabla \xi \|_\infty \{|\bar x- x| + |\tb||\bar v -  v|\}.$ Then as \eqref{lower_diff:xi} we derive that $\xi(\bar x-  t_{\mathbf{b}} \bar v) -  \xi (\bar x -\bar  t_{\mathbf{b}}\bar v) \geq \frac{\theta_\O}{2} |\bar v|^2 | \tb - \bar t_{\mathbf{b}} |^2$. Hence we conclude that $|\bar v||\tb - \bar{t}_{\mathbf{b}}|$ has the same bound of \eqref{diff:tb1}:
 		\begin{align}
 		\frac{\theta_\O}{2} |\bar v|^2 |\tb- \bar{t}_{\mathbf{b}}|^2
 		\leq |\xi(\bar x-\tb \bar v)- \xi(\bar x-\bar{t}_{\mathbf{b}}  \bar v) |    \leq  \| \nabla \xi \|_\infty
 		\{
 		|\bar x- x| +\max\{ {\tb} , {\bar{t}_\mathbf{b}}\} |\bar v - v|
 		\} .\label{diff:tb2}
 		\end{align}
 		Thereby we conclude \eqref{Holder:tb} from \eqref{diff:tb1} and \eqref{diff:tb2}.
 		
 		\smallskip
 		
 		\textbf{Step 2: Proof of \eqref{Holder:X}.} 
 		We consider the case of $\tb(x,v)<\infty$ and $\tb(\bar x,  \bar  v)<\infty$. Then from \eqref{form_XV}  \Be
 		\begin{split}
 			&|X(s;t,x,v) - X(s;t, \bar x ,\bar v)| 	\\
 			&\lesssim 
 			|x- \bar{x}| + |t-s| |v- \bar v|
 			+  | n(\xb) \cdot v
 			|| \tb - \bar t_{\mathbf{b}}| 
 			+ |t-\bar t_\mathbf{b} -s | 
 			|\bar v|  |n(\xb)- n(\bar x_\mathbf{b})|
 			\\
 			& \lesssim 
 			\big\{1+ |t-s| \max\{|v|, |\bar v|\}\big\}\big\{
 			|x- \bar x|+ |t-s||v- \bar v |
 			+  \max\{|v|, |\bar v|\}|\tb - \bar t _\mathbf{b}|
 			\big\}.
 			\label{diff:X}
 		\end{split}
 		\Ee
 		Now we are using \eqref{diff:X} and \eqref{Holder:tb} to conclude \eqref{Holder:X}.
 		
 		If $\tb(x,v)=\infty$ and $\tb(\bar x,  \bar  v)=\infty$ then simply we have 
 		\Be\label{diff:X1}
 		|X(s;t,x,v) - X(s;t, \bar x ,\bar v)| \leq |x- \bar x| + |t-s||v - \bar v|,
 		\Ee
 		which is bounded as \eqref{Holder:X}. 
 		
 		For the rest of case we bound $|X(s;t,\bar  x,v) - X(s;t, x ,  v)|$ and $|X(s;t,  x,\bar v) - X(s;t, x ,  v)|$ separately. We start with $|X(s;t,\bar  x,v) - X(s;t, x ,  v)|$. First we consider the case of $\tb(\bar x, v)<\infty$ and $\tb( x,  v)=\infty$. 
 		Recall $\X(\tau)$ in \eqref{xv para}. From \eqref{tau_pm} there exists $\tau_+=\tau_+ (x,\bar x, v)$ such that $- \nabla \xi (\xb(\X(\tau_+), v)) \cdot v=0$, $\tb(\X(\tau), v)<\infty$ for $\tau \in [  0, \tau_+ ]$, and $\tb(\X(\tau), v)=\infty$ for $\tau \in ( \tau_+,1 ]$. Equipped with $|X(s;t,\bar  x,v) - X(s;t, x ,  v)|\leq |X(s;t,\bar  x,v) - X(s;t, x(\tau_+) ,  v)| + |X(s;t,  x(\tau_+),v) - X(s;t, x ,  v)|$, we have $|X(s;t,\bar  x,v) - X(s;t, \X(\tau_+) ,  v)|\lesssim \eqref{diff:X}$ and $|X(s;t,  \X(\tau_+),v) - X(s;t, x ,  v)|\lesssim \eqref{diff:X1}$, which give us the bound of $|X(s;t,\bar  x,v) - X(s;t, x ,  v)|$ as in \eqref{Holder:X}. The case of $\tb(x, v)<\infty$ and $\tb(\bar  x,  v)=\infty$ can be treated similarly. 
 		
 		Now we bound $|X(s;t,  x,\bar v) - X(s;t, x ,  v)|$. Consider first the case of $\tb( x, \bar v)<\infty$ and $\tb( x,  v)=\infty$. Recall $\V(\tau)$ in \eqref{xv para}. From \eqref{tau_pm} there exists $\tau_+=\tau_+ (   x, v, \bar v, 0)$ such that $- \nabla \xi (\xb(x, \V(\tau_+))) \cdot \V(\tau_+)=0$, $\tb(x, \V(\tau))<\infty$ for $\tau \in [  0, \tau_+ ]$, and $\tb(x, \V(\tau))=\infty$ for $\tau \in ( \tau_+,1 ]$. Then following the previous argument and using \eqref{diff:X} and \eqref{diff:X1}, we easily get the result. Other case $\tb( x, v)<\infty$ and $\tb( x, \bar v)=\infty$ also can be treated similarly.

 		\smallskip
 		
 		\textbf{Step 3: Proof of \eqref{Holder:V}.} We only need to consider the case of $s < \min \{ t_1 (t,x,v), t_1 (t, \bar x, \bar v)\}$. Otherwise, in the case of $s > \max \{ t_1 (t,x,v), t_1 (t, \bar x, \bar v)\}$, we have $|V(s;t,x,v) -V(s;t,\bar x,\bar   v)|= |v- \bar v|$. From \eqref{form_XV} and \eqref{Holder:tb}
 		\Be
 		\begin{split}\label{diff:V}
 			&|V(s;t,x,v)- V(s;t, \bar x, \bar v)| 
 			\lesssim |v- \bar v| + 
 			\max \{|v|, |\bar v|\} |\xb - \bar x_\mathbf{b}|\\
 			&
 			\lesssim \big\{1+ |t-s| \max\{|v|, |\bar v|\}\big\} |v- \bar v| + 
 			\max \{|v|, |\bar v|\}  |x- \bar x|
 			+ \max \{|v|, |\bar v|\}^2 |\tb- \bar{t}_{\mathbf{b}}|.
 		\end{split}
 		\Ee
 		Now using \eqref{diff:V} and \eqref{Holder:tb} we conclude \eqref{Holder:V}.   \\		
 		
 		\hide
 		\smallskip
 		
 		\textbf{Step 4: Proof of .}

 		\smallskip
 		
 		\textbf{Step 5: Proof of .}
 		
 		\unhide\end{proof}

 \subsection{H\"older regularity : Proof of the Main Theorem}
 
 \hide
 \begin{theorem}
 	Suppose the domain is given as in Definition \ref{def:domain} and \eqref{convex_xi}. Assume $f_0 \in C^1_{x,v}(\bar \O \times \R^3)$ and $\| e^{\vartheta_{0}|v|^{2}} f_0 \|_\infty< \infty$ for $0< \vartheta_{0} < 1/4$. Then there exists $0< T \ll 1$ such that in where we can construct a unique solution $f(t,x,v)$ which solves \eqref{f_eqtn} and \eqref{specular} for $0 \leq t \leq T$, and satisfies $\sup_{0 \leq t \leq T}\| e^{\vartheta|v|^2} f (t) \|_\infty\lesssim \| e^{\theta_{0}|v|^2} f_0 \|_\infty$ for some $0\leq \vartheta <\vartheta_{0}$. Moreover $f(t,x,v)$ is H\"older continuous in the following sense:
 	{\color{blue}
 		\begin{equation} \label{est:Holder}
 		\begin{split}
 		&\sup_{0\leq t \leq T} 
 		\sup_{ \substack{ (x,v)\in \overline{\O}\times \R^{3} \\ |(x,v)-(\bx, \bv)|\leq 1   } }
 		\Big| 
 		\langle v \rangle^{-2\b} e^{-\varpi\langle v \rangle^{2}t}\frac{ | f(t,x,v) - f(t, \bx, \bv) | }{ |(x,v) - (\bx, \bv)|^{\b} }
 		\Big|    \\
 		&\lesssim_{\b} 
 		\|w_{0}f_{0}\|_{\infty}
 		\Big[
 		\sup_{v, |x - \bx|\leq 1} 
 		\langle v \rangle  \frac{|f_{0}( x, v ) - f_{0}(\bx, v)|}{|x - \bx|^{2\b}}  
 		+ \sup_{x, |v - \bv|\leq 1}  \langle v \rangle^{2} \frac{|f_{0}( x, v ) - f_{0}( x, \bv)|}{|v - \bv|^{2\b}}  
 		\Big]
 		+ \mathcal{P}_{2}(\|w_{0}f_{0}\|_{\infty}).
 		\end{split}
 		\end{equation}
 		where $\mathcal{P}(s) = |s| + |s|^2$. 
 	} 
 \end{theorem}
\unhide

We provide the proof of Theorem \ref{theo:Holder}.
 \begin{proof} [Proof of Theorem \ref{theo:Holder}]
 	\hide
 	{\color{blue}
 		{\bf Instruction} \\
 		1. We consider duhamel expansion \eqref{basic f-f1}--\eqref{basic f-f4}. For each \eqref{basic f-f1}--\eqref{basic f-f4}. Note that we just keep $(x,v)$ and $(\bx, \bv)$ and apply term by term splittion \eqref{est:split2 raw}, \eqref{est:split4 raw}, \eqref{est:split1 raw}, and \eqref{est:split3 raw}. \\
 		2. In $H$ estimate we splitted $(x, \bx, v)$ and $(\bx, v, \bv)$ so for each $H_{x}$ and $H_{v}$ we only considered \eqref{est:split2 raw}, \eqref{est:split1 raw} OR \eqref{est:split1 raw}, \eqref{est:split3 raw} only. But now, we consider all \eqref{est:split2 raw}--\eqref{est:split3 raw} in one time. \\
 		3. Check condition ($\mathbf{1}$ etc) and apply Lemma \ref{lem:Holder_tb} to control $\frac{1}{2}$ order fraction estimate$X-X$ and $V-V$ and $|v|(t^{1} - s)$. Carefule (and compare as previous version) order of natural growth!!  \\
 		4. Generate $H^{2\b}_{x,v}$ form and impose proper weight. }   \\
 		\unhide
 	First, note that \eqref{est:Holder} is obvious if $|(x,v) - (\bx, \bv)| \geq 1$. If $|(x,v) - (\bx, \bv)| \leq 1$, we consider the following steps. \\ 	
	\noindent{\bf Step 1} Let us assume \eqref{assume_x} for $v+\zeta$ and $\eqref{assume_v}$ for $\bx$, respectively.   \\
	In this step, we consider \eqref{basic f-f1}. To consider \eqref{basic f-f1}, we put $s=0$ of \eqref{split1}--\eqref{split4} and apply Lemma \ref{lem:Holder_tb} to obtain 
 	\begin{equation} \label{est : f-f1-half}
 	\begin{split}
 	&\frac{ \eqref{basic f-f1} }{|(x,v)-(\bx,\bv)|^{\b}}  \\
 	&\lesssim 
 	\Big[ 1 +  \langle v+\zeta \rangle t \Big]^{2\b}  
 	\Big[ \frac{ 1 }{ \langle v+\zeta \rangle}  
 		\sup_{\substack{v\in\R^{3} \\ 0 < |x - \bx|\leq 1}} 
 		\langle v \rangle \frac{|f_{0}( x, v ) - f_{0}(\bx, v)|}{|x - \bx|^{2\b}}  + \frac{1}{w_{0}(v+\zeta)}\|w_{0} f_{0}\|_{\infty} \Big] 
 	\\
 	&\quad +
 	\Big[1 + \langle v+\zeta \rangle \Big]^{2\b} 
 	\Big[ 
 		\frac{ 1}{ \langle v+\zeta \rangle^{2} }	
 		\sup_{ \substack{ x\in \overline{\O} \\ 0 < |v - \bv|\leq 1   } }  \langle v \rangle^{2} \frac{|f_{0}( x, v ) - f_{0}( x, \bv)|}{|v - \bv|^{2\b}} + \frac{1}{w_{0}(v+\zeta)} \|w_{0}f_{0}\|_{\infty}  \Big]
 	\\
 	&\lesssim  
 	\sup_{\substack{v\in\R^{3} \\ 0 < |x - \bx|\leq 1}} 
 	\langle v \rangle \frac{|f_{0}( x, v ) - f_{0}(\bx, v)|}{|x - \bx|^{2\b}}  
 	+ \sup_{ \substack{ x\in \overline{\O} \\ 0 < |v - \bv|\leq 1   } }  \langle v \rangle^{2} \frac{|f_{0}( x, v ) - f_{0}( x, \bv)|}{|v - \bv|^{2\b}}  
 	+ \|w_{0} f_{0}\|_{\infty}.
 	\end{split}
 	\end{equation}
 	Similarly,
 	\begin{equation} \label{est : f-f2-half}
 	\begin{split}
 	&\frac{ \eqref{basic f-f2} }{|(x,v)-(\bx,\bv)|^{\b}}  \\
 	\hide
 	&\lesssim 
 	\int_{0}^{t}
 	\Big[ 1 +  \langle v+\zeta \rangle (t-s) \Big]^{2\b}  
 	{\color{blue}  \Big[ \frac{ 1 }{ \langle v+\zeta \rangle}  
 		\sup_{v, |x - \bx|\leq 1} 
 		\langle v \rangle \frac{|f_{0}( x, v ) - f_{0}(\bx, v)|}{|x - \bx|^{2\b}}  + \frac{1}{w_{0}(v+\zeta)}\|w_{0} f_{0}\|_{\infty} \Big] 
 	}     
 	\\
 	& + 
 	\Big[1 + \langle v+\zeta \rangle \Big]^{2\b} 
 	{\color{blue}  \Big[ 
 		\frac{ 1}{ \langle v+\zeta \rangle^{2} }	
 		\sup_{x, |v - \bv|\leq 1}  \langle v \rangle^{2} \frac{|f_{0}( x, v ) - f_{0}( x, \bv)|}{|v - \bv|^{2\b}} + \frac{1}{w_{0}(v+\zeta)} \|w_{0}f_{0}\|_{\infty}  \Big]
 	}  
 	\\
 	\unhide
 	&\lesssim \int_{0}^{t} \Big[ 1 +  \langle v+\zeta \rangle (t-s) \Big]^{2\b} e^{\varpi \langle v+\zeta \rangle^{2} s}
 	\\
 	&\quad\quad 
 	\times  
 	\Big[ 
 		\Big( 
 		\|wf(s)\|_{\infty}
 		\sup_{\substack{v\in\R^{3} \\ 0 < |x - \bx|\leq 1}}  
 		e^{-\varpi \langle v \rangle^2 s}  
 		\int_{\R^{3}_{u}} k_{c}(v, v+u) \frac{|f(s, x, v+u ) - f(s, \bx, v+u)|}{|x - \bx|^{2\b}} du
 		\Big) 
 	\\
 	&\quad\quad\quad \quad
 		+ \frac{\langle v+\zeta \rangle}{w(v+\zeta)}\|w f(s)\|^{2}_{\infty} \Big] 
 	d\zeta   \\
 	& + \int_{0}^{t} 
 	\Big[1 + \langle v+\zeta \rangle \Big]^{2\b} e^{\varpi\langle v+\zeta \rangle^{2}s}  \\
 	&\quad\quad 
 	\times
 		\Big[ 
 		\Big(
 		\|wf(s)\|_{\infty}
 		\sup_{ \substack{ x\in \overline{\O} \\ 0 < |v - \bv|\leq 1   } } e^{-\varpi \langle v \rangle^2 s}  
 		\int_{\R_{\zeta}^{3}} \mathbf{k}_{c}(v, \bv, u)
 		\frac{|f(s, x, v+u ) - f(s, x, \bv+u)|}{|v - \bv|^{2\b}}
 		\Big) 
 	\\
 	&\quad\quad\quad \quad
 		+ \big( \frac{1}{\langle v+\zeta \rangle} + \frac{\langle v+\zeta \rangle}{w(v+\zeta)} \big) \|wf(s)\|^{2}_{\infty}  \Big]
 	d\zeta  \\
 	&\lesssim  
 	\langle v+\zeta \rangle^{2\b} e^{\varpi\langle v+\zeta \rangle^{2}t} 
 	\Big\{ \|w_{0}f_{0}\|_{\infty}\big[ \sup_{0\leq s \leq T}\mathfrak{H}_{sp}^{2\b}(s) + \sup_{0\leq s \leq T}\mathfrak{H}_{vel}^{2\b}(s) \big] + \|w_{0}f_{0}\|_{\infty}^{2} \Big\},
 	\end{split}
 	\end{equation}
 	and  
 	\begin{equation} \label{est : f-f34-half}
 	\begin{split}
 	&\frac{ \eqref{basic f-f3} + \eqref{basic f-f4} }{|(x,v)-(\bx,\bv)|^{\b}}  \\
 	&\lesssim  
 	\langle v+\zeta \rangle^{2\b} e^{\varpi\langle v+\zeta \rangle^{2}t} 
 	\Big\{ \|w_{0}f_{0}\|_{\infty}\big[ \sup_{0\leq s \leq T}\mathfrak{H}_{sp}^{2\b}(s) + \sup_{0\leq s \leq T}\mathfrak{H}_{vel}^{2\b}(s) \big]  + \|w_{0}f_{0}\|_{\infty} \Big\}.
 	\end{split}
 	\end{equation}
 	From \eqref{est : f-f1-half}, \eqref{est : f-f2-half}, \eqref{est : f-f34-half}, and Lemma \ref{est : H}, for $|(x,v) - (\bx, \bv)| \leq 1$,
 	\begin{equation*}
 	\begin{split}
 	&\langle v+\zeta \rangle^{-2\b} e^{-\varpi\langle v+\zeta \rangle^{2}t}\frac{ | f(t,x,v+\zeta) - f(t, \bx, \bv+\zeta) | }{ |(x,v) - (\bx, \bv)|^{\b} }  \\
 	&\lesssim  \|w_{0}f_{0}\|_{\infty} \big[ \sup_{0\leq s \leq T}\mathfrak{H}_{sp}^{2\b}(s) + \sup_{0\leq s \leq T}\mathfrak{H}_{vel}^{2\b}(s) \big] +  \mathcal{P}_{2}(\|w_{0}f_{0}\|_{\infty})  \\
 	&\lesssim 
 	\|w_{0}f_{0}\|_{\infty}
 	\Big[
 	\sup_{\substack{v\in\R^{3} \\ 0 < |x - \bx|\leq 1}} 
 	\langle v \rangle  \frac{|f_{0}( x, v ) - f_{0}(\bx, v)|}{|x - \bx|^{2\b}}  
 	+ \sup_{ \substack{ x\in \overline{\O} \\ 0 < |v - \bv|\leq 1   } } \langle v \rangle^{2} \frac{|f_{0}( x, v ) - f_{0}( x, \bv)|}{|v - \bv|^{2\b}}  
 	\Big]
 	+ \mathcal{P}_{2}(\|w_{0}f_{0}\|_{\infty}).  \\
 	\end{split}
 	\end{equation*}
 	\\
 	\noindent{\bf Step 2} (Trivial case) Assume \eqref{assume_x} or \eqref{assume_v} do not hold with $v+\zeta$ and  $\bx$, respectively. In this case, we cannot split \eqref{diff f} into \eqref{split1}--\eqref{split4}. Instead, we split
 	\begin{align}
 	& f(s, X(s;t,x,v+\zeta), V(s;t,x,v+\zeta)) - f(s, X(s;t, \bar{x}, \bar{v}+\zeta), V(s;t, \bar{x}, \bar{v}+\zeta))  \notag  \\
 	&\leq  \big[ f(s, X(s;t,x,v+\zeta), V(s;t,x,v+\zeta)) - f(s, X(s;t, \bar{x}, v+\zeta ), V(s;t, \bar{x}, v+\zeta ))  \big]  \label{triv:split1}  \\
 	&\quad + \big[ f(s, X(s;t, \bar{x}, v+\zeta ), V(s;t, \bar{x}, v+\zeta )) - f(s, X(s;t, \bar{x}, \bv+\zeta ), V(s;t, \bar{x}, \bv+\zeta )) \big]. \label{triv:split2}
 	\end{align}
 	When \eqref{assume_x} does not hold, we can estimate \eqref{triv:split1} similar as \eqref{est:split2} since $S_{(\tx, \bx, v+\zeta)}$ is not well-defined by \eqref{def_tildex}. Similarly, when \eqref{assume_v} does not hold, we can estimate \eqref{triv:split2} similar as \eqref{est:split4} since $S_{(\bx, \tv, \bv, \zeta)}$ is not well-defined by \eqref{def_tildev}. Therefore, we obtain the same bound \eqref{est:Holder} even if \eqref{assume_x} for $v+\zeta$ or $\eqref{assume_v}$ for $\bx$ do not hold. 
 	\hide
 	\begin{equation*}
 	\begin{split}
 	&\frac{ | f(t,x,v+\zeta) - f(t, \bx, \bv+\zeta) | }{ |(x,v) - (\bx, \bv)|^{\b} }  \\
 	&\lesssim \langle v+\zeta \rangle^{2\b} e^{\varpi\langle v+\zeta \rangle^{2}t}  \Big\{ \|w_{0}f_{0}\|_{\infty} \big[ \sup_{0\leq s \leq T}\mathfrak{H}_{sp}^{2\b}(s) + \sup_{0\leq s \leq T}\mathfrak{H}_{vel}^{2\b}(s) \big] +  \mathcal{P}_{2}(\|w_{0}f_{0}\|_{\infty}) \Big\}.
 	\end{split}
 	\end{equation*} 
 	and hence,
 	\unhide
\end{proof}
	

	\hide
		\noindent \textbf{Lemma \ref{lem_tau ratio}}
		\textit{ (i) Assume \eqref{assume_x} and \eqref{assume_x2}, and recall $\tau_0(x, \bx, v)$ from \eqref{tau_0}. Then,
		\Be  
		\frac{\tau_{0}(x, \bx, v) - \tau_{-}(x, \bx, v)}{\tau_{+}(x, \bx, v) - \tau_{-}(x, \bx, v)} \gtrsim_{\O} 1.  \\
		\Ee
		(ii) Assume \eqref{assume_v} and \eqref{assume_v2}, and recall $\tau_0(x, v, \bv, \zeta)$ from \eqref{tau_0_v}. Then,
		\Be  
		\frac{\tau_{0}(x, v, \bv, \zeta) - \tau_{-}(x, v, \bv, \zeta)}{\tau_{+}(x, v, \bv, \zeta) - \tau_{-}(x, v, \bv, \zeta)} \gtrsim_{\O} 1.	
		\Ee
	} 
	\unhide

	\section{Acknowledgements}
	CK is supported in part by National Science Foundation under Grant No.1900923, Grant No.2047681, the Wisconsin Alumni Research Foundation, and the Brain Pool program (NRF-2021H1D3A2A01039047) of the Ministry of Science and ICT in Korea. DL is supported by Samsung Science and Technology Foundation under Project Number   SSTF-BA1902-02	 and the National Research Foundation of Korea(NRF) grant funded by the Korea government(MSIT)(No. NRF-2019R1C1C1010915). 
	 
		\bibliographystyle{plain}

\end{document}